\documentclass{amsart}
\usepackage{import}
\usepackage[latin9]{inputenc}
\usepackage{microtype}
\usepackage{fancyhdr}
\usepackage{cite} 
\usepackage{amsfonts}
\usepackage{dsfont} 
\usepackage{amsmath}
\usepackage{amsthm}
\usepackage{amssymb}
\usepackage[all]{xy}
\usepackage{mathrsfs}
\usepackage{bm} 
\usepackage{stmaryrd}
\usepackage{color}   
\usepackage[hidelinks]{hyperref}
\hypersetup{
    colorlinks=false, 
    linktoc=all,     
    linkcolor=black,  
}
\usepackage{graphicx}
\usepackage{wrapfig}
\usepackage{float}
\usepackage{caption}
\usepackage{subcaption}
\usepackage{enumitem}
\usepackage{empheq} 
\usepackage{mathtools} 
\usepackage{adjustbox}
\usepackage{tikz,tikz-cd}
\usetikzlibrary{positioning,arrows}
\setlist{nosep} 
\usepackage{dutchcal} 

\usepackage{makecell} 

\usepackage{imakeidx} 
\makeindex 

\usepackage[splitrule]{footmisc} 

\newcommand{\vol}{\text{vol}}

\newcommand\pmttimes[1]{\mathord{\mathop{#1}\limits^{\scriptscriptstyle(\sim)}}}

\usepackage[top=2.5cm,bottom=2cm, margin=2.5cm]{geometry}
\usepackage[font=small, width=.75\textwidth]{caption}
\usepackage{mwe}

\theoremstyle{definition}
\newtheorem{theorem}{Theorem}[section]

\theoremstyle{definition}
\newtheorem{lemma}[theorem]{Lemma}

\theoremstyle{definition}
\newtheorem{definition}[theorem]{Definition}

\theoremstyle{definition}
\newtheorem{proposition}[theorem]{Proposition}

\theoremstyle{definition}
\newtheorem{example}[theorem]{Example}

\theoremstyle{definition}
\newtheorem{corollary}[theorem]{Corollary}

\theoremstyle{definition}

\theoremstyle{definition}

\newcommand{\innerthmname}{}

\theoremstyle{definition}

\usepackage{xcolor}

\newcommand*{\colorboxed}{}
\def\colorboxed#1#{%
  \colorboxedAux{#1}%
}
\newcommand*{\colorboxedAux}[3]{%
  \begingroup
    \colorlet{cb@saved}{.}%
    \color#1{#2}%
    \boxed{%
      \color{cb@saved}%
      #3%
    }%
  \endgroup
}

\title{\vspace{-3cm} An algorithm for finding minimal volume\\ hyperbolic links and the Dehn parental test}

\author{Misha Schmalian}
\thanks{Mathematical Institute, University of Oxford, OX2 6GG, UK, https://sites.google.com/view/mishaschmalian}

\begin{document}
\vspace*{3cm}
\maketitle
\date{}
\vspace{-1cm}
\begin{abstract}
We describe an algorithm that, given a 3-manifold $M$, outputs a finite set containing all minimal volume $k$-component hyperbolic link complements in $M$. A key step, that might be of independent interest, is an algorithm that, given two 3-manifolds $N$ and $M$, decides whether they are related by Dehn filling. In fact, we show that the set of boundary slopes giving a Dehn filling of $N$ to $M$ is determined by a special class of well-studied quadratic Diophantine (in)equalities, for which solvability is known to be decidable.
\end{abstract}

\section{Introduction}
A {\it hyperbolic} link $L$ in a 3-manifold $M$ is an embedding $\bigsqcup_{i=1}^k S^1\to M$ with $M-L$ admitting a complete finite-volume Riemannian metric with constant curvature $-1$. The {\it volume} of a hyperbolic link is $\text{vol}(L):=\text{vol}_g(M-L)$, for $g$ the hyperbolic metric. Substantial prior work has focused on finding the least volume hyperbolic $k$-component links in $S^3$. This has been resolved for $k=1, 2$, and $4$ \cite{Min_volume_cusped, Agol_2cusp_vol, 4_cusp_vol}. There have also been partial results for $k=3$ \cite{3_cusp_attempt}. See also \cite[Problem 3.60]{Kirby_List}. More generally, we can try to find minimal volume links in an arbitrary 3-manifold $M$. The minimal volume of a hyperbolic link, with any number of components, in $M$ is denoted $\text{volt}(M)$ and has been extensively studied in \cite{voltM}. We describe an algorithm to find minimal volume hyperbolic links and to compute $\text{volt}(M)$. 

\begin{theorem}\label{thm:compute_volt}
There is an algorithm that, given $k\in \mathbb Z_{\geq 0}$, $\varepsilon>0$, and a compact orientable 3-manifold $M$ with, possible empty, toroidal boundary, outputs, with an error of at most $\varepsilon$, $\text{volt}_k(M):=\min\{\text{vol}(L)\mid L\text{ is a hyperbolic }k-\text{component link in }M\}$ and $\text{volt}(M):=\min_{n\geq 0}\text{volt}_n(M)$ . 
\end{theorem}

In fact we show that:

\begin{theorem}\label{thm:volumetheorem}
There is an algorithm that, given $k\in \mathbb N$ and a compact orientable 3-manifold $M$ with, possibly empty, toroidal boundary, outputs a finite list that includes all minimal volume $k$-component hyperbolic links in $M$ up to homeomorphisms of the link complement in $M$. 
\end{theorem}

\begin{theorem}\label{thm:biglist_volumetheorem}
There is an algorithm that, given a compact orientable 3-manifold $M$ with, possibly empty, toroidal boundary, outputs a finite list that includes all minimal volume hyperbolic links in $M$ up to homeomorphisms of the link complement in $M$.
\end{theorem}

For completeness, and to aid the exposition, we will also prove the following much easier result: 

\begin{proposition}\label{prop:min_volume__n_cusp}
There is an algorithm that, given $n\in \mathbb N$, outputs a finite list that includes all minimal volume orientable hyperbolic 3-manifolds with exactly $n$ boundary tori. 
\end{proposition}

We say an oriented 3-manifold $M$ is a {\it Dehn filling} of an oriented 3-manifold $N$ with toroidal boundary if $M$ is orientation-preservingly homeomorphic to the result of gluing solid tori into some number of boundary components of $N$. ``The Dehn filling of $N$'' will refer both to the process of gluing in solid tori and the resulting 3-manifold.\\

 A key step in proving Theorem \ref{thm:volumetheorem} will be the following algorithm, that might be of independent interest. 

\begin{theorem}
\label{thm:maintheorem}
There is an algorithm that, given two compact oriented 3-manifolds $M,N$ with, possibly empty, toroidal boundary, decides whether $M$ is a Dehn filling of $N$. 
\end{theorem}

In fact, we will prove that the question of which Dehn fillings of $N$ are homeomorphic to $M$ is answered by a set of particularly simple Diophantine equations.

\begin{definition} A {\it Diophantine system} in variables $\alpha_1, \ldots, \alpha_k$ is a finite collection of sets $\mathcal{S}_i$ each of which consists of finitely many (in)equalities of the form $p(\alpha_1, \ldots, \alpha_k)\diamond 0$, with $p$ a polynomial with integer coefficients and each $\diamond$ one of $=, <, \leq$. The {\it solution set} to such a Diophantine system is the subset of $\mathbb Z^k$ of values for $\alpha_i$ satisfying all (in)equalities in at least one of the $\mathcal{S}_i$. For a Diophantine system $\mathcal{D}$ in variables $\alpha_1, \ldots, \alpha_k$ and a subset $J\subset \{1, \ldots, k\}$, we say values for $\alpha_j, j\in J$ are {\it part of a solution} to $\mathcal{D}$ if they can be extended by values for $\alpha_j, j\notin J$ to a solution of $\mathcal{D}$.
\end{definition}

There is no algorithm that determines whether a general Diophantine system has empty solution set \cite{Hilbert10}. We will therefore consider only particularly simple Diophantine systems.

\begin{definition}
A {\it linear system} is a Diophantine system in which all polynomials appearing in the (in)equalities have degree at most $1$. A {\it mono-quadratic system} is a Diophantine system $\mathcal{Q}$ such that for each set $\mathcal{S}_i$ of $\mathcal{Q}$ there is a partition of the variables satisfying:
\begin{itemize}
\item[-] no (in)equality of $\mathcal{S}_i$ uses variables from multiple parts of the partition; 
\item[-] for each part there is at most one equality of degree $2$ in $\mathcal{S}_i$ using the variables from that part and all other (in)equalities have degree at most $1$. 
\end{itemize}
In particular, in a mono-quadratic system there are no equalities of degree 3 or higher and all quadratic (in)equalities must be equalities. 
\end{definition}

\begin{definition}
Let $M$ be a compact oriented 3-manifold with toroidal boundary.
A {\it $\partial$-framing} of $M$ is an ordering of its boundary tori and an $H_1$-basis for each boundary torus of $M$. Here an $H_1$-basis for a torus $T$ is an oriented, ordered basis of $H_1(T_i)\cong \mathbb Z^2$. A {\it $\partial$-framed manifold} is a 3-manifold equipped with a $\partial$-framing and $\partial$-framed manifolds admitting orientation-preserving homeomorphisms respecting their $\partial$-framings are {\it $\partial$-framed homeomorphic}. Given a $\partial$-framed 3-manifold $N$, a Dehn filling of $N$ in its first $k$ boundary tori is determined by the homology classes $(p_i, q_i)\in \mathbb Z^2=H_1(T_i)$ of the images of meridians $\partial D^2\times \{1\}\subset D^2\times S^1$ in the boundary tori $T_i\subset \partial N$ under gluing. We call such a collection of slopes in distinct boundary tori of $N$ a {\it multi-slope} on $N$ and we let $N(p_1/q_1, \ldots, p_k/q_k)$ denote this Dehn filling of $N$. When specifying a Dehn filling of $N$ to $N(p_1/q_1, \ldots, p_k/q_k)$ we will require that $p_i, q_i$ are coprime.
\end{definition}

\begin{theorem}\label{prop:mainprop}
Let $M, N$ be $\partial$-framed manifolds. There is a mono-quadratic system $\mathcal{Q}$ such that $\gcd(p_1, q_1)=\ldots=\gcd(p_{k}, q_{k})=1$ and the Dehn filling $N(p_1/q_1, \ldots, p_{k}/q_{k})$ is orientation-preservingly homeomorphic to $M$ if and only if $p_1, q_1, \ldots, p_{k}, q_{k}$ are part of a solution to $\mathcal{Q}$. Moreover, there is an algorithm that, given $M,N$, outputs $\mathcal{Q}$.
\end{theorem}

\begin{definition}
Consider a hyperbolic 3-manifold $N$ and a Dehn filling of $N$ along a multi-slope $S$. We say this Dehn filling is {\it strongly non-exceptional} if every Dehn filling of $N$ along a subset of $S$ is a hyperbolic manifold. In particular, since $S$ is a subset of itself, a strongly non-exceptional Dehn filling of $N$ is hyperbolic.
\end{definition}

\begin{corollary} \label{thm:hyperbolic_Dehn_Ancestry}
Consider $\partial$-framed manifolds $N, M$ with $N$ hyperbolic and $k\in \mathbb N$. There is a mono-quadratic system $\mathcal{Q}_{\text{hyp}}$ such that $\gcd(p_1, q_1)=\ldots=\gcd(p_{k}, q_{k})=1$, the Dehn filling of $N$ to $N(p_1/q_1, \ldots, p_k/q_k)$ is strongly non-exceptional, and $N(p_1/q_1, \ldots, p_k/q_k)$ admits a Dehn filling to $M$ if and only if $p_1, q_1, \ldots, p_{k}, q_{k}$ are part of a solution to $\mathcal{Q}_{\text{hyp}}$. Moreover, there is an algorithm that, given $M,N$, and $k$, outputs $\mathcal{Q}_{\text{hyp}}$.
\end{corollary}

Example \ref{ex:monoquad_needed} will give a heuristic for why solving mono-quadratic systems is a key step in deciding whether a manifold $N$ is a link-complement in a manifold $M$. In fact, mono-quadratic systems are used even in the case of $M\cong S^3$ as is detailed in Example \ref{ex:monoquad_needed}. However this is not an issue since, by a direct application of \cite[Theorem 0.1]{Monoquadratic_Algo}, we obtain the following algorithm to study the solution set of mono-quadratic systems. 

\begin{theorem} \label{lemma:linearsystemsolution} There is an algorithm that, given a mono-quadratic system, decides whether its solution set is empty, finite, or infinite. If the solution set is finite, the algorithm produces the solution set; if infinite, given a number $n$, the algorithm produces $n$ members of the solution set.
\end{theorem}

We record that Theorem \ref{thm:maintheorem} is a direct corollary of Theorem \ref{prop:mainprop} and Theorem \ref{lemma:linearsystemsolution}. We will explain how Theorem \ref{lemma:linearsystemsolution} follows from \cite[Theorem 0.1]{Monoquadratic_Algo} in Section \ref{section:mono_quadratic}. \\

%

\tableofcontents

\subsection{Conventions}

Throughout the article we adopt the following conventions: 
\begin{itemize}
\item[-] All 3-manifolds will be compact and have, possibly empty, toroidal boundary.
\item[-] All 3-manifolds will be oriented and homeomorphisms and Dehn-fillings between 3-manifolds will always be orientation-preserving. 
\item[-] Boundaries of oriented 3-manifolds will always be oriented to point out of the 3-manifold.
\item[-] Hyperbolic manifolds with always be complete and of finite-volume. 
\item[-] Manifolds will not necessarily be connected, but will always have finitely many connected components. 
\item[-] Manifold will be input/output as triangulations. Submanifolds of a manifold will be input/output as subcomplexes of an iterated barycentric subdivision of the triangulation of the ambient manifold.
\item[-] {\it Integral-affine function} will mean a degree 1 polynomial with integer coefficients. 
\item[-] $S^2, K^2, D^2, \text{M\"ob}, I$ will respectively denote the 2-sphere, Klein bottle, 2-disk, M\"obius band, and interval $[0,1]$. 
\item[-] It will often be implicit that certain symbols represent values taken by certain variables. In particular $a_i, b_i, p_i, q_i$ will represent the values taken by the variables $\alpha_i, \beta_i, \rho_i, \sigma_i$ respectively. 
\item[-] {\it Slopes} refer to simple closed curves in the toroidal boundary of a 3-manifold. We often differentiate between oriented and unoriented slopes. With respect to a fixed $H_1$-basis the coefficients of an oriented slope will be expressed as $(p,q)$, while the coefficient of an unoriented slope will be expressed as $p/q$. 
\end{itemize}

\section{Acknowledgements}
First of all, we would like to thank Jonathan Spreer and Stephan Tillmann for suggesting this problem and an approach to tackling it. The author is profoundly grateful to Marc Lackenby for his continued support. Finally, the author owes a great deal of thanks to Adele Jackson who contributed substantially and was offered co-authorship but declined. 

\section{Preparatory Results}\label{section:prep}

The following result is a corollary of Thurston-Perelman Geometrization. There are a number of expositions on this topic, see for example \cite[Theorem 10.9]{Sela_iso_algo}, \cite[\S1.4.1]{Geometrisation_French_Book}, \cite{Kuperberg_Homeo} or \cite{Homeo_closed_algo}. 

\begin{theorem}
\label{thm:homeomorphismalgo} There is an algorithm to decide whether two oriented 3-manifolds with, possibly empty, toroidal boundary are orientation-preservingly homeomorphic.
\end{theorem}

We will now compile some results on the geometrisation of 3-manifolds and how the geometry of 3-manifolds can change under Dehn filling.

\subsection{The JSJ-Decomposition}

\begin{definition}
A 3-manifold is {\it prime} if every separating embedded 2-sphere in it bounds a 3-ball. A {\it prime decomposition} of a 3-manifold $M$ is a set of prime 3-manifolds whose orientation-preserving connect sum is $M$. We call the elements of such a set the {\it components} of the prime decomposition. Every manifold is known to admit a unique prime decomposition \cite{Kneser_prime, Milnor_prime}.
\end{definition}

\begin{theorem}[Algorithm 7.1~\cite{JSJ_algo}] \label{thm:prime_decomp}
There is an algorithm that, given a 3-manifold $M$, produces the components of its prime decomposition.
\end{theorem}

\begin{definition} A {\it JSJ-decomposition} of a manifold $M$ is a minimal collection of disjoint incompressible tori cutting $M$ into components that are hyperbolic or Seifert fibred. We say a manifold {\it admits a JSJ-decomposition} if it can be written as a union of Seifert fibred and hyperbolic spaces, with some of their boundary tori identified in pairs. Every prime manifold is known to admit a unique JSJ-decomposition up to isotopy \cite{Johannson_book, JSJ_book, Thurston_geometrisation, Perelman_1, Perelman_2}. An equivalent characterisation of a JSJ-decomposition is:

\end{definition} 
\begin{proposition}[Proposition V.5.2~\cite{JSJ_book}]
\label{prop:jsjdecompcharacterisation} A collection of disjoint incompressible tori in a 3-manifold $M$ is a JSJ-decomposition if:
 \begin{itemize}
 \item[-] it cuts $M$ into hyperbolic or Seifert fibred components; 
 \item[-] no two adjacent Seifert fibred components admit Seifert fibrations with isotopic fibres; 
 \item[-] if any complementary piece is $T^2\times I$, its two boundary components are the same JSJ-torus and the induced homeomorphism $T^2\to T^2$ from pushing the gluing map through the product structure is not periodic. 
 \end{itemize}
\end{proposition}
The last condition here insists that we have no extraneous $T^2\times I$ pieces (for example between two hyperbolic pieces) and that a torus bundle with periodic monodromy has empty JSJ-decomposition. The middle condition implies, in particular, that a torus bundle with reducible monodromy has empty JSJ-decomposition.\\

The following result is a well-known consequence of some foundational results.

\begin{theorem}
\label{thm:jsjalgo} There is an algorithm that, given a prime manifold $M$, outputs a JSJ-decomposition. 
\end{theorem}

\begin{proof}
The only prime manifolds that are not irreducible or boundary-irreducible are $S^2\times S^1$ and $D^2\times S^1$. We can check whether $M$ is one of these, and in either case we have an empty JSJ-decomposition.
There is an algorithm that, given a prime manifold, detects if the manifold is non-Haken and thus has empty JSJ-decomposition~\cite{Haken_recognition}.
We may thus assume that $M$ is Haken, irreducible, and boundary-irreducible, so the result is given in \cite[Algorithm 8.2]{JSJ_algo}.
\end{proof}

\subsection{Hyperbolic Manifolds}\label{subsection:hyperbolic}
We say a 3-manifold is {\it hyperbolic} if it admits a complete, finite-volume Riemannian metric of constant curvature $-1$. Such a metric is known to be unique if it exists. We wish to describe hyperbolic manifolds using triangulations that are suited to their geometry.

\begin{definition} A {\it geometric tetrahedron} is a metric space isometric to the convex hull of four non-coplanar points in the compactified hyperbolic space $\hat{\mathbb H^3}$ with ideal vertices, that is vertices in $\partial\mathbb H^3$, removed. A {\it geometric triangulation} of a hyperbolic manifold $N$ is a decomposition of $N$ into geometric tetrahedra, which are glued isometrically along faces to give the complete hyperbolic metric on $N$.
\end{definition}

We can specify a geometric triangulation by giving a combinatorial triangulation of a manifold $N$ along with the dihedral angles of each geometric tetrahedron, and the data of which vertices are ideal -- that is, on the boundary of $\hat{\mathbb H^3}$.
See \cite[Section 4.1]{Thurston_book} for further details. Given a geometric triangulation of a hyperbolic manifold, we can compute the volume to arbitrary precision. 

\begin{theorem}\cite[Lemma 5.7]{Kuperberg_Homeo}
\label{thm:hyperbolicTriangulation}
There is an algorithm that, given a 3-manifold, determines whether it is hyperbolic and, if it is hyperbolic, produces a geometric triangulation.
\end{theorem}

We now discuss how Dehn fillings of hyperbolic manifolds modify their geometry.
It is a well-known result of Thurston that Dehn filling avoiding a finite list of exceptional slopes preserves hyperbolicity. 

\begin{theorem}\cite[Theorem 5.8.2]{Thurston_book}
\label{thm:hyperbolicDehnfilling}
Given a $\partial$-framed hyperbolic 3-manifold $N$, there is a finite list of slopes for each cusp $T_i$ such that a Dehn filling $N(p_1/q_1,\ldots,p_k/q_k)$ admits a complete hyperbolic structure so long as, for each $i$, the slope $p_i/q_i$ is not in the list of slopes associated to $T_i$.
\end{theorem}

Hodgson-Kerckhoff angle deformation theory \cite{Hodgson_Kerckhoff} allows us to produce the finite list of slopes described in Theorem~\ref{thm:hyperbolicDehnfilling} and provides effective control on the hyperbolic geometry under sufficiently complicated Dehn fillings. The following theorem, which is an immediate consequence of results from \cite{FPS22}, provides the key input from Hodgson-Kerckhoff theory for this section.\\

Suppose that $s\in \pi_1(T)$ is a slope in a boundary torus $T$ of $N$. The boundary torus $T$ has a closed neighbourhood $C$ called a {\it cusp} that lifts to a horosphere in $\mathbb H^3$ and with Euclidean induced metric on $\partial C$. The slope $s$ has a geodesic representative in $\partial C$ of length $l_{\partial C}(s)$. The {\it normalised length} $\hat{L}(s)$ of $s$ is $l_{\partial C}(s)/\sqrt{\text{Area}(\partial C)}$, which is independent of the choice of cusp. The {\it total normalised length} $\hat{L}(s_1, \ldots, s_b)$ of a collection of slopes is defined by: $$\frac{1}{\hat{L}(s_1, \ldots, s_b)^2}: = \sum_{i=1}^b \frac{1}{\hat{L}(s_i)^2}.$$

\begin{theorem} \cite{FPS22}
\label{theorem:longFillingGeometry}
 Let $N$ be a cusped hyperbolic manifold. Fix $\delta > 0$. 
 Consider a Dehn filling $M$ of $N$ on $b$ boundary components, where for each component, the filling slope $s$ satisfies:
 $$\hat{L}(s) \geq \max\left\{7.823, \sqrt{\left(2\pi/\delta + 28.78\right)}\right\}\cdot \sqrt{b}.$$
 Then $M$ is hyperbolic and the core curves of the Dehn filling are isotopic to geodesics of length at most $\delta$.
\end{theorem} 

\begin{proof} First observe that if each slope has normalised length at least $\sqrt{b}\cdot c$, then the total normalised length, denoted by $\hat{L}$, is at least $c$. \cite[Theorem 5.17]{FPS22} states that if $\hat{L}\geq \sqrt{I(1/\sqrt{3})}$, then the Dehn filling $M$ is hyperbolic. Here $\sqrt{I(1/\sqrt{3})}$ denotes a constant less than $7.584$. \cite[Corollary 6.13]{FPS22} states that if $M$ and $N$ are hyperbolic and $\hat{L}\geq 7.823$, then the core curves from the Dehn filling are isotopic to geodesics of length at most $2\pi/(\hat{L}^2-28.78)$. Rearranging this inequality for $\hat{L}$ completes the proof. 
\end{proof}

\begin{proposition}\label{prop:find_short_slopes}
There is an algorithm that, given $c>0$ and a hyperbolic manifold $N$ with a chosen boundary torus $T$, outputs a finite list that includes all slopes $s\in \pi_1(T)$ of normalised length $\hat{L}(s)$ less than $c$.
\end{proposition}
\begin{proof}
By Theorem~\ref{thm:hyperbolicTriangulation}, we can obtain a geometric triangulation of $N$. By truncating the ideal vertex of this triangulation that corresponds to $T$, we obtain a Euclidean triangle for each incident tetrahedron, given by their angles and side lengths, that glue isometrically to form the induced Euclidean geometry on a cusp $C$ at $T$. We can compute, to arbitrary precision, the area of each triangle and hence $\text{Area}(T)$. Take an arbitrary basis $s_1, s_2$ of $\pi_1(T)$. We compute, to arbitrary precision, $l_1:=l_{\partial C}(s_1), l_2:=l_{\partial C}(s_2)$, and the angle $\theta$ between geodesic representatives of $s_1, s_2$. 
Now observe that $l_{\partial C}^2(as_1+bs_2)=(b\cdot l_2\cos(\theta)+al_1)^2+(b\cdot l_2\sin(\theta))^2.$ 
Hence an algorithm can output a finite list of slopes containing all pairs $a,b$ with $l_{\partial C}(as_1+bs_2)\leq c\cdot \sqrt{\text{Area}(\partial C)}$ and therefore all slopes $s$ with $\hat{L}(s)\leq c$. 
\end{proof}

Given a hyperbolic manifold $M$, the length of the shortest geodesic is called the {\it systole length} and is denoted $\text{sys}(M)$.

\begin{proposition}
\label{prop:hypDehnfilling}
 There is an algorithm that, given two manifolds $N$ and $M$ where $N$ is hyperbolic, outputs a finite set of slopes in $\partial N$ such that any Dehn filling of $N$ to $M$ uses at least one of these slopes.
\end{proposition}

\begin{proof}
If $M$ isn't hyperbolic, or we have a lower bound on $\text{sys}(M)$, we are done by Proposition \ref{prop:find_short_slopes} and Theorem~\ref{theorem:longFillingGeometry}. It is thus sufficient to find an algorithm that, given a hyperbolic manifold $M$, outputs a positive lower bound on $\text{sys}(M)$.

For closed $M$, \cite[Theorem 5.9]{Diameter_bound} states that there is a computable $R>0$ such that $\text{sys}(M)>1/(R\cdot l(P))$, for $l(P)$ the total word length of all relators in some presentation $P$ of $\pi_1(M)$. Hence assume that $M$ has $b>0$ cusps and let $X$ denote the Dehn filling of $M$ along slopes $s_1, \ldots, s_b$. Corollary 7.24 of \cite{FPS22} implies that if the slopes $s_i$ each have normalised length $\hat{L}(s_i)\geq \sqrt{b}\cdot 11.23$, then $X$ is hyperbolic and $\text{sys}(M)\geq \text{min}\{0.056, 0.6034\cdot \text{sys}(X)\}$. By Proposition \ref{prop:find_short_slopes}, we can find slopes $s_1, \ldots, s_b$ satisfying this condition and thus find a positive lower bound for $\text{sys}(M)$. This completes the proof.
\end{proof}

Practical estimates for a finite collection of slopes in $\partial N$ satisfying Proposition \ref{prop:hypDehnfilling} are described in \cite{Practical_Dehn_Ancestry_Haraway}.

\begin{proposition}\cite[Theorem 6.1]{Kuperberg_Homeo} \label{prop:list_MCG}
 There is an algorithm that, given an oriented hyperbolic manifold M, outputs all homeomorphisms $M\to M$ up to isotopy.
\end{proposition}

Although there are (in general) infinitely many hyperbolic manifolds of volume at most some bound $V$, a result of J\o{}rgensen-Thurston \cite[Section 5.12]{Thurston_book} states that for every $V$, there is a finite list of hyperbolic manifolds $M_1,\ldots,M_k$ such that every hyperbolic manifold of volume at most $V$ can be obtained from one of the $M_i$ by strongly non-exceptional Dehn filling. By following the exposition of Kobayashi-Rieck, this statement can be made computable. 

\begin{theorem}\cite{Jorgenson_bound}\label{johansen_algo}
There exists an algorithm that, given $V>0$, outputs a list of hyperbolic 3-manifolds $M_1, \ldots, M_k$ such that every hyperbolic 3-manifold $M$ with $\text{vol}(M)\leq V$ is a strongly non-exceptional Dehn filling of one of the $M_i$. 
\end{theorem}
\begin{proof}
Fix $d>0$ and let $\mu>0$ be a Margulis constant for $\mathbb H^3$ - that is, a constant such that the $\mu$-thin part of any hyperbolic 3-manifold $M$, $M_{[0,\mu)}$, consists of truncated cusps and tubes around geodesics. It is shown in \cite[Theorem 1.1, Proposition 1.2]{Jorgenson_bound} that: 
\begin{itemize}
\item[-] There exists a computable constant $C$ such that, if $M$ is a finite volume hyperbolic 3-manifold, then the $\mu$-thick part, $X:= N_d(M_{[\mu,\infty)})$, can be triangulated using at most $C\cdot \vol(M)$ tetrahedra; 
\item[-] $X$ is obtained from $M$ by removing a finite number of geodesics. 
\end{itemize}
By \cite{Geodesic_Complement} the complement of a geodesic link in a hyperbolic 3-manifold is hyperbolic. In particular $X$ is hyperbolic and admits a strongly non-exceptional Dehn filling to $M$. Finally, note that $\mu=0.103$ is a suitable Margulis constant as shown in \cite{Meyerhoff_Margulis_bound}. In summary there is a computable $C$ such that any hyperbolic 3-manifold $M$ with $\text{vol}(M)\leq V$ is a strongly non-exceptional Dehn filling of a hyperbolic manifold triangulated using at most $C\cdot V$ tetrahedra. We may list all finitely-many ways of gluing $C\cdot V$ tetrahedra and delete from this list all gluing which are not hyperbolic 3-manifolds. The resulting set of manifolds gives the desired list $M_1, \ldots, M_k$. 
\end{proof}
Finally, we mention some results on volume change under Dehn filling. 

\begin{theorem} \label{thm:volume_decrease} \cite[Proposition 5.6.2., Lemma 6.5.4.]{Thurston_book} Suppose $N$ is a hyperbolic manifold and $M$ is a hyperbolic non-trivial Dehn filling of $N$. Then $\vol(N)>\vol(M)$. 
\end{theorem}

\begin{theorem}\cite[Theorem 1.1]{FKP08} \label{thm:volume_change}
Suppose $M$ is a hyperbolic manifold, $s_1, \ldots, s_k$ are boundary slopes in distinct boundary tori of $M$ each with normalised length $\sqrt{3}\cdot \hat{L}(s_i)/2 \geq l_{min}>2\pi$. Then the non-trivial Dehn filling $M(s_1, \ldots, s_k)$ is hyperbolic and satisfies: 
$$\vol(M(s_1, \ldots, s_k))\geq \left(1-\left(\frac{2\pi}{l_{min}}\right)^2\right)^{3/2}\vol(M).$$
\end{theorem}

Let us note that the original formulation of \cite[Theorem 1.1]{FKP08} uses the (non-normalised) length of the $s_i$ in a collection of disjoint cusps $C_i$. However, in any collection of maximal cusps every boundary slope has length at least $1$, thus every cusp has area at least $\sqrt{3}/2$, and our reformulation of the theorem holds.

\subsection{Seifert Fibred Spaces} \label{section:Seifert_section}

\begin{definition}
\label{defn:sfs}
A {\it model Seifert fibred space} is a 3-manifold $M$ built in the following manner:
Take a compact surface $\Sigma'$ with $b > 0$ boundary components.
As $\Sigma'$ has non-empty boundary, there is a unique orientable circle bundle over $\Sigma'$.
Let $M' \cong \Sigma\pmttimes{\times} S^1$ be this 3-manifold and endow it with an orientation.
In this circle bundle $M'$ fix a section $\Sigma'\pmttimes{\times}\{1\}$. This construction gives a pair of oriented slopes $\mu_i, \lambda_i$ on each of the $b$ tori $T_i$ of $\partial M'$, with $\mu_i$ a component of $\partial \Sigma' \pmttimes{\times}\{1\}$ and $\lambda_i$ a circle fibre of the bundle. We fix orientations of $\mu_i$ and $\lambda_i$ such that the algebraic intersection $i(\mu_i, \lambda_i)$ is $+1$ and such that isotopies between fibres $\lambda_i$ in distinct boundary components are orientation-preserving. Let $M$ be obtained from $M'$ by Dehn filling along slopes $q_i\lambda_i + p_i\mu_i, p_i\neq 0$ in $n\leq b$ of the boundary tori. Unlike some authors, we allow for $p_i=\pm 1$. Such an $M$ is a model Seifert fibred space. 
\\
 A {\it Seifert structure} on a 3-manifold $N$ is an orientation-preserving homeomorphism from $N$ to a model Seifert fibred space and a {\it Seifert fibred space} is a 3-manifold that admits a Seifert structure. We say that two Seifert structures on $M$ are {\it isomorphic} if there is an orientation-preserving homeomorphism $M\to M$ taking the fibres of one to the fibres of the other.
They are {\it isotopic} if such a homeomorphism is isotopic to the identity. The {\it Seifert data} corresponding to the above Seifert structure on $M$ is $[\Sigma, q_1/p_1,\ldots,q_n/p_n]$, with $\Sigma$ obtained from $\Sigma'$ by capping-off $n$ boundary components. Seifert data determines a Seifert structure up to isomorphism. We adopt the following conventions: 
\begin{itemize}
\item[-] $[\Sigma, q_1/p_1,\ldots,q_n/p_n]$ often denotes a 3-manifold equipped with the corresponding Seifert data; 
\item[-] $[\Sigma]:=[\Sigma, 0/1]$; 
\item[-] in Seifert data each pair of coefficient $p_i, q_i$ is coprime.
\end{itemize} \end{definition}

By the following result there is an algorithm recognising whether a 3-manifold $M$ is a Seifert fibred space. If so we may find Seifert data for a Seifert structure on $M$ by listing all possible Seifert data and applying Theorem \ref{thm:homeomorphismalgo}.

\begin{theorem}\cite[Algorithm 8.1]{JSJ_algo}, \cite[Theorem 1]{Small_SFS_algo} \label{SFS_recognition}
 There is an algorithm that, given a 3-manifold, determines whether it is Seifert fibred. 
\end{theorem}

We now introduce a notation that allows us to keep track of $H_1$-basis of Seifert fibred spaces. 

\begin{definition}A model Seifert fibred space $M$ as above has an $H_1$-basis on every boundary torus consisting of $\lambda_i, \mu_i$, Hence a Seifert structure on a manifold gives an $H_1$-basis of its boundary tori, which are canonical up to swapping orientations of all basis elements. There is a homeomorphism of $M$ arbitrarily permuting the boundary tori and respecting these $H_1$-basis. Hence, taking an arbitrary ordering of boundary tori, a Seifert structure on $M$ gives a $\partial$-framing on $M$, which is canonical up to reordering boundary tori and up to swapping orientations of all basis elements. Let us call this the $\partial$-framing {\it given} by the Seifert structure. 
Moreover, let us call the oriented slopes $\lambda_i$ the {\it fibred slopes} and $\mu_i$ the {\it section slopes}. We say the oriented slope $q\lambda_i+p\mu_i$ is {\it given by} $(q,p)$ {\it with respect to the Seifert structure}, whilst the corresponding unoriented slope is {\it given by} $q/p$.
\end{definition}

\begin{definition}
Consider two spaces $M, M'$ equipped with fixed Seifert structures and with a choice of oriented slopes in some of their boundary tori. We say $M$ and $M'$ are {\it slope-isomorphic} if there is an isomorphism of their Seifert structures that maps the specified oriented boundary slopes to each other and respects orientations on fibred slopes. \\
Suppose the Seifert structure on $M$ has Seifert data $[\Sigma, q_1/p_1,\ldots,q_n/p_n]$ and the chosen oriented slopes are given by $(q_{n+1},p_{n+1}), \ldots, (q_{n+h},p_{n+h})$ with respect to the Seifert structure. Then we call $[\Sigma, q_1/p_1,\ldots,q_n/p_n \mid (q_{n+1},p_{n+1}), \ldots, (q_{n+h},p_{n+h})]$ the {\it Seifert and slope data} of $M$. 
\end{definition}


We will be interested in precisely determining how $\partial$-framed Seifert fibred spaces change under Dehn filling.
First, most Seifert fibred spaces have unique Seifert structures up to isomorphism, and, in fact, up to isotopy.

\begin{theorem} \cite[Theorem 10.1]{Waldhausen_SFS}
\label{thm:nonisomorphicSFS}
 Seifert fibred spaces admit unique Seifert stuctures up to isomorphism, aside from:
 \begin{enumerate}
 \item the solid torus fibres with Seifert data $[D^2]$ and $[D^2, q/p]$;
 \item The $I$-bundle $K^2\tilde{\times} I$ over the Klein bottle $K^2$ fibres with Seifert data $[D^2, 1/2, -1/2]$ and $[\text{M\"ob}]$;
 \item $[S^2, 1/2, -1/2, q/p] $ and $[\mathbb R P^2, -p/q]$ are Seifert data for the same manifold;
 \item $[S^2, 1/2, 1/2, -1/2, -1/2] $ and $[K^2]$ are Seifert data for the same manifold;
 \item lens spaces (including $S^1\times S^2$) fibre in many ways.
 \end{enumerate}
\end{theorem}

\begin{lemma}\cite{Waldhausen_SFS} \cite[Lemma 3]{SFS_Involution}
\label{lemma:uniquefibres}
Let $M$ be an irreducible, $\partial$-irreducible Haken Seifert fibred space that is not $T^3$, $T^2 \times I$, $K^2\tilde{\times} I$ or $K^2\tilde{\times} S^1$.
Then $M$ has only one Seifert fibred structure up to isotopy.
\end{lemma}

We also have a good understanding of these special cases. While lens spaces have infinitely many Seifert fibrations, these are well understood as is discussed in Subsection \ref{subsection:lens}. Moreover, if we have non-isomorphic Seifert fibrations we often have a good understanding of how these are related. 

\begin{lemma}\cite[Lemma II.2.12]{JSJ_book}
\label{lemma:thickenedkleinbottlestructures}
The manifold $K^2\tilde{\times} I$ admits Seifert structures with Seifert data $[D^2, 1/2, -1/2]$ and $[\text{M\"ob}]$ and these are the only Seifert structures up to isotopy. Let $\lambda_1, \mu_1, \lambda_2, \mu_2$ be the fibred and section slope on $[D^2, 1/2, -1/2]$ and $\lambda_2, \mu_2$ respectively. There is a orientation-preserving homeomorphism $\phi\colon [D^2, 1/2, -1/2]\to [\text{M\"ob}]$ with $\phi(\lambda_1)=\mu_2$ and $\phi(\mu_1)=-\lambda_2$. 
%
%
%
\end{lemma}

Given two sets of Seifert data, it is simple to check if the associated Seifert structures are isomorphic.

\begin{theorem}\cite[Corollary 10.3.13]{Martelli_Book}
\label{thm:SFSclassification}
The Seifert fibrations $$[\Sigma, q_1/p_1, \ldots, q_h/p_h]\text{ and }[\Sigma', q'_1/p'_1, \ldots, q'_{h'}/p'_{h'}]$$ with $p_i, p'_i>0$ are orientation-preservingly isomorphic if and only if $\Sigma\cong\Sigma', e:=\sum_{i=1}^h \frac{q_i}{p_i}=e':=\sum_{i=1}^{h'}\frac{q'_i}{p'_i}$ (where if $S$ has boundary $e,e'$ are taken modulo $\mathbb Z$), and after removing all appearances of $(0, 1)$ the multisets $\{(q_1 \text{ mod } p_1, p_1), \ldots, (q_h \text{ mod } p_h, p_h)\}$ and $\{(q'_1 \text{ mod } p'_1, p'_1), \ldots, (q'_{h'} \text{ mod }p'_{h'}, p'_{h'})\}$ are equal.
The corresponding isomorphism takes a singular fibre with coefficients $(q_i, p_i)$ to a singular fibre with coefficients $(q'_j, p'_j)$ such that $p_i=p'_j$ and $q_i\equiv q'_j\pmod{p_i}$. 
\end{theorem}

Let $M$ be a Seifert fibred space with fixed Seifert structure and corresponding Seifert data $[\Sigma, q_1/p_1,\ldots,q_n/p_n]$. If we Dehn fill a component of $M$ along a non-fibred slope, that is $q\lambda_i+p\mu_i, p\neq 0$, we obtain another Seifert fibred space $M'$ with a Seifert structure $[\Sigma', q_1/p_1,\ldots,q_n/p_n, q/p]$. Here $\Sigma'$ is $\Sigma$ with one boundary component capped-off. We call this a {\it non-fibre-parallel} Dehn filling. The $H_1$-basis on $\partial M'$ given by the Seifert structures on $M\subset M'$ and $M'$ are equal, and hence we say the Seifert structure on $M$ gives a $\partial$-framing on $M'$. \\

\begin{proposition}
\label{prop:slopeisomorphic}
 Two manifolds $M$ and $M'$ equipped with oriented boundary slopes that have Seifert and slope data 
 $$[\Sigma, q_1/p_1, \ldots, q_n/p_n\mid (q_{n+1},p_{n+1}), \ldots, (q_{n+h},p_{n+h})]$$ and $$[\Sigma', q'_1/p'_1, \ldots, q'_{n'}/p'_{n'}\mid (q'_{n'+1},p'_{n'+1}), \ldots, (q'_{n'+h'},p'_{n'+h'})],$$
 with $p_i, p'_j\neq 0$ for $1\leq i\leq n+h, 1\leq j\leq n'+h'$ and $p_i, p'_j>0$ for $1\leq i\leq n, 1\leq j\leq n'$ are {slope-isomorphic} if and only if:
 \begin{itemize}
 \item[-] $\Sigma \cong \Sigma'$;
 \item[-] $h=h'$;
 \item[-] $\sum_{i=1}^{n+h} \frac{q_i}{p_i} = \sum_{j=1}^{n'+h'} \frac{q'_i}{p'_i}$ (where if some boundary component does not have a specified slope, the sum is defined modulo $\mathbb Z$);
 \item[-] up to permutation and adding and removing copies of $q/1$: $p_i=p'_i$ and $q_i \equiv q'_i\pmod {p'_i}$ for $1\leq i\leq n$;
 \item[-] up to permutation: $p_i = p'_{n'-n+i}$ and $q_i \equiv q'_{n'-n+i} \pmod {p'_{n'-n+i}}$ for $n< i\leq n+h$.
 \end{itemize}
\end{proposition}

\begin{proof} Let $N:=M(q_{n+1}/p_{n+1}, \ldots, q_{n+h}/p_{n+h})$ and $N':=M'(q'_{n'+1}/p'_{n'+1}, \ldots, q'_{n'+h'}/p'_{n'+h'})$ denote the Dehn fillings of $M, M'$ along the specified boundary curves. As discussed, the Seifert fibrations on $M, M'$ extend to Seifert fibrations on $N, N'$.\\
Suppose the above list of conditions is satisfied and without loss of generality assume the final condition holds without permutation. By Theorem~\ref{thm:SFSclassification}, the above list of conditions is equivalent the Seifert structures on $N, N'$ being isomorphic via a map taking the singular fibres of slopes $p_i/q_i$ to the singular fibres of slopes $p'_{n'-n+i}/q'_{n'-n+i}$ for $n<i\leq n+h$. Denote such an isomorphism by $\phi\colon N\to N'$ and observe that it maps the core curves of the Dehn filling $M\to N$ to the core curves of the Dehn filling $M'\to N'$. Hence $\phi$ restricts to an isomorphism from $M$ to $M'$ taking the unoriented slopes $q_{n+1}/p_{n+1}, \ldots, q_{n+h}/p_{n+h}$ in $\partial M$ to the unoriented slopes $q'_{n'+1}/p'_{n'+1}, \ldots, q'_{n'+h}/p'_{n'+h}$ respectively. Since $p_i= p'_{n'-n+i}\neq 0$ this isomorphism also respects the orientations of all the slopes $(q_{n+1}, p_{n+1}), \ldots, (q_{n+h}, p_{n+h})$. Hence $M, M'$ are slope-isomorphic.\\
Conversely suppose we have a slope-isomorphism $\phi\colon M\to M'$. By considering the oriented intersection number with the fibred slope such a slope-isomorphism must map an oriented slope $(q_i, p_i)$ with $p_i<0$ to an oriented slope $(q'_j, p'_j)$ with $p'_j<0$. Hence by swapping orientations of some of the oriented slopes, we may assume $p_i, p'_j>0$ for all $i,j$. Since $M, M'$ are slope-isomorphic, we know $N, N'$ are isomorphic via a homeomorphism taking singular fibres of slopes $p_i/q_i$ to the singular fibres of slopes $p'_j/q'_j$ with $p_i=p'_j$ and $q_i\equiv q'_j\pmod{p_i}$ for $n<i\leq n+h$. By Theorem~\ref{thm:SFSclassification}, the above list of conditions follows. 
\end{proof}

Let us consider Seifert structures on a few specific manifolds, which we will need later on.

\begin{example}[The solid torus] 
\label{ex:solidtorus}
As described in Theorem~\ref{thm:nonisomorphicSFS}, the Seifert fibrations of $D^2\times S^1$ are $[D^2]$ and $[D^2, q/p]$.
Considering the associated model Seifert fibred space, we can build the fibration $[D^2, q/p]$ of $D^2\times S^1$ by a Dehn filling of $[S^1\times I]$ along the slope $p\mu_2+q\lambda_2$, where $\mu_2$
is the section slope and $\lambda_2$ is the fibred slope.
In this case, the boundary of the essential disc in $[D^2, q/p]$ is $p\mu_1 + q\lambda_1$.
This is almost the only way to construct a Seifert fibred structure on the solid torus as a non-fibre-parallel Dehn filling of a Seifert fibred space $M$: as this Seifert fibred structure has at most one singular fibre, and the base orbifold has no genus, the base orbifold of $M$ must be planar, with at most one singular fibre, and all filling slopes but one (if $M$ had no singular fibres) must be fillings by $e/1$ for some $e\in \mathbb Z$ so as not to introduce extraneous singular fibres.
Thus the possible non-fibre-parallel Dehn fillings to $D^2\times S^1$ are:
\begin{itemize}
\item[-] $M$ is $[D^2-\{x_1,\ldots,x_{k-1}\}, q/p]$ and the filling slopes are $a_1/1,\ldots,a_{k-1}/1$ which gives a Seifert structure isomorphic to $[D^2, q/p]$, and
\item[-] $M$ is $[D^2-\{x_1,\ldots,x_k\}]$ and filling slopes are $a_1/1,\ldots,a_{k-1}/1, q/p$ which gives a Seifert structure isomorphic to $[D^2, q/p]$.
\end{itemize}
Let us find the contractible unoriented boundary slope on $[D^2, a_1/1, \ldots, a_{k-1}/1, q/p]$ with respect to the $H_1$-basis given by the Seifert structure. By Proposition~\ref{prop:slopeisomorphic}, setting $a:=\sum_{i=1}^{k-1} a_i$, the Seifert structure with slope data $[D^2, a_1/1, \ldots, a_{k-1}/1, q/p \mid (q+ap,p)]$ is slope-isomorphic to $[D^2, (q+ap)/p\mid (q+ap,p)]$.
As the boundary slope on the right side is the boundary of the essential disc, the disc boundary is also $(q+ap)/p$ in the original Seifert structure.
\end{example}

\begin{example}[The thickened torus]
\label{ex:thickenedtorus}
The thickened torus $T^2\times I$ has a unique Seifert fibred structure $[S^1\times I]$ up to isomorphism, but infinitely many up to isotopy.
It is determined by the induced foliation on a boundary torus.
As the base orbifold is an annulus with no singular fibres, the only way to obtain $T^2\times I$ by non-fibre-parallel Dehn filling of a Seifert fibred space is to fill $[S^1\times I-\{x_1,\ldots,x_k\}]$ along slopes $a_1/1,\ldots,a_k/1$.\\
In the thickened torus $T^2\times I$ isotopy through the product structure induces an orientation-reversing homeomorphism $\psi$ between the two boundary tori. Let $\lambda_i, \mu_i, i=0,1$, respectively, denote the oriented fibred and the section slopes in the two boundary tori $T^2\times \{0\}$, $T^2\times\{1\}$ with respect to some Seifert structure with data $[S^1\times I]$. Observe that $\psi(\lambda_0)=\lambda_1$ and $\psi(\mu_0)=-\mu_1$. Now let $\lambda'_i, \mu'_i$, respectively, denote the boundary slopes on $T^2\times \{0\}, T^2\times \{1\}$ that are the image of the fibred and the section slopes on $[S^1\times I-\{x_1,\ldots,x_k\}]$ after Dehn filling. We see that $\psi(\lambda'_0)=\lambda'_1$. Let us find $\psi(\mu'_0)$ in terms of $\lambda'_1, \mu'_1$. \\
By Proposition~\ref{prop:slopeisomorphic} for $a:= \sum_{i=1}^k a_i$ there is a slope-isomorphism between $[S^1\times I, a_1/1,\ldots,a_k/1 \mid (0,1),$ $(a,-1)]$ and $[S^1\times I\mid (0,1), (0,-1)]$. The indicated slopes in the first Seifert structure here are $\mu'_0$ and $a\cdot \lambda'_1-\mu'_1$ and the indicated slopes in the second Seifert structure are $\mu_0$ and $-\mu_1$. Hence $\psi(\mu'_0)=a\cdot \lambda'_1-\mu'_1$. In summary, with respect to the $H_1$-basis on $\partial(T^2\times I)$ inherited from the Seifert structure $[S^1\times I-\{x_1,\ldots,x_k\}]$, the change-of-coordinates map through the product structure is given by 
$\begin{pmatrix}
 1 & a\\
 0 & -1
\end{pmatrix}.$\\
Finally observe that the $\partial$-framing on $T^2\times I$ given by some Seifert structure $[S^1\times I]$ and the $\partial$-framing inherited from $[S^1\times I-\{x_1,\ldots,x_k\}]$ are not necessarily equal. Since $T^2\times I$ admits many non-isotopic Seifert structures, we may choose the $H_1$-basis $\lambda_0, \mu_0$ and $\lambda'_0, \mu'_0$ given by such Seifert structure on $T^2\times \{0\}$ to be equal. Having done so, the corresponding $H_1$-basis $\lambda_1, \mu_1$ and $\lambda'_1, \mu'_1$ on $T^2\times \{1\}$ are related as follows: $\lambda_1=\psi(\lambda_0)=\psi(\lambda'_0)=\lambda'_1, \mu'_1=a\cdot \lambda'_1-\psi(\mu'_0)=a\cdot \lambda_1-\psi(\mu_0)=a\cdot \lambda_1+\mu_1$. In summary, the two $H_1$-basis on $T^2\times \{1\}$ are related by the matrix $\begin{pmatrix} 1 & a \\ 0 & 1
\end{pmatrix}.$
\end{example}

\begin{example}[Gluing Seifert fibred spaces]
\label{ex:gluingsfs}
If we glue two manifolds with Seifert structures together (by an orientation-preserving map), we obtain an induced Seifert structure if the fibred slopes are identified.
(If the map is orientation reversing, we still obtain a Seifert structure, but without an induced orientation.)
Suppose we glue a pair of boundary components of the Seifert fibred spaces 
$$M = [\Sigma, q_1/p_1, \ldots, q_h/p_h]\text{ and }M'= [\Sigma', q'_1/p'_1, \ldots, q'_{h'}/p'_{h'}],$$
where we identify the boundary tori $T_1$ and $T_1'$ by a gluing map that takes the fibred slope $\lambda_1$ to the fibred $\pm\lambda'_1$ up to sign.
We may swap the signs of all $H_1$-basis elements on $M'$ such that $\lambda_1$ is identified with $\lambda'_1$. Since the gluing preserves orientation on the 3-manifolds and so is orientation-reversing on the boundary tori, $\mu_1$ is glued to $-\mu_1'+q\lambda'_1$ for some $q\in \mathbb Z$. There is a slope-isomorphism from $[\Sigma, q_1/p_1, \ldots, q_h/p_h \mid (q,-1)]$ to $M^1 = [\Sigma, -q/1, q_1/p_1, \ldots, q_h/p_h \mid (0,-1)]$ and the gluing map from $M^1$ to $M'$ is now $\lambda^1_1 \mapsto \lambda'_1$ and $\mu^1_1 \mapsto -\mu'_1$.
If we glue $M^1$ and $M'$, the Seifert structures glue up, so $M \cup_{T_1 = T'_1} M'$ has the following Seifert structure:
$$ [\Sigma\#_\partial \Sigma', -q/1, q_1/p_1, \ldots, q_h/p_h, q'_1/p'_1,\ldots, q'_{h'}/p'_{h'}].$$
\end{example}

\subsection{Lens Spaces} \label{subsection:lens}
\begin{definition}\label{def:lens} A {\it lens space} is a pair of solid tori with their boundary tori identified. We let $L(p,q)$ be the lens space obtained as follows: for $i=1,2$ let $A_i$ be an oriented solid torus, let $T_i=\partial A_i$, let $m_i, l_i$ be a positively oriented basis of $\pi_1(T_i)$ with $m_i$ the contractible meridian curve in $A_i$, and finally let $\varphi\colon T_2\to T_1$ be an orientation reversing homeomorphism with $\varphi(m_2)=qm_1+pl_1$. Then $L(p,q)$ denotes the gluing $A_1\cup_\varphi A_2$. In particular, $L(p,q)$ comes equipped with an orientation. Unlike some authors, we do not require that $p\geq 0$. 
\end{definition}

The coefficients $p,q$ are not uniquely specified by the homeomorphism class of $L(p,q)$ but this ambiguity is specified by the following classical result. We provide a proof as we could not find the orientation preserving case readily available in the literature. 

\begin{lemma}\label{lemma:classification_lens}
The lens space $L(p,q)$ and $L(p',q')$ are orientation preservingly homeomorphic if and only if for some $\varepsilon\in \{\pm 1\}$ we have $p'=\varepsilon \cdot p$ and $q'\equiv \varepsilon \cdot q^{\pm}\pmod p$. 
\end{lemma}

\begin{proof}
It is shown, for example in \cite[Proof of Theorem 10.1.12]{Martelli_Book} that a lens space contains, up to isotopy, a unique torus cutting it into two solid tori. Hence we know a lens space is uniquely given by the construction in Definition \ref{def:lens} up to changing the basis $m_i, l_i$ and swapping the solid tori $A_1, A_2$. As unoriented curves the $m_i$ are uniquely defined by being contractible in $A_i$. The unoriented curves $l_i$ are then fixed up to adding multiples of $m_i$, which corresponds to changing $q$ by multiples of $p$. Ambiguity in orientating $m_i, l_i$ swaps the sign of both $p$ and $q$. Hence the only ambiguity we need to account for is swapping the solid tori $A_1, A_2$. 
In the construction of Definition \ref{def:lens} we have that, for some $x,y\in \mathbb Z$ with $qy-px=-1$:
$$\begin{pmatrix} \varphi(m_2)\\ \varphi(l_2)\end{pmatrix}=\begin{pmatrix} q & p \\ x & y \end{pmatrix} \cdot \begin{pmatrix} m_1 \\ l_1 \end{pmatrix} \Rightarrow 
\begin{pmatrix} m_1 \\ l_1 \end{pmatrix} = \begin{pmatrix} -y & p \\ x & -q\end{pmatrix}\cdot \begin{pmatrix} \varphi(m_2)\\ \varphi(l_2)\end{pmatrix}.
$$
Hence swapping $A_1, A_2$ corresponds to the ambiguity $L(p,q)\cong L(p, -y)$. However $y\equiv -q^{-1} \pmod p$, concluding the proof. 
\end{proof}


Let us now discuss the Seifert fibrations of lens spaces.

\begin{lemma}\label{lemma:lensspacecharacterisation}
The Seifert fibrations of a lens space are over $S^2$ with at most two singular fibres, or over $\mathbb R P^2$ with at most one singular fibre. Moreover: 
\begin{enumerate}
\item $L(p,q)$ fibres as $[\mathbb RP^2, b/a]$ if and only if for some $k\neq 0$: $p=4k, q=2k\pm 1, b/a=\mp 1/k]$; 
\item Fix $q_{+}=q$ and $q_{-}\in \mathbb Z$ with $q\cdot q_{-}\equiv 1 \pmod p$. Then $L(p,q)$ fibres as $[S^2, b_0/a_0, b_1/a_1]$ if and only if for some $x, y\in \mathbb Z, \varepsilon \in \{\pm 1\}$: $$\varepsilon \cdot p=b_1\cdot a_0+a_1\cdot b_0, ~~\varepsilon \cdot q_\pm = -y\cdot b_1-a_1\cdot x ,\text{ and }b_0y-a_0x=-1.$$
\end{enumerate}
\end{lemma}
\begin{proof}
In \cite[Lemma 4.1]{Lens_SFS} and its correction \cite[Lemma 2.1]{Lens_SFS_correction} it is shown that a Seifert fibration of a lens space is over $S^2$ with at most two singular fibres, or over $\mathbb R P^2$ with at most one singular fibre. Moreover, by \cite[Proposition 3.1]{Lens_SFS_correction} (1) gives all Seifert fibration of a lens space over $\mathbb RP^2$ up to isomorphism. Hence it remains to determine which Dehn fillings of $T^2\times I$ give $L(p,q)$.\\

Consider the oriented basis $\mu_0, \lambda_0\in T^2\times \{0\}, \mu_1, \lambda_1\in T^2\times \{1\}$ given by the product Seifert fibration on $T^2\times I$. Hence $\lambda_0, \lambda_1$ are isotopic and $\mu_0, -\mu_1$ are isotopic. Then $[S^2, b_0/a_0, b_1/a_1]$ is a Dehn filling of $T^2\times I$ along $m_0=b_0\lambda_0+a_0\mu_0$ and $m_1=b_1\lambda_1+a_1\mu_1=b_1\lambda_0-a_1\mu_0$. Let $m_0, l_0$ form a positively oriented basis for $T^2\times \{0\}$, so $l_0=x\lambda_0+y\mu_0$ with $b_0y-a_0x=-1$. In summary we have: 
$$\begin{pmatrix} m_0 \\ l_0\end{pmatrix}=\begin{pmatrix} b_0 & a_0 \\ x & y\end{pmatrix}\cdot \begin{pmatrix} \lambda_0 \\ \mu_0\end{pmatrix}\Rightarrow \begin{pmatrix} \lambda_0 \\ \mu_0\end{pmatrix}=\begin{pmatrix} -y & a_0 \\ x & -b_0\end{pmatrix}\cdot \begin{pmatrix} m_0 \\ l_0\end{pmatrix}\Rightarrow$$
$$m_1=b_1(-y\cdot m_0+a_0\cdot l_0)-a_1(x\cdot m_0-b_0\cdot l_0)=(-y\cdot b_1-a_1\cdot x)\cdot m_0+(b_1\cdot a_0+a_1\cdot b_0)\cdot l_0.$$
In summary for any $x,y$ with $b_0y-a_0x=-1$, we know that $[S^2, b_0/a_0, b_1/a_1]$ is homeomorphic to $L(b_1\cdot a_0+a_1\cdot b_0, -y\cdot b_1-a_1\cdot x)$. Hence $[S^2, b_0/a_0, b_1/a_1]$ is homeomorphic to $L(p,q)$ if and only if for some $\varepsilon \in \{\pm 1\}$ we have $\varepsilon \cdot p=b_1\cdot a_0+a_1\cdot b_0, \varepsilon \cdot q_\pm \equiv -y\cdot b_1-a_1\cdot x\pmod p$. Finally observe that replacing $x$ by $x+b_0$ and $y$ by $y+a_0$ we still satisfy $b_0y-a_0x=-1$ and change $-yb_1-a_1x$ by $p$. Hence we may take the equality $\text{mod }p$ to be over $\mathbb Z$. This gives the desired equalities. Similar work also appears in \cite[Chapter I.4]{Jankins_Neumann-SFS}. 
\end{proof}

\begin{corollary} \label{Lens_recognition}
There is an algorithm that, given a 3-manifold $M$, determines whether $M$ is a lens space. 
\end{corollary}
\begin{proof}
By Theorem \ref{SFS_recognition} we may check whether $M$ is Seifert fibred and if so find a Seifert structure for $M$. If $M$ has a Seifert structure with base surface a sphere with at most two singular fibres, or with Seifert data isomorphic to $[\mathbb RP^2, \pm 1/k]$ then $M$ is a lens space. Otherwise, by Lemma \ref{lemma:lensspacecharacterisation}, $M$ is not a lens space. 
\end{proof}


\subsection{Torus Bundles}
\begin{definition}
Torus bundles are mapping tori of homeomorphisms $\phi\colon T^2\to T^2$ of the torus. Any such $\phi$ is isotopic to a map induced by a matrix $X\in SL_2(\mathbb Z)$ and we define $\text{tr}(\phi):=\text{tr}(X)$. There are three cases for $\phi$, namely: 
\begin{enumerate}
\item $|\text{tr}(\phi)|=2$ in which case $\phi$ is reducible; 
\item $|\text{tr}(\phi)|<2$ in which case $\phi$ has finite order; 
\item $|\text{tr}(\phi)|>2$ in which case $\phi$ is called {\it Anosov}.
\end{enumerate}
\end{definition} 

In the first two cases, the torus bundle is Seifert fibred and in the last case the torus bundle has non-trivial JSJ-decomposition consisting of a torus-fibre of the bundle. Let us briefly prove the following result: 

\begin{proposition} \label{Recognising_torus_bundle}
There is an algorithm that, given a manifold $M$, determines whether $M$ is a torus bundle. 
\end{proposition}
\begin{proof} It is proved in \cite[Section 2]{Hatcher_book} that Seifert data of a Seifert structure on a torus bundle is isomorphic to one of the following: $[T^2, n/1], [\Sigma_{-2}, n/1], [T^2], [S^2, 1/2, 1/2, -1/2, -1/2], [S^2, 2/3, -1/3, -1/3],$ $[S^2, 1/2, -1/4, -1/4]; [S^2, 1/2, -1/3, -1/6].$ Here $\Sigma_{-2}$ is the genus 2 unorientable surface. By the discussion of Section \ref{section:Seifert_section} we may check whether $M$ is Seifert fibred, if so find Seifert data for $M$, and check whether this Seifert data is isomorphic to one of the above list. In doing so we determine whether $M$ is a Seifert fibred torus bundle. If $M$ has non-empty JSJ-decomposition $\mathcal{T}_M$, we check whether $M-\mathcal{T}_M$ is a thickened torus $T^2\times I$. In doing so we determine whether $M$ is an Anosov torus bundle, concluding the proof.
\end{proof}

\section{Mono-Quadratic Systems}\label{section:mono_quadratic}
In this section we reduce several questions that will appear later on to mono-quadratic systems. But first to gain some familiarity, we describe a superficial augmentation of the class of linear systems. 

\begin{lemma}\label{lemma:linearSystemGeneralise}
There is an algorithm that, given an expression composed of linear equations and inequalities in variables $\{\alpha_i\}$, where some equations are possibly modulo a fixed non-zero integer, joined by \texttt{and} ($\land$), \texttt{or} ($\lor$) and \texttt{not} ($\lnot$) operators, constructs a linear system $\mathcal{L}$ in variables $\{\alpha_i, \beta_j\}$ such that values of the $\alpha_i$ are a solution to the initial expression if an only if they are part of a solution to $\mathcal{L}$. 
\end{lemma}

\begin{proof}
For each linear equation that is modulo a fixed non-zero integer $p$, introduce an additional variable $\beta_j$, and replace the expression ``$\equiv r\pmod p$'' with ``$=\beta_j\cdot p + r$''.
For example, an expression such as ``$\alpha_i \equiv 5 \pmod{7}$'' would become ``$\alpha_i = 7\beta_j + 5$''.

Use the rules that $\lnot(x\land y)\equiv \lnot x \lor \lnot y$ and $\lnot (x\lor y) \equiv \lnot x \land \lnot y$ to move the $\lnot$ operators to the bottom level of the expression, so each acts on a single equation.

Convert negated inequalities to non-negated inequalities (e.g. $\lnot(x\geq y)$ to $x < y$. Moreover, convert negated equalities (e.g.\ $\lnot(x = 5)$) to the equivalent pair of inequalities joined by $\lor$ (e.g.\ $x \leq 4 \lor x \geq 6$).
We are left with only $\lor$ and $\land$ operators.

Then use the rule 
$x \land (y\lor z) \equiv (x\land y) \lor (x\land z)$ to move the $\lor$ operators to the top level of the expression, so that it becomes an expression of the form $A_1\lor A_2 \lor \ldots$, where the $A_i$ contain only $\land$ operators and contains no negations or modulo equations. Define sets $\mathcal{S}_i$ consisting of all equalities and inequalities appearing in each $A_i$. The union of the $\mathcal{S}_i$ gives the desired linear system $\mathcal{L}$.
\end{proof}

Being somewhat more careful to preserve the partition properties, one may similarly encode a set of linear Diophantine (in)equalities and quadratic Diophantine equalities as a mono-quadratic system. \\

Let us recall Theorem \ref{lemma:linearsystemsolution} and discuss how it follows directly from \cite[Theorem 0.1]{Monoquadratic_Algo}. \\

\noindent
{\bf Theorem \ref{lemma:linearsystemsolution}.} There is an algorithm that, given a mono-quadratic system, decides whether its solution set is empty, finite, or infinite. If the solution set is finite, the algorithm produces the solution set; if infinite, given a number $n$, the algorithm produces $n$ members of the solution set.

\begin{proof} Consider first a mono-quadratic system $\mathcal{Q}$ consisting of exactly one set of (in)equalities $\mathcal{S}$, which contains at most one degree 2 equality. In other words, the solution set of $\mathcal{Q}$ is a subset of $\mathbb Z^k$ at which a single degree 2 polynomial vanishes and additionally finitely many degree 1 (in)equalities are satisfied. \cite[Theorem 0.1]{Monoquadratic_Algo} states that there is an algorithm which, given such $\mathcal{Q}$, determines whether the solution is empty, finite, or infinite and if the solution set is finite outputs all solutions. We want to upgrade this to the case of general mono-quadratic systems. \\

Consider instead a mono-quadratic system $\mathcal{Q}$ consisting of exactly one set of (in)equalities $\mathcal{S}$. The partition $\mathcal{P}$ of variables of $\mathcal{Q}$ induces a partition of the (in)equalities of $\mathcal{S}$ whose parts we denote $\mathcal{S}^{(1)}, \ldots, \mathcal{S}^{(l)}$. For $j=1, \ldots, l$ let $\mathcal{Q}^{(j)}$ denote the mono-quadratic system in variables the $j$-th part of $\mathcal{P}$ and with (in)equality set $\mathcal{S}^{(j)}$. Then the solution set of $\mathcal{Q}$ is exactly the product of the solution sets of $\mathcal{Q}^{(j)}$. However each of the $\mathcal{Q}^{(j)}$, by assumption, contain only a single degree 2 equality and are of the form discussed above. Hence we may decide whether each of the $\mathcal{Q}^{(j)}$ has empty, finite, or infinite solution set. We thus may decide whether $\mathcal{Q}$ has empty, finite, or infinite solution set. Similarly if $\mathcal{Q}$ has finite solution set we have an algorithm outputting all solutions. \\

Finally consider some general mono-quadratic system $\mathcal{Q}$ in variables $\alpha_1, \ldots, \alpha_k$ consisting of the sets of (in)equalities $\mathcal{S}_1, \ldots, \mathcal{S}_h$. For $j=1, \ldots, h$, let $\mathcal{Q}_j$ denote the mono-quadratic system given by the single set of (in)equalities $\mathcal{S}_j$. Then the solution set of $\mathcal{Q}$ is the union of the solution sets of the $\mathcal{Q}_j$. However, by the above discussion there is an algorithm that, given $\mathcal{Q}_j$, decides whether $\mathcal{Q}_j$ has empty, finite, or infinite solution set. Hence we may decide whether $\mathcal{Q}$ has empty, finite, or infinite solution and similarly we have an algorithm that outputs all solutions of $\mathcal{Q}$ if there are finitely many. Finally if there are infinitely many solutions we may test all possible values of $\alpha_1, \ldots, \alpha_k$ to produce a list of solutions of arbitrary length. 
\end{proof}

Let us conclude by reducing several 3-manifold questions to linear or mono-quadratic systems. 

\begin{proposition}\label{prop:SFS_linear} There is an algorithm that, given Seifert and slope data $[\Sigma', q'_1/p'_1, \ldots, q'_{n'}/p'_{n'} \mid (q'_{n'+1}, p'_{n'+1}), $ $\ldots, (q'_{n'+h'},p'_{n'+h'})]$, a surface $\Sigma$, and $n,h\in \mathbb Z_{\geq 0}$, outputs a linear system $\mathcal{L}$ such that $p_1, q_1, \ldots, p_{n+h}, q_{n+h}$ are part of a solution to $\mathcal{L}$ if and only if $\gcd(p_1, q_1)=\ldots=\gcd(p_{n+h}, q_{n+h})=1$ and the given data is slope-isomorphic to $[\Sigma, q_1/p_1, \ldots, q_n/p_n \mid $ $(q_{n+1},p_{n+1}), \ldots, $ $(q_{n+h},p_{n+h})]$. 

\end{proposition}
\begin{proof}
This result relies predominantly on Proposition \ref{prop:slopeisomorphic}. By reordering and multiplying $p'_i, q'_i$ by $-1$, we may assume that $p'_1\geq \ldots \geq p'_{n'}> 0$. If $h\neq h'$ or $\Sigma \not \cong \Sigma'$ then no values of $p_i, q_i$ will make the data slope-isomorphic and we may take $\mathcal{L}$ with empty solution set. Hence suppose $h=h'$ and $\Sigma \cong \Sigma'$. Fix $k\leq n, n'$, a permutation $\sigma_1$ of $\{1, \ldots, n\}$, and a bijection $\sigma_2\colon \{n+1, \ldots, n+h\}\to \{n'+1, \ldots, n'+h'\}$. By Lemma \ref{lemma:linearSystemGeneralise} find a linear system $\mathcal{L}_{\sigma_1,\sigma_2, k}$ such that $p_1, q_1, \ldots, p_{n+h}, q_{n+h}$ are part of a solution to $\mathcal{L}_{\sigma_1, \sigma_2, k}$ if and only if: 
\begin{itemize}
\item[-] for $1\leq i\leq h$: $p_{\sigma_2(n+i)}= p'_{n'+i}$ and $q_{\sigma_2(n+i)}\equiv q'_{n'+i} \pmod {p'_{n'+i}}$; 
\item[-] for $1\leq i\leq k$: we have $p_{\sigma_1(i)}=\pm p'_{i}$ and $q_{\sigma_1(i)}\equiv \pm q'_i \pmod{p'_i}$; 
\item[-] for $k<i\leq n$: $p_{\sigma_1(i)}=\pm 1$ and for $k<i\leq n'$: $p'_i=1$; 
\item[-] if $\Sigma$ has exactly $h$ boundary components then: 
$$\sum_{i=1}^k \frac{q_{\sigma_1(i)}}{\pm p'_i}+\sum_{i=k+1}^n \frac{q_{\sigma_1(i)}}{\pm 1}+\sum_{i=1}^h \frac{q_{\sigma_2(n+i)}}{p'_{n'+i}}=
\sum_{i=1}^k \frac{q'_{i}}{p'_i}+\sum_{i=k+1}^{n'} \frac{q'_{i}}{1}+\sum_{i=1}^h \frac{q_{n+i}}{p'_{n'+i}}.$$
By clearing denominators this is a linear equality with integer coefficents; 
\item[-] if $\Sigma$ has fewer than $h$ boundary components then the above sums are equal modulo $\mathbb Z$. By clearing denominators this is a linear equality modulo $\prod_{i=1}^{n'+h'}p'_i$. 
\end{itemize}
Since $\gcd(p'_1, q'_1)=\ldots=\gcd(p'_{n'+h'}, q'_{n'+h'})=1$, if the pairs $(p_1, q_1), \ldots, (p_{n+h}, q_{n+h})$ satisfy the first two points then they are all coprime. Moreover, by Proposition \ref{prop:slopeisomorphic}, $p_1, q_1, \ldots, p_{n+h}, q_{n+h}$ give slope-isomorphic data if and only if they are part of a solution to $\mathcal{L}_{\sigma_1, \sigma_2, k}$ for some $\sigma_1, \sigma_2, k$. Hence taking the union of all the $\mathcal{L}_{\sigma_1, \sigma_2, k}$ gives the desired linear system.
\end{proof}

\begin{lemma}\label{lemma:lens_quadratic} There is an algorithm that, given a lens space $L(p,q)$, outputs a mono-quadratic system $\mathcal{Q}$ such that $a_1, b_1, a_2, b_2$ are part of a solution to $\mathcal{Q}$ if and only if $\gcd(a_1, b_1)=\gcd(a_2, b_2)=1$ and $[S^2, b_1/a_1, b_2/a_2]$ gives a Seifert fibration for $L(p,q)$. 
\end{lemma}

\begin{proof}
Let $q_+=q$ and fix $q_-\in \mathbb Z$ such that $q\cdot q_{-}\equiv 1\pmod p$. Recall from Lemma \ref{lemma:lensspacecharacterisation} that $[S^2, b_1/a_1, b_2/a_2]$ is a Seifert fibration of $L(p,q)$ if and only if for some $\varepsilon \in \{\pm 1\}$ and $x,y\in \mathbb Z$ we satisfy the following three equations: 
$$(1)~ \varepsilon \cdot p=b_2\cdot a_1+a_2\cdot b_1, ~~(2)~ \varepsilon \cdot q_\pm = -y\cdot b_2-a_2\cdot x ,\text{ and }(3)~ b_1\cdot y-a_1\cdot x=-1.$$
Fix $\varepsilon \in \{\pm 1\}$ and suppose $a_i, b_i, x, y$ satisfy equation (3). Then every solution to equation (1) is of the form: 
$$b_2=x\cdot \varepsilon \cdot p+s\cdot b_1,\qquad a_2=-y\cdot \varepsilon \cdot p-s\cdot a_1\text{ for some }s\in \mathbb Z.$$ Plugging this into equation (2) we obtain: $\varepsilon \cdot q_\pm=-y\cdot s\cdot b_1+x\cdot s\cdot a_1=s$. In summary, we see that $[S^2, b_1/a_1, b_2/a_2]$ is a Seifert fibration of $L(p,q)$ if and only if for some $\varepsilon, x, y$ we have: $$b_1\cdot y-a_1\cdot x=-1, \qquad b_2=x\cdot \varepsilon \cdot p+\varepsilon \cdot q_\pm \cdot b_1, a_2=-y\cdot \varepsilon \cdot p-\varepsilon \cdot q_{\pm}\cdot a_1.$$
This is a mono-quadratic system $\mathcal{Q}'$. It only remains to deal with coprimality, that is find a mono-quadratic system $\mathcal{Q}$ such that $a_1, b_1, a_2, b_2$ are part of a solution to $\mathcal{Q}$ if and only if they are part of a solution to $\mathcal{Q}'$ and $\gcd(a_1, b_1)=\gcd(a_2, b_2)=1$. If $a_1, b_1, a_2, b_2$ are part of a solution to $\mathcal{Q}'$ then by equation (3) $\gcd(a_1, b_1)=1$ and by equation (1) $\gcd(a_2, b_2)\mid p$. Hence a suitable mono-quadratic system $\mathcal{Q}$ is: $$\mathcal{Q}:= \mathcal{Q}'\wedge \bigwedge_{r=2}^p \left(\{a_2\not\equiv 0 \pmod r\}\vee\{b_2\not\equiv 0 \pmod r\}\right).$$
\end{proof}

\begin{example}\label{ex:monoquad_needed}
Consider $A=[D^2-\{x_A\}, \frac{r_A}{s_A}], B=[D^2-\{x_B\}, \frac{r_B}{s_B}]$ with $s_A, s_B\neq \pm 1$. Let $T_A$ be a boundary torus of $A$, let $\mu_A, \lambda_A\in H_1(T_A)$ respectively be the section and fibred slopes, and let $T_B, \mu_B, \lambda_B$ be defined similarly. Let $\psi\colon T_B \to T_A$ be the homeomorphism given by $\psi(\lambda_B)=\mu_A, \psi(\mu_B)=\lambda_A$ and let $N=A\cup_\psi B$. Equip $N$ with the $\partial$-framing inherited from the Seifert structures on $A,B$. We want to ask whether $N(\frac{q_1}{p_1}, \frac{q_2}{p_2})$ can be $S^3$. If neither $A$ nor $B$ Dehn fill to become a solid torus, then $N(\frac{q_1}{p_1}, \frac{q_2}{p_2})$ cannot Seifert fibre over the sphere. We may thus, without loss of generality, assume $p_2=1$. 
Hence $B(\frac{q_2}{p_2})=[D^2, \frac{q_2}1, \frac{r_B}{s_B}]=[D^2, \frac{r_B+q_2\cdot s_B}{s_B}]$.
%
By Example \ref{ex:solidtorus} we see that hence $N(\frac{q_1}{p_1}, \frac{q_2}{p_2})$ is the Dehn filling of $[D^2, \frac{q_1}{p_1}, \frac{r_A}{s_A}]$ along $\psi(s_B\cdot \mu_B+(r_B+q_2\cdot s_B)\cdot \lambda_B)=s_B\cdot \lambda_A+(r_B+q_2\cdot s_B)\cdot \mu_A$. Hence $N(\frac{q_1}{p_1}, \frac{q_2}{p_2})=[S^2, \frac{q_1}{p_1}, \frac{r_A}{s_A}, \frac{s_B}{r_B+q_2\cdot s_B}]$. 
We know that all Seifert fibration of $S^3$ have base surface $S^2$ and at most two singular fibres. Moreover, by Lemma \ref{lemma:lensspacecharacterisation}, $[S^2, \frac{b_1}{a_1}, \frac{b_2}{a_2}]\cong S^3\cong L(1,0)$ if and only if $b_1\cdot a_2+b_2\cdot a_1=\pm1$. Hence $N$ admits a Dehn filling to $S^3$ if and only if there are $q_1, q_2\in \mathbb Z$ such that one of the following holds: 
\begin{itemize}
\item[-] $p_1=1$ and $(r_A+q_1\cdot s_A)\cdot (r_B+q_2\cdot s_B)+s_A\cdot s_B=\pm1$, or
\item[-] $r_B+q_2\cdot s_B=\varepsilon $, $s_A\cdot (q_1+\varepsilon\cdot s_B\cdot p_1)+r_A\cdot p_1=\pm 1$, for some $\varepsilon \in \{\pm 1\}$. 
\end{itemize}
The former equation is a mono-quadratic equation in variables $q_1, q_2$. By chaining together more Seifert fibred components we may obtain more complicated mono-quadratic systems in a similar manner. We hope this gives the reader a heuristic as to why considering the general class mono-quadratic equations is useful when deciding whether two 3-manifolds are related by Dehn filling. 
%
\end{example}


\begin{lemma} \label{lemma:sl2z}
Consider matrices $X_1, \ldots, X_{k+1}\in GL_2(\mathbb Z)$ and let \begin{equation*}
 X:= X_1\cdot \begin{pmatrix}
 1 & n_1 \\
 0 & 1
 \end{pmatrix} \cdot X_2 \cdots X_k\cdot 
 \begin{pmatrix}
 1 & n_k \\
 0 & 1 
 \end{pmatrix}\cdot X_{k+1}.
\end{equation*}
There is an algorithm that, given $X_1, \ldots, X_{k+1}, Y\in GL_2(\mathbb Z)$, outputs a linear system $\mathcal{L}$ such that $n_1, \ldots, n_k\in \mathbb Z$ are part of a solution to $\mathcal{L}$ if and only if $X,Y$ are conjugate by a matrix in $SL_2(\mathbb Z)$. In particular, taking $Y=\text{id}$, there is an algorithm that, given $X_1, \ldots, X_{k+1}$, outputs a linear system $\mathcal{L}$ such that $n_1, \ldots, n_k$ are part of a solution to $\mathcal{L}$ if and only if $X=\text{id}$. 
\end{lemma}
A similar result, with $Y$ always $\text{id}$ but with more general equations, is given in \cite[Theorem 3.2]{Baroni_group_algo}.

\begin{proof} Let us proceed by induction on $k$. For $k=0$, $X,Y$ are conjugate by a matrix in $SL_2(\mathbb Z)$ if and only if we can find $a,b,c,d\in \mathbb Z$ such that: $ad-bc=1$ and $\begin{pmatrix} a & b \\ c & d\end{pmatrix}\cdot X=Y\cdot \begin{pmatrix} d& -b \\ -c & a\end{pmatrix}.$
Observe that this is a mono-quadratic system and by Theorem \ref{lemma:linearsystemsolution} there is an algorithm to decide whether such $a,b,c,d$ exist. Hence let us assume $k>0$ and proceed by induction on $k$. \\

{\bf Claim 1}: We may assume that $X_{k+1}=\text{id}$ and that:
\begin{itemize}
\item[-] If $k>1$: $X_1, \ldots, X_k\neq \begin{pmatrix} \varepsilon_1 & n \\ 0 & \varepsilon_2\end{pmatrix}$ for any $n\in \mathbb Z, \varepsilon_1, \varepsilon_2\in \{\pm 1\}$; 
\item[-] If $k=1$: $X_1=\text{id}$ or $X_1\neq \begin{pmatrix} \varepsilon_1 & n \\ 0 & \varepsilon_2 \end{pmatrix}$ for any $n\in \mathbb Z, \varepsilon_1, \varepsilon_2 \in \{\pm 1\}$.
\end{itemize}
Proof of Claim 1: Conjugating $X$ by a suitable matrix in $SL_2(\mathbb Z)$, we may assume $X_{k+1}=\begin{pmatrix} 1 & 0 \\ 0 & -1 \end{pmatrix}$ or $X_{k+1}=\text{id}$. Observe that: $$X_k\cdot \begin{pmatrix}1 & n_k \\ 0 & 1\end{pmatrix} \cdot \begin{pmatrix} 1 & 0 \\ 0 & -1 \end{pmatrix} = \left(X_k\cdot \begin{pmatrix} 1 & 0 \\ 0 & -1\end{pmatrix}\right) \cdot \begin{pmatrix} 1 & -n_k \\ 0 & 1 \end{pmatrix}.$$
Hence, replacing $n_k$ by $-n_k$ and suitably modifying $X_k$, we may assume $X_{k+1}=\text{id}$.\\

Suppose first that $k>1$. If for some $i=2, \ldots, k$: $X_i=\begin{pmatrix} 1 & n \\ 0 & 1 \end{pmatrix}, n\in \mathbb Z$, then: $$\begin{pmatrix} 1 & n_{i-1} \\ 0 & 1 \end{pmatrix} \cdot X_i\cdot \begin{pmatrix} 1 & n_i \\ 0 & 1 \end{pmatrix}=\begin{pmatrix} 1 & n_{i-1}+n+n_i\\ 0 & 1 \end{pmatrix}.$$ Hence we are done by induction on $k$. Similarly if $X_1=\begin{pmatrix} 1 & n \\ 0 & 1 \end{pmatrix}, n\in \mathbb Z$, we may conjugate and are done by induction. If for some $i=2, \ldots, k$: $X_i=\begin{pmatrix} 1 & n \\0 & -1 \end{pmatrix}$, then $$X_{i-1}\cdot \begin{pmatrix} 1 & n_{i-1} \\ 0 & 1 \end{pmatrix}\cdot X_i=\left(X_{i-1}\cdot \begin{pmatrix} 1 & 0 \\ 0 & -1 \end{pmatrix}\right) \cdot \begin{pmatrix} 1 & -n_{i-1} \\ 0 & 1 \end{pmatrix}\cdot \begin{pmatrix} 1 & n \\ 0 & 1 \end{pmatrix}.$$
Hence, by replacing $n_{i-1}$ with $-n_{i-1}$ and suitably modifying $X_{i-1}$, we are in the case of $X_i=\begin{pmatrix} 1 & n \\ 0 & 1\end{pmatrix}$ and are done by induction as above. Similarly if $X_1=\begin{pmatrix} 1 & n \\ 0 & -1\end{pmatrix}$, by replacing $n_1$ with $-n_1$ and suitably modifying $X_2$, we are done by induction. Finally suppose that for some $i=1, \ldots, k$: $X_i=\begin{pmatrix} \pm1 & n \\ 0 & 1\end{pmatrix}$. Since $-\text{id}$ is in the centre of $GL_2(\mathbb Z)$, we may replace $X_i$ by $-X_i$ and $Y$ by $-Y$ and then apply induction on $k$ as above. \\

Suppose now that $k=1$ and $X_1=\begin{pmatrix} 1 & n \\ 0 & 1 \end{pmatrix}$, then we may replace $X_1$ by $\text{id}$ and $n_1$ by $n+n_1$. Suppose that instead $X_1=\begin{pmatrix} 1 & n \\ 0 & -1 \end{pmatrix}$ and therefore $\text{tr}(X)=\text{tr}\begin{pmatrix} 1 & n+n_1 \\ 0 & -1\end{pmatrix}=0$. Thus we may assume $\text{tr}(Y)=0$ as else we are immediately done. It is shown in \cite[Theorem 3]{GL2Z_tr0} that traceless matrices in $GL_2(\mathbb Z)$ are conjugate in $GL_2(\mathbb Z)$. Hence there exists $Z\in GL_2(\mathbb Z)$ such that $ZXZ^{-1}=Y$. We also have $(ZX)X(ZX)^{-1}=Y$ and one of $Z, ZX$ must be in $SL_2(\mathbb Z)$. So $X,Y$ are conjugate in $SL_2(\mathbb Z)$ and we are done. Finally, suppose that $X_1=\begin{pmatrix} -1 & n \\ 0 & \pm1\end{pmatrix}$. Then multiplying $X_1$ and $Y$ by $-\text{id}$ we reduce to the aforementioned cases. This completes the proof of Claim 1. \\

Assume from now on that $X_1, \ldots, X_{k+1}$ satisfy the properties laid out in Claim 1. \\

{\bf Claim 2.} There exists $C\in \mathbb N$ such that if $|n_1|, \ldots, |n_k|\geq C$, then $X,Y$ are not conjugate by a matrix in $SL_2(\mathbb Z)$. Moreover there exists an algorithm that, given $X_1, \ldots, X_k, Y$, outputs such $C$. \\
Proof of Claim 2: Observe that if $X^2, Y^2$ are not conjugate, then neither are $X,Y$. Hence it is sufficient to prove this result for $X,Y$ replaced by $X^2,Y^2$. We have the following expression for $X^2$: 
$$X^2=X_1\cdot \begin{pmatrix} 1 & n_1 \\ 0 & 1 \end{pmatrix} \cdot X_2 \cdots X_k \cdot \begin{pmatrix} 1 & n_k \\ 0 & 1 \end{pmatrix} \cdot X_1\cdot \begin{pmatrix} 1 & n_1 \\ 0 & 1 \end{pmatrix} \cdot X_2 \cdots X_k\cdot \begin{pmatrix} 1 & n_k \\ 0 & 1\end{pmatrix}.$$
Unless $k=1$ and $X_1=\text{id}$, this new expression satisfies the properties of Claim 1. If $k=1$ and $X_1=\text{id}$, then $X=\begin{pmatrix} 1 & 2\cdot n_1 \\ 0 & 1\end{pmatrix}$ is an expression for $X^2$ satisfying the properties of Claim 1. Clearly $\det(X^2)=\det(Y^2)=1$. In summary, it is sufficient to prove Claim 2 under the assumption $\det(X)=\det(Y)=1$.\\

Suppose that for some $1\leq j\leq k$: $\det(X_j)=-1$ but $\text{det}(X_i)=1$ for all $i<j$. In particular, $j< k$ as otherwise $\det(X)=-1$. Observe that:
 \begin{equation*}
 X_j\cdot 
 \begin{pmatrix}
 1 & n_j \\ 0 & 1 
 \end{pmatrix}
 \cdot X_{j+1}
 = 
 \left[
 X_j\cdot 
 \begin{pmatrix}
 1 & 0 \\ 0 & -1
 \end{pmatrix}
 \right]
 \cdot 
 \begin{pmatrix}
 1 & -n_j \\ 0 & 1
 \end{pmatrix}\cdot 
 \left[
 \begin{pmatrix}
 1 & 0 \\ 0 & -1
 \end{pmatrix}\cdot 
 X_{j+1}
 \right]
 \end{equation*}
By replacing $n_j$ by $-n_j$ and suitably modifying $X_{j+1}$, we may thus assume $\text{det}(X_i)=1$ for $i\leq j$. Repeating this step, we may assume $\text{det}(X_i)=1$ for all $i$. Observe that if $X_i\neq \begin{pmatrix} \varepsilon_1 & n \\ 0 & \varepsilon_2 \end{pmatrix}$, for $\varepsilon_1, \varepsilon_2\in \{\pm 1\}, n\in \mathbb Z$, then $\begin{pmatrix} 1 & 0 \\ 0 & -1\end{pmatrix}\cdot X_i, X_i\cdot \begin{pmatrix} 1 & 0 \\ 0 & -1 \end{pmatrix}\neq \pm \begin{pmatrix} 1 & n \\ 0 & 1 \end{pmatrix}, n\in \mathbb Z$. Hence we may also assume that in our new expression for $X$ we have $X_i\neq \pm \begin{pmatrix} 1 & n \\ 0 & 1 \end{pmatrix}$ or we have $k=1$ and $X_1=\text{id}$. \\

Recall that there is an isomorphism $\varphi:G:=\mathbb Z/2\mathbb Z * \mathbb Z/3\mathbb Z=\langle a\rangle * \langle b \rangle \to PSL_2(\mathbb Z)$
 given by the following map on the generators:
 \begin{eqnarray*}
 a \mapsto \begin{pmatrix}
 0 & -1 \\ 1 & 0 
 \end{pmatrix}, ~~~~~
 b\mapsto \begin{pmatrix}
 1 & -1 \\ 1 & 0
 \end{pmatrix}.
 \end{eqnarray*}
By enumerating all possible words, we may find words $v_1, \ldots, v_{k+1}, w$ in alphabet $a,b, a^{-1}, b^{-1}$ such that $\varphi(v_i)=X_i/\pm, \varphi(w)=Y/\pm$. Notice that $\varphi(ba)=\begin{pmatrix} 1 & 1 \\ 0 & 1 \end{pmatrix}/\pm$. Hence we have that $v_i$ are not of the form $(ba)^n$ for $n\in \mathbb Z$ or we have that $k=1$ and $v_1=\text{id}$. Moreover, $\varphi(v)=X/\pm$ for $v:=v_1\cdot (ba)^{n_1}\cdot v_2\cdot\ldots \cdot v_k\cdot (ba)^{n_k}.$ Hence to prove Claim 2 it is sufficient, given such $v_1, \ldots, v_k, w\in \mathbb Z/2\mathbb Z*\mathbb Z/3\mathbb Z$, find $C\in \mathbb N$ such that if $|n_1|, \ldots, |n_k|\geq C$, then $v,w$ are not conjugate.\\

We adopt the following notation: a {\it word} is a string in the alphabet $\{a, b, b^{-1}\}$. A {\it reduced word} is a word not containing $aa, bb, b^{-1}b^{-1}, bb^{-1}$, or $b^{-1}b$. A {\it cyclically reduced word} is a reduced word such that either its first or its last letter is $a$ but not both. Every word represents an element in $G$ and every element in $g\in G$ is represented by a unique reduced word of length $l(g)$. A cyclically reduced word will have minimal length among all words representing elements of a fixed conjugacy class in $G$. We will require the following claim, which we will prove after completing the proof of Claim 2. \\

{\bf Claim 3}. Consider $g\in \mathbb Z/2\mathbb Z*\mathbb Z/3\mathbb Z$ that is not a power of $ba$ and $m,n\in \mathbb Z$ with $|m|, |n|>l(g)$, then $(ba)^n\cdot g\cdot (ba)^m$ is represented by a reduced word that starts with the same letter as $(ba)^n$, ends with the same letter as $(ba)^m$, and has length at least $2|m|+2|n|-l(g)-4$. \\

Let us now conclude the proof of Claim 2. Suppose that for each $i\in \{1, \ldots, k\}: |n_i|>2\cdot \text{max}\{l(v_1), \ldots, l(v_k)\}+1$. Then $v$ is conjugate to the following: 
$$ (ba)^{c_k}\cdot v_1\cdot (ba)^{b_2}\cdot (ba)^{c_2}\cdot v_2 \cdot \ldots \cdot (ba)^{c_{k-1}}\cdot v_k \cdot (ba)^{b_k},
$$
with $b_i+c_i=n_i$, with $b_i, c_i$ having the same sign, and $|b_i|, |c_i|>\text{max}\{l(v_1), \ldots, l(v_k)\}$. Thus by the Claim 3 $v$ is conjugate to a cyclically reduced word of length at least $2|n_1|+\ldots+2|n_k|-l(v_1)-\ldots -l(v_k)-4k$. In summary there is an algorithmically computable $C$ such that if all $|n_i|>C$ then $v$ has a cyclically reduced word of length greater than $l(w)$ and in particular is not conjugate to $w$. By the above discussion this is the desired $C$ in Claim 2. \\

Let us now deliver on the promised proof of Claim 3.\\

Proof of Claim 3: Take a reduced word for $g$ of length $l(g)$ and consider the corresponding word of length $2|m|+2|n|+l(g)$ representing $(ba)^n\cdot g\cdot (ba)^m$. From this word repeatedly delete appearances of $aa, bb^{-1}, b^{-1}b$ until no more deletions are possible or until all letters of $g$ have been deleted. Call the resulting word $u$. So $u$ represents $(ba)^n\cdot g\cdot (ba)^m$ and has length at least $2|m|+2|n|-l(g)$. Let us separately consider the two reasons for the deletion process to terminate: \\
\indent Case 1. No more deletions are possible and letters of $g$ remain: Observe that the word for $g$ and the words $(ba)^n, (ba)^m$ are reduced and hence all deletions occur at the boundary between these two sets of letters. Hence $u$ is reduced apart from at most two appearances of $bb$ or $b^{-1}b^{-1}$ on the boundary between the letters of $g$ and the letters of $(ba)^n, (ba)^m$. If the only letter of $g$ left remaining is $b$ and it is flanked in $u$ by two other letters $b$, then we remove this $bbb$ and proceed as in Case 2 with a word $u$ of length at least $2|m|+2|n|-l(g)-1$. We proceed similarly if the only letter of $g$ left remaining is $b^{-1}$ flanked by two $b^{-1}$. Moreover, observe that in this subcase $m,n$ must have the same sign for their to be two $b$ or two $b^{-1}$ flanking the remainder of $g$. In all other situations, we may replace the instances of $bb$ or $b^{-1}b^{-1}$ by $b^{-1}$ or $b$ respectively, and we obtain a reduced word for $(ba)^n\cdot g\cdot (ba)^m$ of length at least $2|m|+2|n|-l(g)-2$. \\
\indent Case 2. All letters of $g$ have been deleted: We assume $\text{length}(u)\geq 2|m|+2|n|-l(g)-1$. If $u$ is reduced we are immediately done. If $m,n>0$, then possibly replacing $abaaba$ or $abba$ by $ab^{-1}a$ we obtain a reduced word of length at least $2|m|+2|n|-l(g)-4$. Similarly if $m, n<0$, then possibly replacing $ab^{-1}aab^{-1}a$ or $ab^{-1}b^{-1}a$ by $aba$ we obtain a reduced word of length at least $2|m|+2|n|-l(g)-4$. Hence suppose the $m,n$ have opposite sign. First of all, this implies we didn't delete $bbb$ or $b^{-1}b^{-1}b^{-1}$ to obtain $u$. If we have only deleted letters from $(ba)^n$ and $g$, then $g$ must be a power of $(ba)$. Hence we deleted letters from $g, (ba)^n$, and $(ba)^m$ and obtained a non-reduced word. Thus the reduced word $r$ for $g$ splits into two subwords the first being an initial segment $r_1$ of $(ba)^{-n}$ and the latter being a terminal segment $r_2$ of $(ba)^{-m}$. For the resulting $u$ not to be a reduced word, $r_1$ and $r_2$ must both be even or both be odd length. In particular the last letter of $r_1$ is inverse to the first letter of $r_2$ and $r$ is not reduced, a contradiction. This proves Claim 3. \\

Let us complete the proof of the Lemma. By induction on $k$, for each $n\in \mathbb Z$ and each $i\in \{1, \ldots, k\}$ we can find a linear system $\mathcal{L}_{n,i}$ such that $n_1, \ldots, n_{i-1}, n_{i+1}, \ldots, n_k$ are part of a solution to $\mathcal{L}_{n,i}$ if and only if the corresponding $v,w$ for $n_i=n$ are conjugate. Recall that if all $|n_i|\geq C$, then $X,Y$ are not conjugate. Thus the desired linear system is given by: 
$$\mathcal{L}:=\bigvee_{n, i\in \mathbb Z: |n|<C, 1\leq i\leq k} ((n_i=n) \wedge \mathcal{L}_{n,i}).$$\end{proof}

\section{Finding Minimal Volume $n$-Cusped Manifolds} 
In this section we prove Proposition \ref{prop:min_volume__n_cusp}. We prove this result for completeness and because the proof gives a hint at the ideas required to prove Theorem \ref{prop:mainprop}. \\

\noindent
{\bf Proposition \ref{prop:min_volume__n_cusp}.} There is an algorithm that, given $n\in \mathbb N$, outputs a finite list that includes all minimal volume orientable hyperbolic 3-manifolds with exactly $n$ cusps.

\begin{proof}
Construct an arbitrary $n$-cusped hyperbolic 3-manifold and by the discussion of Subsection \ref{subsection:hyperbolic} compute an upper bound $V$ on its volume. By Theorem \ref{johansen_algo} we may find a finite list of 3-manifolds such that any hyperbolic 3-manifold with volume at most $V$ is a Dehn filling of one of the 3-manifolds of the list. Hence it is sufficient to provide an algorithm that, given $n\in \mathbb N$ and a compact orientable 3-manifold $N$ with, possibly empty, toroidal boundary, outputs a finite list containing all minimal volume $n$-cusped 3-manifolds that are a Dehn filling of $N$. Let us proceed by induction on the number $k$ of cusps of $N$. The case of $k\leq n$ is immediate.\\

By Theorem \ref{thm:prime_decomp}, given $N$, we may find a prime decomposition $N=N_1\#\ldots \#N_p$. Any Dehn filling of $N$ is a connect sum of the corresponding Dehn fillings of the $N_i$. In particular, a Dehn filling of $N$ is hyperbolic only if $p-1$ of the $N_i$ Dehn fill to $S^3$. Thus a hyperbolic Dehn filling of $N$ will be a Dehn filling of one of the $N_i$ and we may assume $N$ is prime.\\

Since $N$ is prime, $N$ admits a JSJ-decomposition. By the discussion of Subsections \ref{subsection:hyperbolic} and \ref{section:Seifert_section} we may find a finite collection of slopes $S$ in $\partial M$ such that $S$ includes all fibred slopes in Seifert fibred JSJ-components of $N$ and such that Dehn filling hyperbolic components JSJ-components of $N$ avoiding slopes in $S$ preserves hyperbolicity. For each slope $s\in S$, we may consider $N(s)$ separately by induction. Hence, it is sufficient to find a finite list containing all minimal volume $n$-cusped 3-manifolds that are a Dehn filling of $N$ not using slopes in $S$. \\

Any such Dehn filling of a Seifert fibred $N$ will be Seifert fibred and, in particular, not hyperbolic. Suppose instead that $N$ is hyperbolic. Consider some arbitrary Dehn filling of $N$, not using slopes in $S$, and compute an upper bound $V<\text{vol}(N)$ for its volume. By Theorem \ref{thm:volume_change}, we may find a constant $C$ such that a Dehn filling of $N$ using slopes of total normalised length at least $C$ will have volume greater than $V$. By Proposition \ref{prop:find_short_slopes}, we may find a finite list of slopes $S'$ containing all slopes of normalised length less than $C$. We are then done by considering $N(s), s\in S'$ separately by induction on $k$. Hence we may assume $N$ has multiple JSJ-components $A_1, \ldots, A_l$ and proceed by induction on $l$. \\

A Dehn filling $M$ of $N$ avoiding slopes in $S$ will be obtained, by identifying some boundary tori in pairs, from the corresponding Dehn fillings $B_1, \ldots, B_l$ of the JSJ-components $A_1, \ldots, A_l$. Recall that the only boundary-incompressible Seifert fibred space is the solid torus $D^2\times S^1$. Hence, if none of the $B_i$ are solid tori then by an innermost disc argument, we see that $M$ contains an incompressible torus and is not hyperbolic. So for $M$ to be hyperbolic some $B_i\cong D^2\times S^1$ and, in particular, $M$ is a Dehn filling of $N-A_i$. We are thus done by induction on the number of JSJ-components.
\end{proof}

\section{Finding Minimal Volume Links}\label{section:min_volume}
In this section we prove Theorems \ref{thm:compute_volt}, \ref{thm:volumetheorem}, and \ref{thm:biglist_volumetheorem} assuming Corollary \ref{thm:hyperbolic_Dehn_Ancestry}. \\

\noindent
{\bf Theorem \ref{thm:volumetheorem}.} There is an algorithm that, given $k\in \mathbb N$ and a compact orientable 3-manifold $M$ with, possibly empty, toroidal boundary, outputs a finite list that includes all minimal volume $k$-component hyperbolic links in $M$ up to homeomorphisms of the link complement in $M$. 

\begin{proof} [Proof of Theorem \ref{thm:volumetheorem} assuming Corollary \ref{thm:hyperbolic_Dehn_Ancestry}] Observe first that it is sufficient to work under the assumption that $M$ is connected. For $k=0$ we need only check whether $M$ is itself hyperbolic, which can be done by Theorem \ref{thm:hyperbolicTriangulation}. Hence assume $k>0$. It is shown in \cite[Theorem 6.1]{Myers_Knot_Construction} that any compact connected orientable 3-manifold $A$ with, possibly empty, toroidal boundary contains a hyperbolic knot $K$. Moreover, the proof of \cite[Theorem 6.1]{Myers_Knot_Construction} gives a construction of $K$ from a triangulation of $A$.\footnote{The construction of the hyperbolic knot $K$ in $A$ consists of finding a handle decomposition for $A$ as detailed in \cite[Lemma 5.1]{Myers_Knot_Construction} and replacing each $0$-cell with the complement of the, so called, true lover's tangle.} Repeating this construction we obtain an algorithm that, given $M, k$, outputs a $k$-component hyperbolic link $L$ in $M$. As discussed in Subsection \ref{subsection:hyperbolic} we may compute an upper bound for $\vol(M-L)$ and hence compute some $V>0$ such that any minimal volume hyperbolic $k$-component link in $M$ has volume less than $V$. \\

By Corollary \ref{johansen_algo} we may produce a finite set $\{N_i\}$ of hyperbolic 3-manifolds such that any hyperbolic manifold of volume at most $V$ is obtained by strongly non-exceptional Dehn filling one of the $N_i$. Given a 3-manifold $N$, let $\mathcal{I}(M, k, N)$ denote the set of homeomorphism-classes of hyperbolic 3-manifolds that are a strongly non-exceptional Dehn filling of $N$ and are the complement of a $k$-component link in $M$. Let $\mathcal{I}_{min}(M, k, N)\subset \mathcal{I}(M, k, N)$ denote the set of those 3-manifolds with minimal volume in $\mathcal{I}(M, k, N)$. By the above discussion the complement of a minimal volume $k$-component hyperbolic link in $M$ will be contained in one of the $\mathcal{I}_{min}(M, k, N_i)$. Hence it is sufficient to find an algorithm that, given $M, k$ and a hyperbolic 3-manifold $N$, outputs a finite set containing $\mathcal{I}_{min}(M, k, N)$. \\

Let $m,n$ denote the number of cusps on $M,N$ respectively and prove such an algorithm exists by induction on $n-m-k$. For $n-m-k<0$ the set $\mathcal{I}_{min}(M, k, N)$ is empty. For $n-m-k=0$ we have $\mathcal{I}_{min}(M, k, N)\subset\{N\}$. Thus let us assume $n-m-k=b>0$. Choose an arbitrary $\partial$-framing on $N$. By Corollary \ref{thm:hyperbolic_Dehn_Ancestry} we may find a mono-quadratic system $\mathcal{Q}_{\text{hyp}}$ such that $\gcd(p_1, q_1)=\ldots =\gcd(p_b, q_b)=1$, the Dehn filling of $N$ to $N(p_1/q_1, \ldots, p_b/q_b)$ is strongly non-exceptional, and $N(p_1/q_1, \ldots, p_b/q_b)$ admits a Dehn filling to $M$ if and only if $p_1, q_1, \ldots, p_b, q_b$ is part of a solution to $\mathcal{Q}_{\text{hyp}}$. Therefore, by Theorem \ref{lemma:linearsystemsolution} there is an algorithm that, given $M, k, N$, decides whether $\mathcal{I}(M, k, N)$ is empty and if not outputs an element of $\mathcal{I}(M, k, N)$. \\

If $\mathcal{I}(M, k, N)$ is empty, we are immediately done. Hence we may assume $\mathcal{I}(M, k, N)\neq \emptyset$ and we may compute an upper bound $V_N<\vol(N)$ for the volume of some element in $\mathcal{I}(M, k, N)$. In particular, any element of $\mathcal{I}_{min}(M, k, N)$ has volume at most $V_N$. Observe that by Theorem \ref{thm:volume_change} and Proposition \ref{prop:find_short_slopes} we may find a finite list of slopes $S$ in the boundary tori of $N$ such that any Dehn filling of $N$ avoiding these slopes has volume greater than $V_N$. In particular, any element of $\mathcal{I}_{min}(M, k, N)$ is a strongly non-exceptional Dehn filling of $N$ using a slope in $S$. Any strongly non-exceptional Dehn filling of $N$ using a slope in $S$ is a strongly non-exceptional Dehn filling of $N(s)$ for $s\in S$ and $N(s)$ hyperbolic. However, by induction on $n-m-k$, we can find finite sets containing $\mathcal{I}_{min}(M, k, N(s))$ for every hyperbolic $N(s), s\in S$ and thus a finite set containing $\mathcal{I}_{min}(M, k, N)$.
\end{proof}

Observe that this result outputs a finite list containing all minimal-volume $k$-components hyperbolic links in $M$ up to homeomorphism of the link complement, but possibly also containing other link complements in $M$. To the author's best knowledge there is no known algorithm that, given two hyperbolic 3-manifolds, decides whether they have the same volume. Such an algorithm is exactly what would be required to find the exact list of all minimal volume $k$-component hyperbolic links in $M$ up to homeomorphism of the link complement. \\

\noindent
{\bf Theorem \ref{thm:biglist_volumetheorem}.}
There is an algorithm that, given a compact orientable 3-manifold $M$ with, possibly empty, toroidal boundary, outputs a finite list that includes all minimal volume hyperbolic links in $M$ up to homeomorphisms of the link complement in $M$.

\begin{proof}
As in the proof of Theorem \ref{thm:volumetheorem} we find some hyperbolic link in $M$ and compute an upper bound $V$ for the volume of this hyperbolic link. It is shown in \cite{Adams_Lower_Bound} that a hyperbolic manifold with $b$ boundary components has volume at least $b\cdot 1.01$. Hence any minimal volume hyperbolic link in $M$ has at most $V$ components. Thus the theorem follows by taking the union of finitely many lists from Theorem \ref{thm:volumetheorem}.
\end{proof}

Let us now recall and prove Theorem \ref{thm:compute_volt}: \\

\noindent
{\bf Theorem \ref{thm:compute_volt}.} There is an algorithm that, given $k\in \mathbb Z_{\geq 0}, \varepsilon>0$, and a compact orientable 3-manifold $M$ with, possibly empty, toroidal boundary, outputs $\text{volt}_k(M)$ and $\text{volt}(M)$ within an error of $\varepsilon$.

\begin{proof} To compute $\text{volt}_k(M)$ we find the minimal volume, within an error of $\varepsilon$, of the finite list of hyperbolic manifolds given by Theorem \ref{thm:volumetheorem}. To compute $\text{volt}(M)$ we find the minimal volume, within an error of $\varepsilon$, of the finite list of hyperbolic manifolds given by Theorem \ref{thm:biglist_volumetheorem}.
\end{proof}

\section{Recognising Dehn Ancestry} \label{section:Dehn_Ancestry}

In this section we prove Theorem \ref{prop:mainprop} and Corollary \ref{thm:hyperbolic_Dehn_Ancestry}. Let us recall Theorem \ref{prop:mainprop}: \\

\noindent
{\bf Theorem \ref{prop:mainprop}.} Let $M, N$ be $\partial$-framed manifolds. There is a mono-quadratic system $\mathcal{Q}$ such that $\gcd(p_1, q_1)=\ldots=\gcd(p_{k}, q_{k})=1$ and the Dehn filling $N(p_1/q_1, \ldots, p_{k}/q_{k})$ is orientation-preservingly homeomorphic to $M$ if and only if $p_1, q_1, \ldots, p_{k}, q_{k}$ are part of a solution to $\mathcal{Q}$. Moreover, there is an algorithm that, given $M,N$, outputs $\mathcal{Q}$. 
%

\subsection{Proof Sketch}

The upcoming proof is based on a simple idea, but becomes somewhat involved due to a few ill-behaved special cases. As an aide to the reader we now describe a sketch of how the proof would work in lieu of these special cases, then mention some of the technical challenges and how we plan to address these. The following is a proof sketch whose only purpose is to aide the exposition. A proper proof is given in Subsection \ref{subsection:main_proof_new}. As an additional aide to the reader we consider several larger examples in Subsection \ref{subsection:examples}.\\

To check whether $N(p_1/q_1, \ldots, p_k/q_k)\cong M$, we want to simplify $N,M$ until both are a disjoint union of ``well-behaved'' Seifert fibred spaces and hyperbolic spaces. To achieve this we want to only consider Dehn fillings under which the JSJ-decomposition of $N$ becomes the JSJ-decomposition of $M$ and then cut $N,M$ along their JSJ-decompositions. We can, by induction on $k$, separately consider Dehn fillings using short slopes in hyperbolic JSJ-components of $N$ or fibred slopes in Seifert fibred JSJ-components of $N$. Dehn fillings avoiding short slopes and fibred slopes are known to often preserve the JSJ-decomposition. When we cut $N,M$ along their JSJ-decompositions $\mathcal{T}_N, \mathcal{T}_M$ we want to keep track of how the JSJ-components glue back to form $N,M$. To this end, we will introduce the additional data of a, so called, $\partial$-equation system $\mathcal{F}$ such that $p_1, q_1, \ldots, p_k, q_k$ gives a Dehn filling $(N-\mathcal{T}_N)\to (N-\mathcal{T}_N)(p_1/q_1, \ldots, p_k/q_k)$ satisfying the $\partial$-equation system if and only if $N(p_1/q_1, \ldots, p_k/q_k)\cong M$. Consider a JSJ-torus $T_N\in \mathcal{T}_N$, let $T'_N, T''_N$ denote the two boundary tori of $N-\mathcal{T}_N$ corresponding to $T_N$, and let $\tau_N\colon T'_N\to T''_N$ be the regluing map. Let $T_M\in \mathcal{T}_M$ be the corresponding JSJ-torus of $M$ and let $T'_M, T''_M, \tau_M$ be defined analogously. Then a $\partial$-equation system should ensure that there is a homeomorphism $\phi\colon (N-\mathcal{T}_N)(p_1/q_1, \ldots, p_k/q_k)\to M-\mathcal{T}_M$ such that $\phi\circ \tau_N|_{T'_N}=\tau_M\circ \phi|_{T'_N}$. Given $\partial$-framings on $N,M$ we will see this can be encoded by the equality or conjugacy of certain matrices in $GL_2(\mathbb Z)$. In summary, we try to simplify $N,M$ to be a disjoint union of Seifert fibred and hyperbolic spaces at the cost of introducing the additional data of $\partial$-equation systems. \\

Having simplified $N,M$, by Proposition \ref{prop:hypDehnfilling}, Dehn fillings of hyperbolic components of $N$ using only sufficiently long slopes in $N$ will not be homeomorphic to any component of $M$. Again, by induction on $k$, we separately consider Dehn fillings using short slopes or fibred slopes. Hence, we may assume we do not Dehn fill hyperbolic components of $N$. Moreover, given Seifert data for $N$, the Seifert data for $N(p_1/q_1, \ldots, p_k/q_k)$ has coefficients depending linearly on $p_1, q_1, \ldots, p_k, q_k$. In most cases this Seifert structure must be isomorphic to a fixed choice of Seifert structure on the corresponding component of $M$. Now applying the classification of Seifert fibred spaces in Theorem \ref{thm:SFSclassification}, we see that $N(p_1/q_1, \ldots, p_k/q_k)\cong M$ if and only if $p_1, q_1, \ldots, p_k, q_k$ satisfy some linear system. Moreover, the concept of slope-isomorphisms for Seifert fibred spaces will allow us to additionally address the $\partial$-equation systems. \\

There are several technical issues in the above proof sketch. The bulk of technical definitions are introduced to handle these issues. To motivate the setup in the upcoming section we list several of the definitions and describe the technical issues that motivate them: 
\begin{itemize}
\item[-] {\bf Mono-quadratic systems}: Once we have simplified $N,M$, we rely on the fact that a fixed Seifert structures on components of $M$ are extension of Seifert structures on components of $N$. However, we must separately consider the case of $M$ having infinitely many Seifert structures, only some of which are isomorphic to the extension of the Seifert structure on $N$. Such a setup is described in Example \ref{ex:monoquad_needed}. These ill-behaved special cases are what necessitate the introduction of mono-quadratic systems. \\

\item[-] {\bf $\partial$-equation systems}: In the above discussion we mention introducing, so called, $\partial$-equation systems to ensure that a Dehn filling respects regluing along a JSJ-decomposition. However, the $\partial$-equation systems we introduce in the upcoming subsection are more general. In the above proof-sketch, we only mentioned the case when $N$ has a JSJ-decomposition which remains a JSJ-decomposition after Dehn filling. If $N$ admits a JSJ-decomposition and we Dehn fill avoiding fibred slopes in Seifert fibred JSJ-components and short slopes in hyperbolic JSJ-components, the JSJ-components remain Seifert fibred and hyperbolic. By Proposition \ref{prop:jsjdecompcharacterisation}, in most cases this implies the JSJ-decomposition of $N$ becomes the JSJ-decomposition of $M$. A key exception to this is if a JSJ-component $A$ of $N$ becomes a thickened torus $T^2\times I$ under Dehn filling. As discussed in Example \ref{ex:thickenedtorus}, in this case $A$ has Seifert data $[S^1\times I-\{x_1, \ldots, x_b\}]$ and we Dehn fill along slopes $a_1/1, \ldots, a_b/1$. Consider the $H_1$-basis on $\partial(T^2\times I)$ inherited from the Seifert structure on $[S^1\times I-\{x_1, \ldots, x_b\}]$. Isotopy through the product structure on $T^2\times I$ will relate these $H_1$-basis via the matrix $\begin{pmatrix} 1 & a \\ 0 & -1\end{pmatrix}, a=a_1+\ldots+a_b$. Let $T\subset M$ be the torus obtained from collapsing down the thickened torus Dehn filling of $A$. Then we wish to consider Dehn fillings of $N-A$ to $M-T$ that respect regluing across $A$ and $T$. Such a regluing condition would not be given by equality of fixed matrices in $GL_2(\mathbb Z)$, as discussed above, but is more generally an equation involving matrices of the form $\begin{pmatrix} 1 & a \\ 0 & -1\end{pmatrix}$ with $a$ an integral-affine function. This more general class of peripheral conditions will be called $\partial$-equation systems. \\

\item[-] {\bf $\partial$-parametrisations}: There will be Dehn fillings of $N$, such as exceptional fillings of hyperbolic components, which we consider separately. To do so, we will reduce such ill-behaved Dehn fillings of $N$ to Dehn fillings of a manifold $N'$ of ``lower complexity'' and study such Dehn fillings by induction on the ``complexity'' of $N$. The coefficients of the filling slopes on $N'$ will be integral-affine functions in the coefficients of the filling slopes in $N$. We would like to encode the idea of such a reduction into a proof by induction. To this end we will consider the more general setup where the filling slopes on $N$ are given by pairs $(\rho_1, \sigma_1), \ldots, (\rho_k, \sigma_k)$ of integral-affine functions in variables $\alpha_1, \beta_1, \ldots, \alpha_l, \beta_l, \gamma_1, \ldots, \gamma_m.$ Here one should think of $(\alpha_1, \beta_1), \ldots, (\alpha_l, \beta_l)$ as the filling slopes on some initial ``most-complicated'' manifold and one should think of $\gamma_1, \ldots, \gamma_m$ as just some additional variables. \\

We will keep track of two more pieces of data as part of a $\partial$-parametrisation. First of all, when we reduce to Dehn fillings of a manifold $N'$ of a ``lower complexity'', we might not want to consider all Dehn fillings of $N'$. The Dehn fillings we want to consider will be given by a linear system. Therefore, we record a linear system $\mathcal{L}$ in variables $\alpha_1,\beta_1, \ldots, \alpha_l, \beta_l, \gamma_1, \ldots, \gamma_m$ and only consider Dehn fillings given by variables satisfying this linear system.\\

Eventually, we will complete our reductions and find equations in variables $\rho_1, \sigma_1, \ldots, \rho_k, \sigma_k$ that detect which Dehn fillings of $N$ give $M$. Some of these equations will be quadratic. To ensure these equations are mono-quadratic in variables $\alpha_1, \beta_1, \ldots, \alpha_l, \beta_l, \gamma_1, \ldots, \gamma_m$ we need to ensure that we do not mix variables of distinct quadratic equations. Thus a $\partial$-parametrisation will also come equipped a partition $\mathcal{P}$ that interacts well with the linear system and integral-affine functions. \\

Finally, recall that our desired mono-quadratic system should only have solutions $p_1, q_1, \ldots, p_k, q_k$ with $\gcd(p_1,q_1)=\ldots=\gcd(p_k, q_k)=1$. Hence, we will require that in a $\partial$-equation system, assuming that $\alpha_1, \beta_1, \ldots, \alpha_l, \beta_l, \gamma_1, \ldots, \gamma_m$ satisfy $\mathcal{L}$, the pairs $(\rho_1, \sigma_1, \ldots, \rho_k, \sigma_k)$ are coprime if and only if the pairs $(\alpha_1, \beta_1), \ldots, (\alpha_l, \beta_l)$ are coprime. The data of these integral-affine functions, linear system, and partition satisfying these properties will be called a $\partial$-parameterisation.\\

\item[-] {\bf 6-tuples $(M, N, \mathcal{F}, (\rho_j, \sigma_j), \mathcal{L}, \mathcal{P})$}: In the proof we will repeatedly simplify our manifold. In each step we want to keep track of not just the manifold, but also $\partial$-equation systems and $\partial$-parameterisations. We combine all this data into a 6-tuple of the form $(M, N, \mathcal{F}, (\rho_j, \sigma_j), \mathcal{L}, \mathcal{P})$. Moreover, we want our $\partial$-equation systems and $\partial$-parametrisations to have variables the coefficients of the filling slopes on our initial manifolds before we begin simplifying. Therefore, all our $\partial$-equation systems and $\partial$-parametrisations will be in variables $\alpha_1, \beta_1, \ldots, \alpha_l, \beta_l, \gamma_1, \ldots, \gamma_m$. 
\end{itemize}

Now that we have informally discussed the technical setup let us give formal definitions.

\subsection{Formalism for Recognising Dehn Ancestry}\label{subsection:formalism}

We introduce the following conventions:

\begin{definition} A {\it $\partial$-parameterisation} $(\mathcal{L}, (\rho_j, \sigma_j), \mathcal{P})$ {\it in variables} $\alpha_1, \beta_1, \ldots, \alpha_l, \beta_l, \gamma_1, \ldots, \gamma_m$ is the following: 
\begin{itemize}
\item[-] A linear system $\mathcal{L}$ in variables $\alpha_1, \beta_1, \ldots, \alpha_l, \beta_l, \gamma_1, \ldots, \gamma_m$, 
\item[-] A, possibly empty, set $\{\rho_1, \sigma_1, \ldots, \rho_k, \sigma_k\}$ of integral-affine functions in variables $\alpha_1, \beta_1, \ldots, \alpha_l, \beta_l, \gamma_1, \ldots, \gamma_m$, 
\item[-] A partition $\mathcal{P}$ of the variables $\alpha_1, \beta_1, \ldots, \alpha_l, \beta_l, \gamma_1, \ldots, \gamma_m$.
\end{itemize}
We require that this data satisfies the following conditions: 
\begin{itemize}
\item[-] Each equation of $\mathcal{L}$ and each integral-affine function $\rho_j, \sigma_j$ have variables in at most one part of $\mathcal{P}$. 
\item[-] For $\chi_i\in \{\rho_i, \sigma_i\}, \chi_j\in \{\rho_j, \sigma_j\}, i\neq j$ the two integral-affine functions $\chi_i, \chi_j$ do not have variables in the same part of $\mathcal{P}$, 
\item[-] For each $i$: $\alpha_i$ and $\beta_i$ are in the same part of $\mathcal{P}$, 
\item[-] If $\gcd(a_1, b_1)=\ldots=\gcd(a_l, b_l)=1$ and $a_1, b_1, \ldots, a_l, b_l, c_1, \ldots, c_m$ are part of a solution to $\mathcal{L}$, then $\gcd(\rho_1, \sigma_1)=\ldots =\gcd(\rho_k, \sigma_k)=1$ when evaluated at $\alpha_i=a_i, \beta_i=b_i, \gamma_i=c_i$,
\item[-] If $k>0$, $a_1, b_1, \ldots, a_l, b_l, c_1, \ldots, c_m$ are part of a solution to $\mathcal{L}$, and $\gcd(\rho_1, \sigma_1)=\ldots =\gcd(\rho_k, \sigma_k)=1$ when evaluated at $\alpha_i=a_i, \beta_i=b_i, \gamma_i=c_i$, then $\gcd(a_1, b_1)=\ldots=\gcd(a_l, b_l)=1$, 
\item[-] If $k=0$ and $a_1, b_1, \ldots, a_l, b_l, c_1, \ldots, c_m$ are part of a solution to $\mathcal{L}$, then $\gcd(a_1, b_1)=\ldots=\gcd(a_l, b_l)=1$. \\
\end{itemize}

A $\partial$-parameterisation $(\mathcal{L}, (\rho_j, \sigma_j), \mathcal{P})$, $l$ coprime pairs $a_1, b_1, \ldots, a_l, b_l$, and $m$ integers $c_1, \ldots, c_m$ that are part of solution to $\mathcal{L}$ {\it specify} a Dehn filling $N(p_1/q_1, \ldots, p_k/q_k)$ of a $\partial$-framed manifold $N$ where $p_j, q_j$ are equal to $\rho_j, \sigma_j$ evaluated at $\alpha_i=a_i, \beta_i=b_i, \gamma_i=c_i$.
\end{definition}

\begin{definition} \label{def:framing_pairing} Consider $\partial$-framed manifolds $N,M$, a homeomorphism $\phi:N\to M$, and a boundary torus $T\in \partial M$. Let $\mathcal{M}_{\phi, T}=\mathcal{M}_T\in SL_2(\mathbb Z)$ denote the matrix inducing the map $\phi|_{\phi^{-1}(T)}$ with respect to the $H_1$-basis on $\phi^{-1}(T), T$ given by the $\partial$-framings. A {\it $\partial$-equation} on a $\partial$-framed manifold $M$ in variables $\alpha_1, \beta_1, \ldots, \alpha_l, \beta_l$ is an equation of the form: 
$$\mathcal{M}_{\phi,T}\cdot \left(X_1\cdot \begin{pmatrix} 1 & n_1\\ 0 & 1\end{pmatrix} \cdot X_2 \cdots X_h \cdot \begin{pmatrix} 1 & n_h\\ 0 & 1\end{pmatrix} \cdot X_{h+1}\right)=X'_1\cdot \begin{pmatrix} 1 & n'_1\\ 0 & 1\end{pmatrix}\cdot X'_2 \cdots X'_{h'}\cdot \begin{pmatrix} 1 & n'_{h'}\\ 0 & 1\end{pmatrix} \cdot X'_{h'+1}, \text{ or}$$
$$\mathcal{M}_{\phi,T}\cdot \left(X_1\cdot \begin{pmatrix} 1 & n_1\\ 0 & 1\end{pmatrix} \cdots \begin{pmatrix} 1 & n_h\\ 0 & 1\end{pmatrix} \cdot X_{h+1}\right)=Y\cdot \mathcal{M}_{\phi, T'}\cdot \left(X'_1\cdot \begin{pmatrix} 1 & n'_1\\ 0 & 1\end{pmatrix} \cdots \begin{pmatrix} 1 & n'_{h'}\\ 0 & 1\end{pmatrix} \cdot X'_{h'+1}\right),$$
for $T, T'$ boundary tori of $M$; fixed matrices $X_i, X'_i, Y \in GL_2(\mathbb Z)$; and $n_i, n'_i$ integral-affine functions in variables $\alpha_1, \beta_1, \ldots, \alpha_l, \beta_l$. We say $T$ is {\it involved} in the former $\partial$-equation. Similarly we say $T$ and $T'$ are {\it involved} in the latter $\partial$-equation.


A {\it $\partial$-equation system} $\mathcal{F}$ in variables $\alpha_1, \beta_1, \ldots, \alpha_l, \beta_l$ is a family $\mathcal{F}$ of $\partial$-equations in variables $\alpha_1, \beta_1, \ldots, \alpha_l, \beta_l$ on a, not necessarily connected, manifold $M$ such that: 
\begin{itemize}
\item[-] Every boundary torus is involved in at most one $\partial$-equation; 
\item[-] The boundary tori appearing in a $\partial$-equation are distinct; 
\item[-] No $\partial$-equation involves a boundary torus of a solid torus $D^2\times S^1$; 
\item[-] If a $\partial$-equation involves a boundary torus of a thickened torus $T^2\times I$, then it is of the second type and involves both boundary tori of this $T^2\times I$;
\item[-] If a component of $M$ is not prime then none of its boundary tori are involved in a $\partial$-equation.
\end{itemize}
\end{definition}

\begin{example}\label{ex:partial-system} 
Consider $\partial$-framed manifolds $N, M$ each with a pair of boundary tori $T'_N, T''_N$ and $T'_M, T''_M$ and suppose we fix homeomorphisms $\tau_N:T'_N\to T''_N$ and $\tau_M:T'_M\to T''_M$. Then a homeomorphism $\phi:N\to M$ taking $T'_N$ and $T'_M$ to $T''_N$ and $T''_M$ respectively satisfies $\tau_M\circ \phi|_{T'_N}=\phi\circ \tau_N|_{T'_N}$ if and only if $Y\cdot \mathcal{M}_{\phi, T'_M}=\mathcal{M}_{\phi, T''_M}\cdot X_1$. Observe this is a $\partial$-equation of the second kind above with $h=h'=0$ and $X'_1=\text{id}$.
%
%
%
\end{example}


\begin{definition}\label{def:bound-param}
Consider (not necessarily connected) $\partial$-framed 3-manifolds $M,N$ and the following data: 
\begin{itemize}
\item[-] a $\partial$-parameterisation $(\mathcal{L}, (\rho_j, \sigma_j), \mathcal{P})$ in variables $\alpha_1, \beta_1, \ldots, \alpha_l, \beta_l, \gamma_1, \ldots, \gamma_m$,
\item[-] a $\partial$-equation system $\mathcal{F}$ on $M$ in variables $\alpha_1, \beta_1, \ldots, \alpha_l, \beta_l, \gamma_1, \ldots, \gamma_m$,
\item[-] integers $a_1,b_1, \ldots, a_l, b_l, c_1, \ldots, c_m$, 
\item[-] the values $p_j, q_j$ of $\rho_j, \sigma_j$ evaluated at $\alpha_i=a_i, \beta_i=b_i, \gamma_i=c_i$.
\end{itemize}
We say that $a_1, b_1 \ldots, a_l, b_l, c_1, \ldots, c_m$ give a {\it $\partial$-framed filling} of $N$ to $(M, \mathcal{F})$ via $(\mathcal{L}, (\rho_j, \sigma_j), \mathcal{P})$ if: 
\begin{itemize}
\item[-] $\gcd(a_1, b_1)=\ldots =\gcd(a_l, b_l)=1$,
\item[-] $a_1, b_1, \ldots, a_l, b_l, c_1, \ldots, c_m$ are part of a solution to $\mathcal{L}$, 
\item[-] each $\partial$-equation of $\mathcal{F}$ has variables in at most one part of $\mathcal{P}$, distinct $\partial$-equations have variables in distinct parts of $\mathcal{P}$, and no $\rho_j$ or $\sigma_j$ has variables in the same part of $\mathcal{P}$ as a $\partial$-equation,
\item[-] the Dehn filling $N(p_1/q_1, \ldots, p_k/q_k)$, with $\partial$-framing inherited from $N$, admits a homeomorphism $\phi$ to $M$ such that: 
\begin{itemize}
\item[-] $\phi$ respects the ordering of boundary tori,
\item[-] all equations of $\mathcal{F}$ with respect to $\phi$ are true when evaluated at $\alpha_i=a_i, \beta_i=b_i, \gamma_i=c_i$.
\end{itemize}
\end{itemize}
\end{definition}

\begin{definition}
With the same convention as Definition \ref{def:bound-param}, we say a Diophantine system $\mathcal{D}$ {\it detects} $a_1, b_1, \ldots, a_l, b_l$ giving a $\partial$-framed filling of $N$ to $(M, \mathcal{F})$ via $(\mathcal{L}, (\rho_j, \sigma_j), \mathcal{P})$ if: 
\begin{itemize}
 \item[-] $a_1, b_1, \ldots, a_l, b_l$ are part of a solution to $\mathcal{D}$; 
 \item[-] If $a'_1, b'_1, \ldots, a'_l, b'_l, c'_1, \ldots, c'_m$ are also part of a solution to $\mathcal{D}$, then they give a $\partial$-framed filling of $N$ to $(M, \mathcal{F})$ via $(\mathcal{L}, (\rho_j, \sigma_j), \mathcal{P})$. 
\end{itemize}
\end{definition}

\begin{definition}
We say the 6-tuple $(M', N', \mathcal{F}', (\rho'_j, \sigma'_j), \mathcal{L}', \mathcal{P}')$ {\it detects} $a_1, b_1, \ldots, a_l, b_l$ giving a $\partial$-framed filling of $N$ to $(M, \mathcal{F})$ via $(\mathcal{L}, (\rho_j, \sigma_j), \mathcal{P})$ if: 
\begin{itemize}
	\item[-] $M', N'$ are $\partial$-framed 3-manifolds; $\mathcal{F}'$ is a $\partial$-equation system on $M'$ in variables $\alpha_1, \beta_1, \ldots, \alpha_l, \beta_l$, $\gamma'_1, \ldots, \gamma'_{m'}$; and $(\mathcal{L}', (\rho'_j, \sigma'_j), \mathcal{P}')$ is a $\partial$-parameterisation of $N'$ in variables $\alpha_1, \beta_1, \ldots, \alpha_l, \beta_l, \gamma'_1, \ldots, \gamma'_{m'}$;
	\item[-] there exist integers $c'_1, \ldots, c'_{m'}$ such that $a_1, b_1, \ldots, a_l, b_l, c'_1, \ldots, c'_{m'}$ give a $\partial$-framed filling of $N'$ to $(M', \mathcal{F}')$ via $(\mathcal{L}', (\rho'_j, \sigma'_j), \mathcal{P}')$; 
	\item[-] If integers $a'_1, b'_1, \ldots, a'_l, b'_l, c'_1, \ldots, c'_{m'}$ also give a $\partial$-framed filling of $N'$ to $(M', \mathcal{F}')$ via $(\mathcal{L}', (\rho'_j, \sigma'_j), \mathcal{P}')$, then exists integers $c_1, \ldots, c_m$ such that $a'_1, b'_1, \ldots, a'_l, b'_l, c_1, \ldots, c_m$ give a $\partial$-framed filling of $N$ to $(M, \mathcal{F})$ via $(\mathcal{L}, (\rho_j, \sigma_j), \mathcal{P})$. 
\end{itemize}
\end{definition}

Now that we have introduced the above setup, let us state our main technical result.

\begin{proposition}\label{prop:Rephrased_main_prop}
Consider $\partial$-framed (not necessarily connected) 3-manifolds $N, M$, a $\partial$-parameterisation $(\mathcal{L}, (\rho_j, \sigma_j), \mathcal{P})$, and a $\partial$-equation system $\mathcal{F}$ on $M$ in variables $\alpha_1,\beta_1, \ldots, \alpha_l, \beta_l, \gamma_1, \ldots, \gamma_m$. There is a mono-quadratic system $\mathcal{Q}$ which detects every $a_1,b_1, \ldots, a_l, b_l$ giving a $\partial$-framed filling of $N$ to $(M, \mathcal{F})$ via $(\mathcal{L}, (\rho_j, \sigma_j), \mathcal{P})$. Moreover, there is an algorithm that, given $M, N, \mathcal{F}, \mathcal{L}, (\rho_j, \sigma_j), \mathcal{P}$, outputs $\mathcal{Q}$. 
\end{proposition}

Observe that Theorem \ref{prop:mainprop} for $M$ follows directly from Proposition \ref{prop:Rephrased_main_prop} by setting $\mathcal{L},\mathcal{F}=\emptyset, k=l, m=0, \rho_1=\alpha_1, \sigma_1=\beta_1, \ldots, \rho_k=\alpha_l, \sigma_k=\beta_l$, and $\mathcal{P}=\{\{\alpha_i, \beta_i\} \mid i=1, \ldots, l\}$.\\

 

In the upcoming proof, we want to assume that the JSJ-decomposition of $N$ is preserved under Dehn filling. Lemma \ref{lemma:technicallemma} will allow us to do so or else detect the coefficients of the Dehn filling by a 6-tuple of lower ``complexity''. Let us first specify what we mean by the JSJ-decomposition being preserved under Dehn filling.

\begin{definition}
A Dehn filling of manifold $N$ admitting a JSJ-decomposition $\mathcal{T}_N$ to another manifold $M$ is not {\it JSJ-exceptional} if:
\begin{enumerate}
 \item[-] The image of $\mathcal{T}_N$ is a JSJ-decomposition of $M$; 
 \item[-] Hyperbolic JSJ-components of $N$ become hyperbolic JSJ-components of $M$; 
 \item[-] Any Seifert structure on a JSJ-components of $N$ extends to a Seifert structure on the resulting JSJ-components of $M$. 
\end{enumerate}
\end{definition}

\begin{lemma}\label{lemma:technicallemma}
Consider a $\partial$-framed 3-manifold $N$ with boundary, no component of which is $T^2\times I$ or $D^2\times S^1$, and which admits a JSJ-decomposition $\mathcal{T}_N$. Consider also a $\partial$-framed 3-manifold $M$, a $\partial$-equation system $\mathcal{F}$ on $M$, and a $\partial$-parameterisation $(\mathcal{L}, (\rho_j, \sigma_j), \mathcal{P})$ in variables $\alpha_1,\beta_1, \ldots, \alpha_l, \beta_l, \gamma_1, \ldots, \gamma_m$. There exists a finite list of 6-tuples $\mathcal{N}_{red}:=(N_i, M_i, \mathcal{F}_i, (\rho_j^{(i)}, \sigma_j^{(i)}), \mathcal{L}_i, \mathcal{P}_i)_{i\in I}$ such that:
\begin{itemize}
 \item[-] for each $i\in I$: $M_i, N_i$ are $\partial$-framed manifolds, $\mathcal{F}_i$ is a $\partial$-equation system on $M_i$, and each $(\mathcal{L}_i, (\rho_j^{(i)}, \sigma_j^{(i)}), \mathcal{P}_i)$ is a $\partial$-parameterisation in variables $\alpha_1,\beta_1, \ldots, \alpha_l, \beta_{l}, \gamma^{(i)}_1, \ldots, \gamma^{(i)}_{m_i}$; 
 \item[-] If $a_1, b_1, \ldots, a_l, b_l, c_1, \ldots, c_m$ give a JSJ-exceptional $\partial$-framed filling of $N$ to $(M, \mathcal{F})$ via $(\mathcal{L}, (\rho_j, \sigma_j), \mathcal{P})$, then $a_1, b_1, \ldots, a_l, b_l$ are detected by some 6-tuple of $\mathcal{N}_{red}$; 
 \item[-] If for some $i\in I$ some integers $a^{(i)}_1, b^{(i)}_1, \ldots, a^{(i)}_l, b^{(i)}_l, c^{(i)}_1, \ldots, c^{(i)}_{m_i}$ give a $\partial$-framed filling of $N_i$ to $(M_i, \mathcal{F}_i)$ via $(\mathcal{L}_i, (\rho^{(i)}_j, \sigma^{(i)}_j), \mathcal{P}_i)$, then exists integers $c_1, \ldots, c_m$ such that $a^{(i)}_1, b^{(i)}_1, \ldots, a^{(i)}_l, b^{(i)}_l, c_1, \ldots, c_m$ give a $\partial$-framed filling of $N$ to $(M, \mathcal{F})$ via $(\mathcal{L}, (\rho_j, \sigma_j), \mathcal{P})$,
 \item[-] Each $N_i$ has at most as many boundary tori as $N$, 
 \item[-] Each $N_i$ has fewer boundary tori than $N$ or has a JSJ-decomposition with fewer JSJ-tori than $N$,
 \item[-] There is an algorithm that, given $N, M, \mathcal{F}, \mathcal{L}, (\rho_j, \sigma_j), \mathcal{P}$, outputs such $\mathcal{N}_{red}.$
\end{itemize}
\end{lemma}

We now prove Proposition \ref{prop:Rephrased_main_prop} assuming Lemma \ref{lemma:technicallemma} and leave the proof of Lemma \ref{lemma:technicallemma} to Subsection \ref{section:Lemma_proof}. However, first let us consider some examples that illustrate the setup and parts of the proof.

\subsection{Examples of Recognising Dehn Ancestry}\label{subsection:examples}
The following two examples serve to illustrate what the data in a 6-tuple $(N, M, \mathcal{F}, (\rho_j,\sigma_j), \mathcal{L}, \mathcal{P})$ might look like, but also to observe some of the main phenomena that will appear in the upcoming proofs. 

\begin{example} Consider the following manifolds: 
$$N:=N_1 \sqcup N_2\sqcup N_3, \text{ for }N_1:=T^2\times I, N_2:=S^3-K_{9\_32}, N_3:=S^3-K_{9\_33},$$
$$M:=M_1 \sqcup M_2\sqcup M_3, \text{ for }N_1:=S^3, M_2:=S^3-K_{9\_32}, M_3:=S^3-K_{9\_33}.$$
Here $K_{9\_32}, K_{9\_32}$ are nine-crossing knots of the knot census. Consider the following data: 
\begin{itemize}
\item[-] The $\partial$-framing on $N$ is given as follows: Let $T_0^N:=T^2\times \{0\}, T_1^N:=T^2\times \{1\}\subset \partial N_1, T_2^N:=\partial N_2$, and $T_3^N:=\partial N_3$ denote the boundary tori of $N$. Order the boundary tori as $T_0^N, T_1^N, T_2^N, T_3^N$. Choose arbitrary $H_1$-basis on $T_2^N, T_3^N$. Choose $H_1$-basis on $T_0^N, T_1^N$ given by a Seifert structure on $T^2\times I$; 
\item[-] The $\partial$-framing on $M$ is given as follows: Let $T_2^M:=T_2^N, T_3^M:=T_3^N$ denote the boundary tori of $M$. Order the boundary tori as $T_2^M, T_3^M$. Choose the same $H_1$-basis on $T_2^M, T_3^M$ as we chose on $T_2^N, T_3^N$; 
\item[-] Consider the following $\partial$-parametrisation $(\mathcal{L}, (\rho_j, \sigma_j), \mathcal{P})$ in variables $\alpha_1, \beta_1, \ldots, \alpha_4, \beta_4, \gamma_1$: 
\begin{itemize}
\item[-] $\mathcal{L}$ consists of the equations $\{\alpha_1-\beta_2=1, \beta_1=1, \alpha_2=1, \alpha_3=2, \beta_3\equiv 1\pmod 2\}$; 
\item[-] $(\rho_j, \sigma_j)$ consists of the integral-affine functions $\rho_1=\alpha_1+3\beta_2, \sigma_1=2\alpha_1+7\beta_2, \rho_2=\alpha_4, \sigma_2=\beta_4$; 
\item[-] $\mathcal{P}$ is the following partition $\{\{\alpha_1, \beta_1, \alpha_2, \beta_2\}, \{\alpha_3, \beta_3, \gamma_1\}, \{\alpha_4, \beta_4\}\}$; 
\end{itemize}
\item[-] $\mathcal{F}$ consists of the single $\partial$-equation $\mathcal{M}_{T_2^M}\cdot \begin{pmatrix}1 & 7\alpha_3+\beta_3-11\gamma_1+2\\ 0 & 1\end{pmatrix}=\mathcal{M}_{T_3^M}.$
\end{itemize}
We can check that this data satisfies all the condition of Proposition \ref{prop:Rephrased_main_prop}. In this special case, Proposition \ref{prop:Rephrased_main_prop} states that there is a mono-quadratic system $\mathcal{Q}$ such that values $a_1, b_1, \ldots, a_4, b_4, c_1$ of $\alpha_1, \beta_1, \ldots, \alpha_4, \beta_4, \gamma_1$ are part of a solution to $\mathcal{Q}$ if and only if the following holds: 
\begin{itemize}
\item[-] $\gcd(a_1, b_1)=\ldots =\gcd(a_4, b_4)=1$; 
\item[-] $a_1-b_2=1, b_1=1, a_2=1, a_3=2$, and $b_3\equiv 1 \pmod 2$; 
\item[-] There is a homeomorphism respecting $\partial$-framings: $$\phi:N_1\left(\frac{a_1+3b_2}{2a_1+7b_2}, \frac{a_4}{b_4}\right) \sqcup N_2\sqcup N_2\to M;$$ 
\item[-] Moreover if for the above $\phi$, the induced boundary maps $\phi|_{T_2^N}:T_2^N\to T_2^M, \phi|_{T_3^N}:T_3^N\to T_3^M$ are given by the matrix $Z_2, Z_3$ with respect to the chosen $H_1$-basis, then $$Z_2\cdot \begin{pmatrix}1 & 7a_3+b_3-11c_1+2\\ 0 & 1\end{pmatrix}=Z_3.$$
\end{itemize}
Let us find such $\mathcal{Q}$ in this example. As can be verified by SnapPy \cite{SnapPy} the mapping class groups of both $N_2$ and $N_3$ are trivial. Hence for any such homeomorphism $\phi$ we know that $Z_2=Z_3=1$ and thus the $\partial$-equation becomes $7a_3+b_3-11c_1+2=0$. Also observe that if $a_1, b_1, \ldots, a_4, b_4, c_1$ are part of a solution to $\mathcal{L}$, then $\gcd(a_1, b_1)=\ldots=\gcd(a_3, b_3)=1$. In summary we want $\mathcal{Q}$ such that $a_1, b_1, \ldots, a_4, b_4, c_1$ is part of a solution to $\mathcal{Q}$ if and only if we have: 
\begin{itemize}
\item[-] $\gcd(a_4, b_4)=1$; 
\item[-] $a_1-b_2=1, b_1=1, a_2=1, a_3=2, b_3\equiv 1\pmod 2, 7a_3+b_3-11c_1+2=0$; 
\item[-] $N_1\left(\frac{a_1+3b_2}{2a_1+7b_2}, \frac{a_4}{b_4}\right)\cong S^3$. 
\end{itemize}
By construction $N_1\left(\frac{a_1+3b_2}{2a_1+7b_2}, \frac{a_4}{b_4}\right)$ has a Seifert structure with data $\left[S^2, \frac{a_1+3b_2}{2a_1+7b_2}, \frac{a_4}{b_4}\right]$. By Lemma \ref{lemma:lensspacecharacterisation} this is Seifert data for $S^3$ if and only if $b_4\cdot (a_1+3b_2)+a_4\cdot (2a_1+7b_2)=\pm 1$. Hence our desired mono-quadratic system is: 
$$\{\alpha_1-\beta_2=1\} \wedge \{\beta_1=1\}\wedge \{\alpha_2=1\}\wedge \{\alpha_3=2\}\wedge \{\beta_3\equiv 1\text{ (mod } 2)\}\wedge \{7\alpha_3+\beta_3-11\gamma_1+2=1\}\wedge$$ $$\{\beta_4(\alpha_1+3\beta_2)+\alpha_4(2\alpha_1+7\beta_2)=\pm 1\}.$$
%
%
%
%
%
%
%
\end{example}

\begin{example} We will now construct manifolds $N_1, N_2, N_3, N_4, M_1, M_2$ and homeomorphisms $\psi_1, \psi_2, \psi_3, \theta$ between some of their boundary tori. Having done so we will consider the following manifolds obtained by gluing the $N_i, M_i$ along the homeomorphisms: $$N:=N_1\cup_{\psi_1}N_2\cup_{\psi_2} N_3\cup_{\psi_3}N_4 \text{ and } M:=M_1\cup_\theta M_2.$$
Let $\Sigma_{g,n}$ denote the compact surface of genus $g$ with $n$ boundary components. Define $N_1, N_2, N_3, N_4$ as: 
$$N_1:=[\Sigma_{2,1}, 5/8];~ N_2:=[\Sigma_{0,3}]=\Sigma_{0, 3}\times S^1; ~N_3:=[\Sigma_{2,2}, -1/4]; ~N_4:=S^3-L,$$
for $L$ the link shown in Figure \ref{fig:link}. 
\vskip 0.2cm
{\centering
\includegraphics[width=8cm]{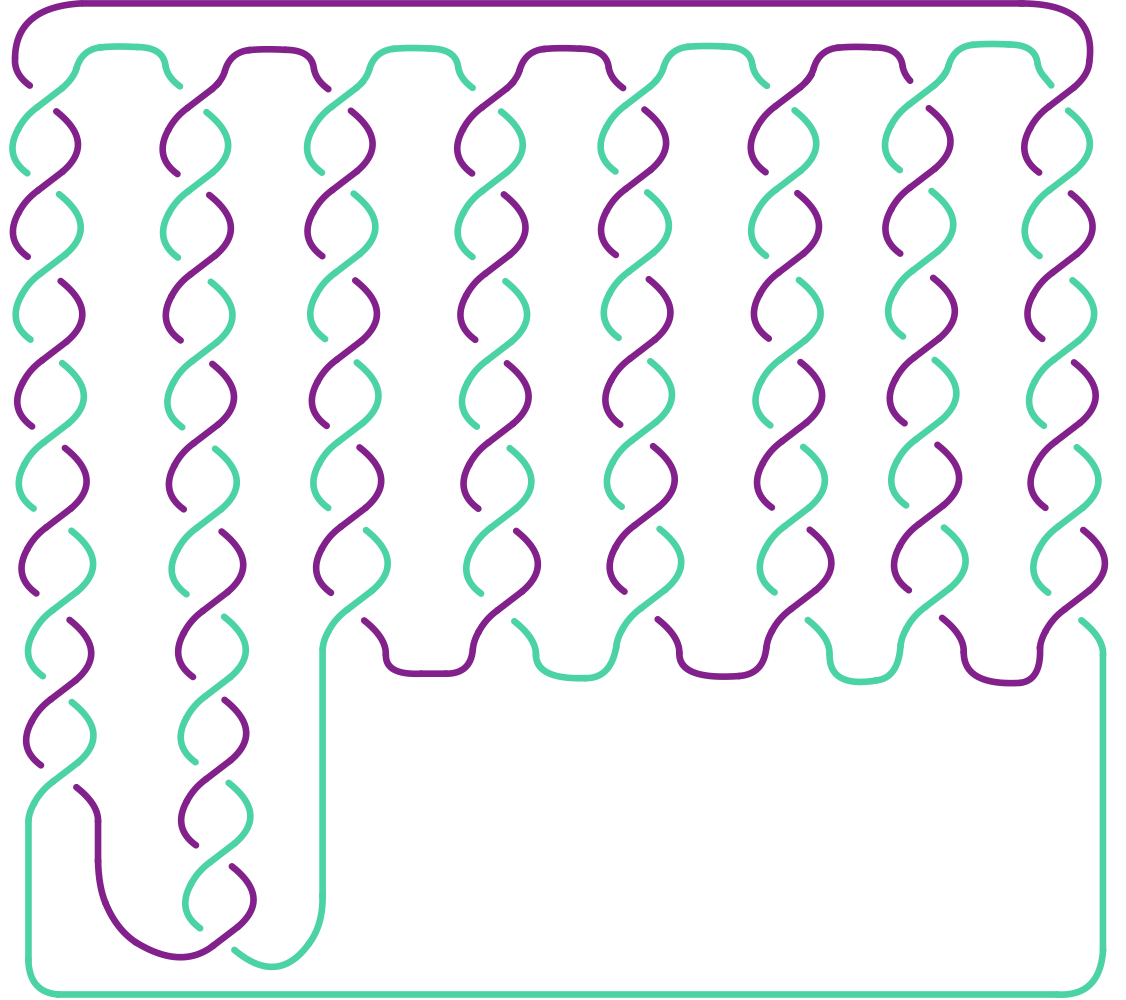}
\captionof{figure}{Link $L$ in $S^3$ with first and second components coloured green and purple respectively.}
\label{fig:link}
}
\vskip 0.2cm
Let us denote the boundary torus of $N_1$ by $T_1^N$, the boundary tori of $N_2$ by $T_2^N, T_3^N, T_4^N$, and the boundary tori of $N_3$ by $T_5^N, T_6^N$. Finally let $T_7^N$ and $T_8^N$ be the first and second boundary torus of $N_4$, with ordering as given by the SnapPy. Equip the boundary tori of $N_1, N_2, N_3$ with the $H_1$-basis given by the above Seifert structures. Equip the boundary tori of $N_4$ with the $H_1$-basis given by meridian and longitude in $S^3$. With respect to these $H_1$-basis we define homeomorphisms: 
$$\psi_1:T_1^N\to T_2^N, \psi_2:T_4^N\to T_5^N, \psi_3:T_6^N\to T_7^N \text{ respectively given by } \begin{pmatrix} -9 & -22 \\ -7 & -17\end{pmatrix}, \begin{pmatrix} 3 & 1 \\ 7 & 2\end{pmatrix}, \begin{pmatrix} 1 & 1 \\ 3 & 2\end{pmatrix}.$$
We now define $N$, as mentioned above, as the gluing of $N_1, N_2, N_3, N_4$ by $\psi_1, \psi_2, \psi_3$. We take the $\partial$-framing on $N$ with $H_1$-basis as discussed above and the boundary tori ordered as $T_3^N$ first and $T_5^N$ second. \\
\vskip 0.2cm
{\centering
\includegraphics[width=8cm]{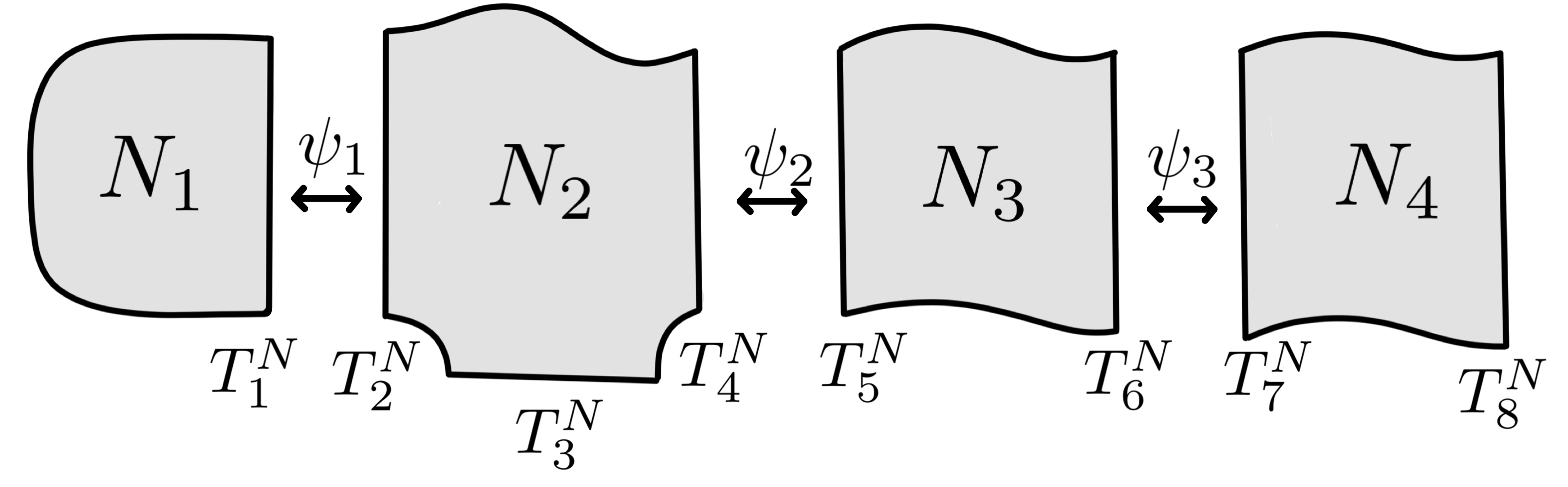}
\captionof{figure}{A diagrammatic representation of $N_1, N_2, N_3, N_4$, their boundary tori, and gluing maps between them.}
\label{fig_diagram_N}
}
\vskip 0.2cm
Similarly let us construct $M$. Define $M_1, M_2$ as: 
$$M_1:=[\Sigma_{4,1}, -3/4, 1/8], ~ M_2:=N_4(\cdot, 1/1).$$
Let $T_1^M:=\partial M_1$ and $T_2^M:=\partial M_2$. Equip $T_1^M$ with the $H_1$-basis given by the above Seifert structure and equip $T_2^M$ with the $H_1$-basis given by SnapPy. With respect to these $H_1$-basis let $\theta:T_1^M\to T_2^M$ be the homeomorphism given by the matrix $\begin{pmatrix} 1 & 1 \\ 3 & 2\end{pmatrix}$. Then, as above, define $M$ to be the gluing of $M_1, M_2$ along $\theta$. \\
We need to choose some $\partial$-equation system $\mathcal{F}$ and some $\partial$-parametrisation $(\mathcal{L}, (\rho_j, \sigma_j), \mathcal{P})$ in variables $\alpha_1, \beta_1, \alpha_2, \beta_2$. In this example we make the simplest choice, namely: 
\begin{itemize}
\item[-] Let $\mathcal{F}=\emptyset, \mathcal{L}=\emptyset$; 
\item[-] Let $(\rho_j, \sigma_j)$ consist of the functions $\rho_1:=\alpha_1, \sigma_1:=\beta_1, \rho_2:=\alpha_2, \sigma_2:=\beta_2$; 
\item[-] Let $\mathcal{P}$ be the partition $\{\{\alpha_1, \beta_1\}, \{\alpha_2, \beta_2\}\}$. \\
\end{itemize}

In this special case Proposition \ref{prop:Rephrased_main_prop} states that there is a mono-quadratic system $\mathcal{Q}$ such that values $a_1, b_1, a_2, b_2$ of $\alpha_1, \beta_1, \alpha_2, \beta_2$ are part of a solution to $\mathcal{Q}$ if and only if $\gcd(a_1, b_1)=\gcd(a_2, b_2)=1$ and $N(a_1/b_1, a_2/b_2)\cong M$. Let us find $\mathcal{Q}$ by hand in this example. \\

Before we begin, let us mention that we choose $N_4$ to be particularly well-behaved under Dehn fillings. As we see in Figure \ref{fig:link}, the link $L$ has an alternating, prime, and twist-reduced diagram such that every twist region has at least 7 crossings and each component of $L$ passes through 8 twist regions. For an introduction to this notation see \cite{Purcell_Futer_No_Exceptional}. It is shown in \cite[Theorem 1.7]{Purcell_Futer_No_Exceptional} that therefore any non-meridian Dehn filling of $S^3-L$ is hyperbolic. In particular, $N_4(\cdot, a_2/b_2)$ with $b_2\neq 0$ is hyperbolic. We can use SnapPy to check that $M_2=N_4(\cdot, 1/1)$ has trivial isometry group. \\

{\bf Case 1.} $a_1=\pm 1, b_1=0$, that is we fill along the fibred slopes in $N_2$: Hence we want to find $a_2, b_2$ such that $\gcd(a_2, b_2)=1$ and $N(1/0, a_2/b_2)\cong M$.\\

Aside on notation: In the previous subsection we introduced notation to describe that we have simplified our manifolds $N,M$. As an aside let us describe what this notation looks like in case 1. All $a_1, b_1, a_2, b_2$ with $a_1=\pm 1, b_1=0$ giving a $\partial$-framed filling of $N$ to $(M, \mathcal{F})$ via $(\mathcal{L}, (\rho_j, \sigma_j), \mathcal{P})$ are detected by the 6-tuple $(M', N', \mathcal{F}', \mathcal{L}', (\rho'_j, \sigma'_j), \mathcal{P}')$ with: 
\begin{itemize}
\item[-] $M':=M, N':=N(1/0, \cdot)$ with $\partial$-framings inherited from $N,M$; 
\item[-] $\mathcal{F}':=\emptyset$ and $\mathcal{P}':=\mathcal{P}$; 
\item[-] $\mathcal{L}:=\{\alpha_1=\pm 1\}\wedge \{\beta_1=0\}$; 
\item[-] The functions $(\rho'_j, \sigma'_j)$ are the pair $\rho_2, \sigma_2$. 
\end{itemize}
This notation might seem like an absurd way of saying ``we want to find $a_2, b_2$ such that $\gcd(a_2, b_2)=1$ and $N(0/1, a_2/b_2)\cong M$''. In a more complicated case, where we keep track of more complicated partitions and linear systems, this notation is necessary. Let us now return to Case 1. \\

Observe that the Dehn filling of $N_2=(\Sigma_{0,3}\times S^1)(\cdot, 1/0, \cdot)$ along the fibred slopes in $T_3^N$ is homeomorphic to complement of the two-component unlink in $S^3$, also known as $(D^2\times S^1)\#(D^2\times S^1)$. Moreover, under this homeomorphisms the fibred slopes $\lambda_2\in T_2^N, \lambda_4\in T_4^N$ become the curves $\partial D^2\times \{1\}\subset \partial (D^2\times S^1)$ in each prime-component. Hence $N_1\cup_{\psi_1}N_2(\cdot, 1/0, \cdot )\cup_{\psi_2} N_3$ is homeomorphic to the connect sum of a Dehn filling of $N_1$ and a Dehn filling of $N_3$. The Dehn filling of $N_1$ will be along the slope $\psi_1^{-1}(\lambda_2)$ and the Dehn filling of $N_2$ will be along the slope $\psi_2(\lambda_4)$. In terms of the $H_1$-basis on $T_1^N$, $\psi_1^{-1}(\lambda_2)$ has coordinates $\begin{pmatrix} -9 & -22 \\ -7 & -17\end{pmatrix}^{-1}\cdot \dbinom{1}{0}=\dbinom{17}{-7}$. In terms of the $H_1$-basis on $T_5^N$, $\psi_2(\lambda_4)$ has coordinates $\begin{pmatrix} 3 & 1 \\ 7 & 2\end{pmatrix}\cdot \dbinom{1}{0}=\dbinom{3}{7}$. In summary: $$N_1\cup_{\psi_1}N_2(\cdot, 1/0, \cdot )\cup_{\psi_2} N_3\cong [\Sigma_{2,0}, 5/8, -17/7]\#[\Sigma_{2,1}, -1/4, 3/7].$$
Observe that $[\Sigma_{2,0}, 5/8, -17/7]\not\cong S^3$ and $[\Sigma_{2,1}, -1/4, 3/7]\not\cong D^2\times S^1.$ If $b_2\neq 0$ then recall that $N_4(\cdot, a_2/b_2)$ is hyperbolic. In particular, $[\Sigma_{2,1}, -1/4, 3/7]\cup_{\psi_3}N_4(\cdot, a_2/b_2)$ has non-trivial JSJ-decomposition and in particular is not $S^3$. On the other hand, we see that $N_4(\cdot, 1/0)$ is the complement in $S^3$ of one of the components of the link $L$. Each component of $L$ is individually unknotted and therefore $N_4(\cdot, 1/0)\cong D^2\times S^1$ with contractible slope $\mu$ having coordinates $(0,1)$ with respect to the $H_1$-basis on $T_7^N$. Hence $\psi_3^{-1}(\mu)$ has coordinates $\begin{pmatrix} 1 & 1 \\ 3 & 2\end{pmatrix}^{-1}\cdot \dbinom{0}{1}$ with respect to the $H_1$-basis on $T_6^N$ and hence $[\Sigma_{2,1}, -1/4, 3/7]\cup_{\psi_3}N_4(\cdot, a_2/b_2)\cong [\Sigma_{2,0}, -1/4, 3/7, -1/1]$. Again  $[\Sigma_{2,1}, -1/4, 3/7]\cup_{\psi_3}N_4(\cdot, a_2/b_2)\not\cong S^3$. It follows that $N(1/0, a_2/b_2)$ is not prime and in particular not homeomorphic to $M$. Hence there are no desired Dehn fillings in Case 1. \\

{\bf Case 2.} $b_2=0$, that is we fill along the meridian in $N_4$: By the previous case, we may assume $b_1\neq 0$. As discussed in the previous case, $N_4(\cdot, 1/0)\cong D^2\times S^1$ and $N_3\cup_{\psi_3}N_4(\cdot, 1/0)$ is Seifert fibred. Hence $N(a_1/b_1, 1/0)$ is a union of Seifert fibred spaces, none of which is a solid torus. Hence, possibly combining some of these to a larger Seifert fibred space, $N(a_1/b_1, 1/0)$ has JSJ-decomposition consisting entirely of Seifert fibred spaces. In particular, $N(a_1/b_1, 1/0)\not\cong M$ and there are no desired Dehn fillings in Case 2. \\

{\bf Case 3.} $N_2$ Dehn fills to become $T^2\times I$, that is $b_1=\varepsilon\in \{\pm 1\}$: We may assume $b_2\neq 0$. After such a Dehn filling and collapsing down the resulting $T^2\times I$ onto a single torus, we get an identifiction of the tori $T_1^N, T_2^N, T_4^N,$ and $T_5^N$. This gives a change of $H_1$-basis from the chosen $H_1$-basis on $T_1^N$ and $T_5^N$. As discussed in Example \ref{ex:thickenedtorus}, this change of basis is given by the matrix $\begin{pmatrix} -9 & -22 \\ -7 & -17\end{pmatrix} \cdot \begin{pmatrix}1 & \varepsilon\cdot a_1 \\ 0 & -1\end{pmatrix}\cdot \begin{pmatrix} 3 & 1 \\ 7 & 2\end{pmatrix}$. \\

{\bf Case 3.1.} After collapsing $T^2\times I$ the fibred slopes $\lambda_1\in T_1^N$ and $\lambda_4\in T_4^N$ are distinct: Recall that $N_4(\cdot, a_2/b_2)$ is hyperbolic. So $N(a_1/\varepsilon, a_2/b_2)$ consists of two Seifert fibred spaces and a hyperbolic manifold glued along their boundaries such that fibred slopes in adjacent Seifert fibred spaces don't line up. Thus, by Theorem \ref{prop:jsjdecompcharacterisation}, $N(a_1/\varepsilon, a_2/b_2)$ has three JSJ-components and isn't homeomorphic to $M$. Hence there are no desired Dehn fillings in Case 3.1. \\

{\bf Case 3.2.} After collapsing $T^2\times I$, the fibred slopes $\lambda_1\in T_1^N$ and $\lambda_4\in T_4^N$ are equal up to swapping orientation: This is equivalent to the matrix $\begin{pmatrix} -9 & -22 \\ -7 & -17\end{pmatrix}\cdot \begin{pmatrix}1 & \varepsilon\cdot a_1 \\ 0 & -1\end{pmatrix}\cdot \begin{pmatrix} 3 & 1 \\ 7 & 2\end{pmatrix}$ taking $\dbinom{1}{0}$ to $\dbinom{\delta}{0}$ for $\delta \in \{\pm 1\}$, which in turn is equivalent to $\begin{pmatrix} -9 & -22 \\ -7 & -17\end{pmatrix}\cdot \begin{pmatrix}1 & \varepsilon\cdot a_1 \\ 0 & -1\end{pmatrix}\cdot \begin{pmatrix} 3 & 1 \\ 7 & 2\end{pmatrix}=\begin{pmatrix} \delta & c_1 \\ 0 & -\delta \end{pmatrix}$ for some $c_1\in \mathbb Z$. Multiplying out the matrices, or more generally by Lemma \ref{lemma:sl2z}, for each $\varepsilon, \delta \in\{\pm 1\}$ there is a linear-system $\mathcal{L}_{\varepsilon, \delta}$ in variables $\alpha_1, \gamma_1$ such that $a_1, c_1$ are part of a solution to $\mathcal{L}_{\varepsilon, \delta}$ if and only if the above matrix equations holds. By Example \ref{ex:gluingsfs}, if $a_1, c_1$ are part of a solution to $\mathcal{L}_\varepsilon, \delta$ we obtain the following Seifert structure for $$N_1\cup_{\psi_1}N_2(\cdot, a_1/\varepsilon, \cdot)\cup_{\psi_2}N_3=[\Sigma_{4,1}, (-\delta\cdot c_1)/1, 5/8, -1/4].$$
Recall that $N_4(\cdot, a_2/b_2)$ is hyperbolic. Hence a homeomorphism $\phi:N(a_1/\varepsilon, a_2/b_2)\to M$ must take $[\Sigma_{4,1}, (-\delta\cdot c_1)/1, a/b, c/d]$ to $M_1$ and $N_4(\cdot, a_2/b_2)$ to $M_2$. Moreover, observe that the mapping class group of $M_2$ acts trivially on $\partial M_2$ and $\psi_3, \theta$ are given by the same matrix. Hence such a homeomorphism $\phi$ must respect the $H_1$-basis given by Seifert structures on the boundaries of $[\Sigma_{4,1}, (-\delta\cdot c_1)/1, 5/8, -1/4]$ and $M_1=[\Sigma_{4,1}, 1/8, -3/4]$. In summary, in this case we want to find $a_1, b_1, a_2, b_2, c_1\in \mathbb Z$ such that for some $\varepsilon, \delta\in \{\pm 1\}$: 
\begin{itemize}
\item[-] $b_1=\varepsilon$, $\gcd(a_2, b_2)=1$, and $a_1, c_1$ are part of a solution to $\mathcal{L}_{\varepsilon, \delta}$;
\item[-] $N_4(\cdot, a_2/b_2)\cong M_2$; 
\item[-] The Seifert and slope data $[\Sigma_{4,1}, (-\delta\cdot c_1)/1, 5/8, -1/4 \mid (0,1)]$ and $[\Sigma_{4,1}, 1/8, -3/4 \mid (0,1)]$ are slope-isomorphic. 
\end{itemize}
By Proposition \ref{prop:slopeisomorphic}, the final bullet point is equivalent $-\delta\cdot c_1+5/8-1/4=1/8-3/4$. By Proposition \ref{prop:hypDehnfilling}, we can find a finite list of slopes $a_2/b_2$ such that if $N_4(\cdot, a_2/b_2)\cong M_2$ then $a_2/b_2$ belongs to this list. Testing each element of the list, in this case using SnapPy and in general using Theorem \ref{thm:homeomorphismalgo}, we see that $N_4(\cdot, a_2/b_2)\cong M_2$ if and only if $a_2/b_2=1$. Hence $a_1, b_2, a_2, b_2$ give a Dehn filling of $N$ to $M$ in Case 2.2 if and only if they are part of a solution to the following linear system: 
$$\bigvee_{\varepsilon, \delta, \tau \in \{\pm 1\}}\left(\{\beta_1=\varepsilon\}\wedge \mathcal{L}_{\varepsilon, \delta} \wedge\{-\delta\cdot \gamma_1=-1\}\wedge \{\alpha_2=\beta_2=\tau\}\right).$$

{\bf Case 4}. $N_2$ does not Dehn fill to becomes $T^2\times I$ and $b_1\neq 0$: We may again assume $b_2\neq 0$. Then, by Theorem \ref{thm:nonisomorphicSFS}, $N_2$ Dehn fills to a Seifert fibred space with a unique Seifert structure up to isotopy. This Seifert structure will therefore not have fibred slopes lining up with the fibred slopes on $N_1, N_3$ under gluing by $\psi_1, \psi_2$. Thus $N(a_1/b_1, a_2/b_2)$ has a JSJ-decomposition with four JSJ-components and in particular $N(a_1/b_1, a_2/b_2)\not\cong M$. Hence there are no desired Dehn fillings in Case 4. \\

By the above discussion, the linear system given in Case 3.2. is the desired mono-quadratic system for this example. 
\end{example}

\subsection{Proof of Proposition \ref{prop:Rephrased_main_prop}}\label{subsection:main_proof_new}

Let us recall and prove the following proposition. \\

\noindent
{\bf Proposition \ref{prop:Rephrased_main_prop}.} Consider $\partial$-framed (not necessarily connected) 3-manifolds $N, M$, a $\partial$-parameterisation $(\mathcal{L}, (\rho_j, \sigma_j), \mathcal{P})$, and a $\partial$-equation system $\mathcal{F}$ on $M$ in variables $\alpha_1,\beta_1, \ldots, \alpha_l, \beta_l, \gamma_1, \ldots, \gamma_m$. There is a mono-quadratic system $\mathcal{Q}$ which detects every $a_1,b_1, \ldots, a_l, b_l$ giving a $\partial$-framed filling of $N$ to $(M, \mathcal{F})$ via $(\mathcal{L}, (\rho_j, \sigma_j), \mathcal{P})$. Moreover, there is an algorithm that, given $M, N, \mathcal{F}, \mathcal{L}, (\rho_j, \sigma_j), \mathcal{P}$, outputs $\mathcal{Q}$. 

\begin{proof}[Proof of Proposition \ref{prop:Rephrased_main_prop} assuming Lemma \ref{lemma:technicallemma}]
Let us proceed by induction on $k$. The base case $k=0$ is implicit in the following proof, as we only apply the induction hypothesis if $k>0$. Let us observe that changing the $H_1$-basis of a torus changes the coefficients of a slopes by a linear transformation. Hence we are free to choose any $H_1$-basis of boundary tori in the $\partial$-framings of $N,M$. Throughout the following proof, let $p_j, q_j$ be the value of $\rho_j, \sigma_j$ evaluated at $\alpha_i=a_i, \beta_i=b_i, \gamma_i=c_i$.\\

Suppose $\mathcal{L}$ is of the form $\bigvee_s \bigwedge_t \mathcal{L}_{s,t}$ for each $\mathcal{L}_{s,t}$ a single equation. Let $\mathcal{L}_s:=\bigwedge_t\mathcal{L}_{s,t}$. Then each $(\mathcal{L}_s, (\rho_j, \sigma_j), \mathcal{P})$ is a $\partial$-parametrisation. Moreover $a_1, b_1, \ldots, a_l, b_l, c_1, \ldots, c_m$ give a $\partial$-framed filling of $N$ to $(M, \mathcal{F})$ via $(\mathcal{L}, (\rho_j, \sigma_j), \mathcal{P})$ if and only if they give a $\partial$-framed filling of $N$ to $(M, \mathcal{F})$ via at least one of $(\mathcal{L}_s, (\rho_j, \sigma_j), \mathcal{P})$. Suppose we can find a mono-quadratic system $\mathcal{Q}_s$ satisfying the proposition for $\mathcal{L}$ replaced by $\mathcal{L}_s$. Then $\mathcal{Q}=\bigvee_s \mathcal{Q}_s$ is the desired mono-quadratic system. Hence we may, and will, assume that $\mathcal{L}$ only includes equations and $\vee$'s. \\

Suppose $N,M$ are connected. By Theorem \ref{thm:prime_decomp}, we compute the prime decompositions $N=N^{(1)}\#\ldots \#N^{(a)},$ $M=M^{(1)}\# \ldots\#M^{(b)}$. Any Dehn filling of $N$ will be the connect sum of the corresponding Dehn fillings of $N^{(i)}$. Hence fix an ordered partition $\mathcal{OP}=(P_1, \ldots, P_a)$ of $\{1, \ldots, b\}$ into $a$, possibly empty, sets. Suppose we can find a mono-quadratic system $\mathcal{Q}_{\mathcal{OP}}$ satisfying the proposition for $N$ replaced by $\bigsqcup_{i=1}^a N^{(i)}$ and $M$ replaced by $\bigsqcup_{i=1}^a\#_{j\in P_i}M^{(j)}$. Here we fix the empty connect sum $\#_{j\in \emptyset}M^{(j)}$ to be $S^3$. Then $\mathcal{Q}:=\bigvee_\mathcal{OP}\mathcal{Q}_{\mathcal{OP}}$, where we index over all ordered partitions $\mathcal{OP}$, would give the desired mono-quadratic system $\mathcal{Q}$. We similarly find $\mathcal{Q}$ if $N,M$ are not connected. Hence we may assume $N$ is prime, that is admits no essential spheres.\\

Check whether $N$ is $D^2\times S^1$. A Dehn filling of $N$ is $D^2\times S^1$ or a lens space. Hence let us assume $M$ is one of these two and thus by convention, we know $\mathcal{F}=\emptyset$. Consider the following cases: 
\begin{enumerate}
 \item If $k=0$, that is we do not Dehn fill $N$, then $a_1, b_1, \ldots, a_l, b_l, c_1, \ldots, c_m$ gives a $\partial$-framed filling of $N$ to $(M, \mathcal{F})$ via $(\mathcal{L}, (\rho_j, \sigma_j), \mathcal{P})$ if and only if: $\gcd(a_1, b_1)=\ldots=\gcd(a_l, b_l)=1$, $a_1, b_1, \ldots, a_l, b_l, c_1, \ldots, c_m$ are part of a solution to $\mathcal{L}$, and $M\cong D^2\times S^1$. We adopted the convention that if $k=0$, then $\gcd(a_1, b_1)=\ldots=\gcd(a_l, b_l)=1$ follows from $a_1, b_1, \ldots, a_l, b_l$ part of a solution to $\mathcal{L}$. Hence we may take $\mathcal{Q}=\mathcal{L}$ if $M\cong D^2\times S^1$ and otherwise we take $\mathcal{Q}$ with empty solution set. \\
 \item If $k=1$, without loss of generality, let $N$ have a $\partial$-framing such that $N(p_1/q_1)$ is the lens space $L(p_1, q_1)$. By Corollary \ref{Lens_recognition} we may check whether $M$ is a lens space. If not, we are done, and if so find $p,q\in \mathbb Z$ such that $M\cong L(p,q)$. Then $a_1, b_1, \ldots, a_l, b_l, c_1, \ldots, c_m$ give a $\partial$-framed filling of $N$ to $(M, \mathcal{F})$ via $(\mathcal{L}, (\rho_j, \sigma_j), \mathcal{P})$ if and only if: $\gcd(a_1, b_1)=\ldots=\gcd(a_l, b_l)=1$, $a_1, b_1, \ldots, a_l, b_l, c_1, \ldots, c_m$ are part of a solution to $\mathcal{L}$, and $N(p_1/q_1)\cong L(p,q)$. By convention if $\gcd(p_1, q_1)=1$ and $a_1, b_1, \ldots, a_l, b_l$ are part of a solution to $\mathcal{L}$, then $\gcd(a_1, b_1)=\ldots=\gcd(a_l, b_l)=1$. Moreover by Lemma \ref{lemma:classification_lens}, $N(p_1/q_1)\cong L(p_1, q_1)\cong L(p,q)$ if and only if for some $\varepsilon \in \{\pm 1\}$: $p_1=\varepsilon \cdot p, q_1\equiv \varepsilon\cdot q^{\pm} \pmod p$. Solutions $p_1, q_1$ to the latter equations are necessarily coprime. Hence we may take: $$\mathcal{Q}=\mathcal{L}\wedge \{\rho_1=\varepsilon\cdot p, \sigma_1\equiv \varepsilon\cdot q^{\pm} \pmod p, \varepsilon\in \{\pm 1\}\}.$$
\end{enumerate}
In summary, we may assume $N$ is not $D^2\times S^1$. \\

Check whether $N$ is $T^2\times I$ and if so let $N$ have a $\partial$-framing with $H_1$-basis equal under isotopy through the product structure of $T^2\times I$. Consider the following cases: 
\begin{enumerate}
 \item If $k=0, \mathcal{F}=\emptyset$, we only need to check whether $M\cong T^2\times I$ and $a_1, b_1, \ldots, a_l, b_l, c_1, \ldots, c_m$ is part of solution to $\mathcal{L}$.
 \item Suppose $k=0$ and $\mathcal{F}$ is of the form $\mathcal{M}_T\cdot X=Y\cdot \mathcal{M}_{T'}\cdot X'$ for $$X:=X_1\cdot \begin{pmatrix} 1 & n_1\\ 0 & 1\end{pmatrix} \cdots \begin{pmatrix} 1 & n_h\\ 0 & 1\end{pmatrix} \cdot X_{h+1}, \text{ and } X':=X'_1\cdot \begin{pmatrix} 1 & n'_1\\ 0 & 1\end{pmatrix} \cdots \begin{pmatrix} 1 & n'_{h'}\\ 0 & 1\end{pmatrix} \cdot X'_{h'+1},$$ with the same conventions as in Definition \ref{def:framing_pairing}. For $a_1,b_1, \ldots, a_l, b_l, c_1, \ldots, c_m$ to give a $\partial$-framed filling of $N$ to $(M,\mathcal{F})$ via $(\mathcal{L}, (\rho_j, \sigma_j), \mathcal{P})$ we certainly need $M\cong T^2\times I$ and $a_1, b_1, \ldots, a_l, b_l, c_1, \ldots, c_m$ part of a solution to $\mathcal{L}$. Without loss of generality assume the $\partial$-framings on $N,M$ are isotopic through their product structures. In particular, for any homeomorphism $\phi:N\to M$, we must have $\mathcal{M}_{\phi, T}=\mathcal{M}_{\phi, T'}\in SL_2(\mathbb Z)$. Moreover, we can choose $\phi$ such that $\mathcal{M}_{\phi,T}$ becomes an arbitrary element of $SL_2(\mathbb Z)$. Hence $\phi, a_1, b_1, \ldots, a_l, b_l, c_1, \ldots, c_m$ satisfy the $\partial$-equation if and only if $X\cdot (X')^{-1}$ and $Y$ are conjugate by some matrix in $SL_2(\mathbb Z)$. By Lemma \ref{lemma:sl2z}, we may find a linear system $\mathcal{L}'$ such that $a_1, b_1, \ldots, a_l, b_l, c_1, \ldots, c_m$ are part of a solution to $\mathcal{L}'$ if and only if $X\cdot (X')^{-1}$ is conjugate to $Y$ by a matrix in $SL_2(\mathbb Z)$. Hence $\mathcal{Q}=\mathcal{L}\wedge \mathcal{L}'$ is the desired Diophantine system. 
%
%
%
%
%
 \item If $k=1$, that is we fill along one boundary torus, then all Dehn fillings of $N$ are $D^2\times S^1$. Thus we are done by checking whether $M$ is $D^2\times S^1$, whether $a_1, b_1, \ldots, a_l, b_l, c_1, \ldots, c_m$ are part of a solution to $\mathcal{L}$, and whether $\gcd(a_1, b_1), \ldots, \gcd(a_l, b_l)=1$. Recall from the definition of $\partial$-parameterisations, that the final condition is equivalent to $\gcd(p_1, q_1)=1$, which in turn is equivalent to $p_1, q_1$ being part of a solution to $p_1\cdot r+q_1\cdot s=1$. Hence we obtain our desired mono-quadratic system $\mathcal{Q}=\{\rho_1\cdot r+\sigma_1\cdot s=1\}\wedge \mathcal{L}$.
 \item If $k=2$, check whether $M$ is a lens space and if so find $p,q\in \mathbb Z$ such that $M\cong L(p,q)$. Then Lemma \ref{lemma:lens_quadratic} gives a desired mono-quadratic equation.
\end{enumerate}
In summary, we may assume $N$ is not $T^2\times I$.\\

Suppose some component $C$ of $N$ is $D^2\times S^1$ or $T^2\times I$. Hence, a possibly trivial, Dehn filling of $C$ is a $T^2\times I$, $D^2\times S^1$, or a lens space. Fix some component $D$ of $M$ that is $T^2\times I$, $D^2\times S^1$, or a lens space. If there is no such $D$, we are immediately done. By taking the $\vee$ of the resulting mono-quadratic equations as we consider all possible $D$, it is sufficient to find a mono-quadratic system detecting all $a_1, b_1, \ldots, a_l, b_l, c_1, \ldots, c_m$ giving $\partial$-framed filling of $N$ to $(M, \mathcal{F}$) via $(\mathcal{L}, (\rho_j, \sigma_j), \mathcal{P})$ in which $C$ becomes $D$. Fix the following conventions: 
\begin{itemize}
\item[-] Let $\mathcal{F}_D\subset \mathcal{F}$ be the $\partial$-equations involving boundary tori of $D$, 
\item[-] Let $\{(\rho^C_j, \sigma^C_j)\}\subset\{(\rho_j, \sigma_j)\}$ be the integral-affine functions corresponding to filling slopes in boundary tori of $C$ and let $\{(\rho^{N-C}_j, \sigma^{N-C}_j)\}$ be the remaining integral-affine functions, 
\item[-] Let $\mathcal{P}_{D}\subset \mathcal{P}$ be those parts of $\mathcal{P}$ containing variables appearing in the $\mathcal{F}_D$ or $(\rho^C_j, \sigma^C_j)$, 
\item[-] Let $\mathcal{L}= \mathcal{L}_{D}\wedge \mathcal{L}_{M-D}$ such that $\mathcal{L}_D$ and $\mathcal{L}_{M-D}$ respectively only use variables from $\mathcal{P}_D$ and $\mathcal{P}_{M-D}$. 
\end{itemize}
Observe that $(\mathcal{L}_D, (\rho^C_j, \sigma^C_j), \mathcal{P}_D)$ and $(\mathcal{L}_{M-D}, (\rho^{N-C}_j, \sigma^{N-C}_j), \mathcal{P}\setminus\mathcal{P}+_{D})$ are $\partial$-parameterisations with a disjoint set of variables. By convention $\mathcal{F}_D$ does not involve boundary tori from $M-D$ and is hence a $\partial$-equation system on $D$. Hence we can separately consider $\partial$-framed fillings of $C$ to $(D, \mathcal{F}_D)$ via $(\mathcal{L}_D, (\rho_j^C, \sigma_j^C), \mathcal{P}_D)$ and $\partial$-framed fillings of $N-C$ to $(M-D, \mathcal{F}-\mathcal{F}_D)$ via $(\mathcal{L}_{M-D}, (\rho^{N-C}_j, \sigma^{N-C}_j), \mathcal{P}\setminus\mathcal{P}_{D})$. Repeating this we may assume no component of $N$ is $D^2\times S^1$ or $T^2\times I$. An identical argument lets us separately consider the case of $M$ homeomorphic to $D^2\times S^1, T^2\times I$, or a lens space, and the case of $M$ having no component which is $D^2\times S^1, T^2\times I$, or a lens space. We may also assume no component of $N$ is closed. \\


In summary, $N$ is prime and no component of $N$ is $D^2\times S^1$ or $T^2\times I$. Therefore we are in the setup of Lemma \ref{lemma:technicallemma} and find the list of 6-tuples $\mathcal{N}_{red}=(N_i, M_i, \mathcal{F}_i, (\rho_j^{(i)}, \sigma_j^{(i)}), \mathcal{L}_i, \mathcal{P}_i)_{i\in I}$ given by Lemma \ref{lemma:technicallemma}. By induction, for each $i\in I$, we may find a mono-quadratic system $\mathcal{Q}_i$ detecting all $a_1, b_1, \ldots, a_l, b_l$ giving a $\partial$-framed filling of $N_i$ to $(M_i, \mathcal{F}_i)$ via $(\mathcal{L}_i, (\rho_j^{(i)}, \sigma_j^{(i)}), \mathcal{P}_i)$. By Lemma \ref{lemma:technicallemma} the mono-quadratic system $\bigvee_{i\in I} \mathcal{Q}^{(i)}$ detects every $a_1, b_1, \ldots, a_l, b_l$ giving a JSJ-exceptional $\partial$-framed filling of $N$ to $(M, \mathcal{F})$ via $(\mathcal{L}, (\rho_j, \sigma_j), \mathcal{P})$. Suppose we can find a mono-quadratic system $\mathcal{Q}'$ detecting every $a_1, b_1, \ldots, a_l, b_l$ giving a non-JSJ-exceptional filling from $N$ to $(M, \mathcal{F})$ via $(\mathcal{L}, (\rho_j, \sigma_j), \mathcal{P})$. Then $\mathcal{Q}=\mathcal{Q}'\vee\left(\bigvee_{i\in I} \mathcal{Q}^{(i)}\right)$ is the desired mono-quadratic system. Hence let us find $\mathcal{Q}'$.\\

Let $\mathcal{T}_N$ denote the JSJ-decomposition of $N$. If $M$ does not admit a JSJ-decomposition, we are immediately done. Therefore let $\mathcal{T}_M$ denote the JSJ-decomposition of $M$. By Theorem \ref{thm:jsjalgo}, there is an algorithm outputting $\mathcal{T}_N$ and $\mathcal{T}_M$. Coefficients $a_1, b_1, \ldots, a_l, b_l, c_1, \ldots, c_m$ give a non-JSJ-exceptional $\partial$-framed filling of $N$ to $(M, \mathcal{F})$ via $(\mathcal{L}, (\rho_j, \sigma_j), \mathcal{P})$ only if they give a $\partial$-framed filling of $N-\mathcal{T}_N$ to $(M-\mathcal{T}_M, \mathcal{F}')$ via $(\mathcal{L}, (\rho_j, \sigma_j), \mathcal{P})$ where: 
\begin{itemize}
 \item[-] For each pair of boundary tori $T, T'\in \partial M-\mathcal{T}_M$ which are obtained by cutting along a single torus of $\mathcal{T}_M$ we fix the $H_1$-basis that are equal under regluing of the JSJ-tori. 
 \item[-] The $\partial$-framing on $M-\mathcal{T}_M$ on the remaining boundary tori consists of the $H_1$-basis inherited from the $\partial$-framing on $M$. The $\partial$-framing on $N-\mathcal{T}_N$ is defined similarly. 
 \item[-] The set $\mathcal{F}'$ is $\mathcal{F}$ with additional $\partial$-equations $\mathcal{M}_T=\mathcal{M}_{T'}$ for each pair of boundary tori $T, T'\in \partial (M-\mathcal{T}_M)$ which are obtained by cutting along a torus of $\mathcal{T}_M$. 
\end{itemize}
As discussed in Example \ref{ex:partial-system}, the additional $\partial$-equations in $\mathcal{F}'$ ensure that a homeomorphism from a Dehn filling of $N-\mathcal{T}_N$ to $M-\mathcal{T}_M$ will respect regluing along the JSJ-tori. Moreover if $a_1, b_1, \ldots, a_l, b_l, c_1, \ldots, c_m$ give a $\partial$-framed filling of $N-\mathcal{T}_N$ to $(M-\mathcal{T}_M, \mathcal{F}')$ via $(\mathcal{L}, (\rho_j, \sigma_j), \mathcal{P})$ then they give a $\partial$-framed filling of $N$ to $(M, \mathcal{F})$ via $(\mathcal{L}, (\rho_j, \sigma_j), \mathcal{P})$. Hence it is now sufficient to find $\mathcal{Q}'$ in the case of $N, M$ having empty JSJ-decomposition, that is being a disjoint union of hyperbolic and Seifert fibred pieces. \\

Suppose we Dehn fill along a boundary torus of a hyperbolic component $C$ of $N$. By Proposition \ref{prop:hypDehnfilling}, there is an algorithm outputing a finite list of slopes in the boundary tori of $C$ such that Dehn filling avoiding these finite lists will give a hyperbolic manifold not homeomorphic to any component of $M$. By induction on $k$ we obtain a mono-quadratic system detecting $a_1, b_1, \ldots, a_l, b_l$ giving a $\partial$-framed filling of $N$ to $(M, \mathcal{F})$ via $(\mathcal{L}, (\rho_j, \sigma_j), \mathcal{P})$ using a slope from one of these finite lists. Therefore, we may assume no boundary torus of $N$ we fill along belongs to a hyperbolic JSJ-component. \\

Hence we only consider Dehn fillings of $N$ to $M$ that induce a matching between hyperbolic components of $N$ and $M$ and restrict to homeomorphisms on these pieces. By Proposition \ref{prop:list_MCG}, for a pair of homeomorphic hyperbolic manifolds there is a finite set of homeomorphisms, up to isotopy, between them and these may be listed by an algorithm. We may therefore separately consider the Dehn fillings of $N$ to $M$ inducing a fixed homeomorphism $\psi$ between the union $N_{hyp}$ of hyperbolic components of $N$ and the union $M_{hyp}$ of hyperbolic components of $M$. For each $T\in \partial M_{hyp}$, consider the fixed matrix $Z_T:=M_{\psi, T}\in SL_2(\mathbb Z)$. Hence, coefficients $a_1, b_1, \ldots, a_l, b_l, c_1, \ldots, c_m$ give a $\partial$-framed Dehn filling of $N$ to $(M, \mathcal{F})$ via $(\mathcal{L}, (\rho_j, \sigma_j), \mathcal{P})$ that restricts to $\psi$ on $N_{hyp}$ if and only if $a_1, b_1, \ldots, a_l, b_l, c_1, \ldots, c_m$ give a $\partial$-framed Dehn filling of $N-N_{hyp}$ to $(M-M_{hyp}, \mathcal{F}')$ via $(\mathcal{L}'', (\rho_j, \sigma_j), \mathcal{P})$ where: 
\begin{itemize}
 \item[-] We choose the $\partial$-framing of $N-N_{hyp}, M-M_{hyp}$ inherited from $N, M$; 
 \item[-] For every $\partial$-equation of $\mathcal{F}$ we replace each occurrence $\mathcal{M}_{\phi,T}$ with $T\in\partial M_{hyp}$ by $Z_T$. Under this operation every $\partial$-equation of $\mathcal{F}$ either becomes a $\partial$-equation for $M-M_{hyp}$ (if it involves tori not in $\partial M_{hyp}$), or becomes an equation of matrices in $GL_2(\mathbb Z)$ with coefficients integral-affine functions. Call these two sets of equation $\mathcal{F}'$ and $\mathcal{G}$ respectively. We take $\mathcal{F}'$ to be the new $\partial$-equation system on $M-M_{hyp}$; 
 \item[-] Observe that each equation in the newly obtained set $\mathcal{G}$ is of the form described in Lemma \ref{lemma:sl2z}. Hence by Lemma \ref{lemma:sl2z}, we may find linear system $\mathcal{L}'$ such that $a_1, b_1, \ldots, a_l, b_l$ is part of a solution to $\mathcal{L}'$ if and only if when evaluated at $a_1, b_1, \ldots, a_l, b_l$ all the equations of $\mathcal{G}$ are true. We take $\mathcal{L}''=\mathcal{L}\wedge \mathcal{L}'$. 
%
%
%
%
%
%
%
\end{itemize}
Hence, we have reduced to the case of $N$ and $M$ both disjoint unions of Seifert fibred spaces. \\

Suppose, for now, that $M$ has finitely many Seifert structures up to isotopy, or up to isomorphism if $\partial M=\emptyset$. By the discussion of Section \ref{section:Seifert_section}, there is an algorithm outputting all possible Seifert structures of $M$ up to isotopy. Hence, having fixed Seifert structures on $N,M$, it is sufficient to find a mono-quadratic system $\mathcal{Q}'$ detecting all $a_1, b_1, \ldots, a_l, b_l$ giving Dehn filling under which the fixed Seifert structure of $N$ extends to the fixed Seifert structure of $M$ up to isotopy. Moreover, it is sufficient to detect the Dehn fillings inducing a fixed matching of the components of $N$ and $M$. Let us fix the following convention: 
\begin{itemize}
\item[-] $\partial$-framings on $N,M$ will have $H_1$-basis on each boundary torus consisting of the fibred slopes $\lambda$ and the section slope $\mu$; 
\item[-] Let $[\Sigma, r_1/s_1, \ldots, r_h/s_h]$ denote the Seifert data on a component $C$ of $N(\rho_1/\sigma_1, \ldots, \rho_k/\sigma_k)$ given by extending the Seifert structure on $N$. In particular, $r_1, s_1, \ldots, r_h, s_h$ are integral-affine functions in variables $\alpha_1,\beta_1, \ldots, \alpha_l, \beta_l, \gamma_1, \ldots, \gamma_m$; 
\item[-] Let $D$ be the component of $M$ corresponding to $C$ and consider tori $T^D_{1}, \ldots, T^D_{m_D}\subset \partial D$; 
\item[-] Let $[\Sigma', r'_1/s'_1, \ldots, r'_{h'}/s'_{h'}]$ denote the Seifert data on $D$. Here $r'_1, s'_1, \ldots, r'_{h'}, s'_{h'}$ are fixed integers; 
\item[-] Let $Z^D_i$ denote matrices in $GL_2(\mathbb Z)$. 
\end{itemize}
Recall that isomorphisms of Seifert structures preserve unoriented fibred slopes and therefore restrict on boundary tori to maps of the form $\pm \begin{pmatrix} 1 & z \\ 0 & 1\end{pmatrix}$ with respect to the $H_1$-basis given by the Seifert structures. In fact there is an isomorphism $\phi\colon C\to D$ of Seifert structures with $\mathcal{M}_{\phi, T^D_i}=Z^D_i$ for all $i=1, \ldots, m_D$ if and only if exists $z^D_1, \ldots, z^D_{m_D}\in \mathbb Z, \varepsilon\in \{\pm 1\}$ such that: 
\begin{itemize}
 \item[-] For each $i=1, \ldots, m$: $Z^D_i=\varepsilon \cdot \begin{pmatrix} 1 & z^D_i \\ 0 & 1
 \end{pmatrix}$; 
 \item[-] The Seifert and slope data $[\Sigma, r_1/s_1, \ldots, r_h/s_h \mid (0,1), \ldots, (0,1)]$ and $[\Sigma', r'_1/s'_1, \ldots, r'_{l'}/s'_{l'} \mid \varepsilon\cdot (z^D_1,1), \ldots, \varepsilon\cdot (z^D_{m_D},1)]$ are slope-isomorphic.\footnote{Here $\varepsilon$ encodes whether $\phi$ preserves the orientation of fibred slopes in $\partial C, \partial D$ or flips these orientations.}
\end{itemize}
By Proposition \ref{prop:SFS_linear} there is a linear system $\mathcal{L}_{D}$ such that $r_1, s_1, \ldots, r_h, s_h, z^D_1, \ldots, z^D_{m_D}, \varepsilon$ are part of a solution to $\mathcal{L}_D$ if and only if $\gcd(r_1, s_1)=\ldots=\gcd(r_h, s_h)=1$ and the latter bullet point above is satisfied. Hence $a_1, b_1, \ldots, a_l, b_l, c_1, \ldots, c_m$ gives a $\partial$-framed filling of $N$ to $(M, \mathcal{F})$ via $(\mathcal{L}, (\rho_j, \sigma_j), \mathcal{P})$ if and only if for each component $D$ of $M$ exist $z^D_1, \ldots, z^D_{m_D}\in \mathbb Z, \varepsilon\in \{\pm 1\}$ such that: 
\begin{itemize}
\item[-] $a_1, b_1, \ldots, a_l, b_l, c_1, \ldots, c_m$ are part of a solution to $\mathcal{L}$; 
\item[-] for each component $D$ of $N$, $r_1, s_1, \ldots, r_h, s_h, z^D_1, \ldots, z^D_{m_D}, \varepsilon$ are part of a solution to $\mathcal{L}_D$ when evaluated at $a_1, b_1, \ldots, a_l, b_l, c_1, \ldots, c_m$; 
\item[-] after replacing each instance of $\mathcal{M}_{T^D_i}$ in $\mathcal{F}$ by $\varepsilon\cdot \begin{pmatrix} 1 & z^D_i \\ 0 & 1\end{pmatrix}$ in $\mathcal{F}$, all $\partial$-equations are true when evaluated at $a_1, b_1, \ldots, a_l, b_l, c_1, \ldots, c_m$. 
\end{itemize}
By Lemma \ref{lemma:sl2z} there is a linear system $\mathcal{L}'$ such that $a_1, b_1, \ldots, a_l, b_l, c_1, \ldots, c_m$, and $z^D_i$ are part of a solution to $\mathcal{L}'$ if and only if the final condition above holds. Hence a desired mono-quadratic system is $\mathcal{Q}=\left(\bigwedge_{D}\mathcal{L}_D\right)\wedge\mathcal{L}'\wedge \mathcal{L}$.\\

Let us conclude by addressing the case of $M$ having infinitely many Seifert fibrations up to isotopy, or up to isomorphism if $\partial M=\emptyset$. In this case, some component $D$ of $M$ must be a solid torus, thickened torus, or lens space. As discussed above, it is sufficient to consider the case of $M$ homeomorphic to a solid torus, thickened torus, or lens space. In the following let the $\partial$-framing on $N$ be given by a Seifert structure on $N$ (if a Seifert structure exists). \\

Suppose $M$ is a solid torus. Then $\mathcal{F}$ is empty. By the discussion of Example \ref{ex:solidtorus}, we see that $a_1, b_1, \ldots, a_l, b_l, c_1, \ldots, c_m$ give a non-JSJ-exceptional $\partial$-framed filling of $N$ to $(M, \mathcal{F})$ via $(\mathcal{L}, (\rho_j, \sigma_j), \mathcal{P})$ if and only if: 
\begin{itemize}
\item[-] $N$ is $[D^2-\{x_1, \ldots, x_k\}]$; 
\item[-] $a_1, b_1, \ldots, a_l, b_l, c_1, \ldots, c_m$ are part of a solution to $\mathcal{L}$; 
\item[-] and for some reindexing $p_1=\ldots=p_{k-1}=1$ and $\gcd(p_k, q_k)=1$;
\end{itemize}
or 
\begin{itemize}
\item[-] $N$ is $[D^2-\{x_1, \ldots, x_k\}, p/q]$; 
\item[-] $a_1, b_1, \ldots, a_l, b_l, c_1, \ldots, c_m$ are part of a solution to $\mathcal{L}$; 
\item[-] and $p_1=\ldots=p_k=1$.
\end{itemize}
There is a mono-quadratic system $\mathcal{Q}$ such that $a_1, b_1, \ldots, a_l, b_l, c_1, \ldots, c_m$ is part of a solution to $\mathcal{Q}'$ if and only if $a_1, b_1, \ldots, a_l, b_l$ satisfies the above. This $\mathcal{Q}'$ therefore detects all $a_1, b_1, \ldots, a_l, b_l$ giving non-JSJ-exceptional $\partial$-framed fillings of $N$ to $(M, \mathcal{F})$ via $(\mathcal{L}, (\rho_j, \sigma_j), \mathcal{P})$. \\

Suppose $M$ is a thickened torus $T^2\times I$. By the discussion of Example \ref{ex:thickenedtorus}, $N\mapsto N(p_1/q_1, \ldots, p_k/q_k)$ is a non-JSJ-exceptional Dehn filling to $M$ if only if $N$ is $[S^1\times I-\{x_1, \ldots, x_k\}]$ and, up to reordering $p_1=\ldots=p_k=1$. However in such a Dehn filling the $\partial$-framing on $\partial M$ given by the Seifert structure $[S^1\times I]$ and the $\partial$-framing inherited from the Seifert structure on $N$ will not necessarily be the same. We may take the two corresponding $H_1$-basis on $T^2\times \{0\}\subset \partial M$ to be equal in which case, as discussed in Example \ref{ex:thickenedtorus}, the two corresponding $H_1$-basis on $T^2\times \{1\}$ will be related by the matrix $\begin{pmatrix} 1 & q \\ 0 & 1\end{pmatrix}, q=\sum_j q_j$. Hence $a_1, b_1, \ldots, a_l, b_l$ giving a $\partial$-framed filling of $N$ to $(M, \mathcal{F})$ via $(\mathcal{L}, (\rho_j, \sigma_j), \mathcal{P})$ are detected by $(T^2\times I, T^2\times I, \mathcal{F}', (\emptyset), \mathcal{L}', \mathcal{P}')$ where: 
\begin{itemize}
\item[-] The $\partial$-framing on both copies of $T^2\times I$ are given by the Seifert structure $[S^1\times I]$;
\item[-] $\mathcal{L}'=\mathcal{L}\wedge\{\rho_1=\ldots=\rho_k=1\}$; 
\item[-] If $\mathcal{F}=\emptyset$, then $\mathcal{F}'=\emptyset$; 
\item[-] If $\mathcal{F}$ consists of a $\partial$-equation of the form $\mathcal{M}_{T\times \{0\}}\cdot X=Y\cdot \mathcal{M}_{T\times \{1\}}\cdot X'$, then $\mathcal{F}'$ consists of the $\partial$-equation $\mathcal{M}_{T\times \{0\}}\cdot X=Y\cdot \mathcal{M}_{T\times \{1\}}\cdot \begin{pmatrix} 1 & \sigma \\ 0 & 1 \end{pmatrix}\cdot X'$, for $\sigma=\sum_j\sigma_j$; 
\item[-] $\mathcal{P}'$ is the partition with one part. 
\end{itemize}
Yet, we have already addressed the case of $N$ a thickened torus and $k=0$.\\

Suppose $M$ is a lens space. Find $p,q$ such that $M$ is $L(p,q)$. By the same argument as the case of $M$ having finitely many Seifert structures up to isotopy, we may find a linear system $\mathcal{L}'$ detecting all $a_1, b_1, \ldots, a_l, b_l$ giving a $\partial$-framed filling for which the Seifert structure of $N$ extends to a Seifert structure of $L(p,q)$ with base surface $\mathbb RP^2$ as given in Lemma \ref{lemma:lensspacecharacterisation}. Hence suppose $a_1, b_1, \ldots, a_l, b_l, c_1, \ldots, c_m$ gives a non-JSJ-exceptional $\partial$-framed filling of $N$ to $(M, \mathcal{F})$ via $(\mathcal{L}, (\rho_j, \sigma_j), \mathcal{P})$ extending the Seifert structure on $N$ to a Seifert structure of $L(p,q)$ with base surface $S^2$. Then $N$ must be a Seifert fibred space over a planar surface with at most $b\leq 2$ singular fibres and all but at most $2-b$ of the filling slopes must have $p_j=1$. In summary, for each viable $N$ there is a linear system $\mathcal{L}''$ and integral-affine functions $r_1, s_1, r_2, s_2$ in variables $\rho_1, \sigma_1, \ldots, \rho_k, \sigma_k$ such that $a_1, b_1, \ldots, a_l, b_l, c_1, \ldots, c_m$ give a framed filling of $N$ to $(M, \mathcal{F})$ if and only if $a_1, b_1, \ldots, a_l, b_l, c_1, \ldots, c_m$ are part of a solution to $\mathcal{L}''$, $\gcd(r_1, s_1)=\gcd(r_2, s_2)=1$, and $[S^2, r_1/s_1, r_2/s_2]$ is a Seifert fibration of $L(p,q)$. Combining all this and Lemma \ref{lemma:lens_quadratic}, we obtain a mono-quadratic system detecting every $a_1, b_1, \ldots, a_l, b_l$ giving such a $\partial$-framed filling of $N$ to $(M, \mathcal{F})$ via $(\mathcal{L}, (\rho_j, \sigma_j), \mathcal{P})$. \\

This concludes the proof. 
\end{proof}

\subsection{Proof of Lemma \ref{lemma:technicallemma}} \label{section:Lemma_proof}
A quick note to the reader. In the following proof, we construct several $6$-tuples $(N, M, \mathcal{F}, (\rho_j, \sigma_j), \mathcal{L}, \mathcal{P})$. In each construction we do not explicitly check all the properties required of a $6$-tuple to detect integers $a_1, b_1, \ldots, a_l, b_l$. For example we do not explicitly check that $(\mathcal{L}, (\rho_j, \sigma_j), \mathcal{P})$ are a $\partial$-parametrisation. All such checks are routine and merely require individually considering each of the conditions in the definitions of Subsection \ref{subsection:formalism}. Moreover, on a first read-through the reader should feel free to ignore the technicalities associated with the partitions $\mathcal{P}$.

\begin{proof}[Proof of Lemma~\ref{lemma:technicallemma}]
Call the boundary tori of $N-\mathcal{T}_N$: {\it JSJ-} if they belong to $\mathcal{T}_N$; {\it filling} if they are among the first $k$ tori of $N$, that is if they are used in the Dehn filling specified by $(p_1, q_1, \ldots, p_k, q_k)$; and ${\it free}$ otherwise. Suppose we find suitable $\mathcal{N}_{red}$ for some choice of $\partial$-framings on $N,M$ and modify these $\partial$-framings by changing the choice of $H_1$-basis in some boundary torus. Then we may obtain new suitable $\mathcal{N}_{red}$ by composing each $\rho_j^{(i)}, \sigma_j^{(i)}$, $\mathcal{L}_i, \mathcal{F}_i$ with the suitable change of basis matrix. Hence we may, without loss of generality, choose $H_1$-basis in each boundary torus as we please. \\

In this proof, we start with empty $\mathcal{N}_{red}$ and, in steps, add additional 6-tuples until the lemma holds. It will be clear from the construction that there is an algorithm outputting $\mathcal{N}_{red}$. We will call coefficient $a_1, b_1, \ldots, a_l, b_l$ {\it undetected} if they give a $\partial$-framed filling of $N$ to $(M, \mathcal{F})$ via $(\mathcal{L}, (\rho_j, \sigma_j), \mathcal{P})$ but are not yet detected by $\mathcal{N}_{red}$. We call a Dehn filling of $N$ {\it undetected} if it is given by undetected $a_1, b_1, \ldots, a_l, b_l$. Thus our goal is to keep adding 6-tuples to $\mathcal{N}_{red}$ until no undetected Dehn filling is JSJ-exceptional. \\

We can decide which of the JSJ-components of $N$ are hyperbolic. By Proposition \ref{prop:find_short_slopes} and Theorem \ref{theorem:longFillingGeometry} we can find a finite list $S$ of slopes in the boundary tori of the hyperbolic JSJ-components, such that a Dehn filling of a hyperbolic JSJ-component avoiding this list will remain hyperbolic. For each slope $s\in S$ given by coprime coefficients $p_h, q_h$ for some $1\leq h\leq k$, add to $\mathcal{N}_{red}$ the 6-tuple $(N_s, M, \mathcal{F}, (\rho_j^{(s)}, \sigma_j^{(s)}), \mathcal{L}_s, \mathcal{P})$ with: 
\begin{itemize}
 \item[-] $N_s$ the Dehn filling of $N$ along $s$; 
 \item[-] $(\rho_j^{(s)}), (\sigma_j^{(s)})$ obtained from $(\rho_j), (\sigma_j)$ by omitting $\rho_h, \sigma_h$;
 \item[-] $\mathcal{L}_s=\mathcal{L}\wedge\{\rho_h=p_h, \sigma_h=q_h\}$.
\end{itemize}
A Dehn filling $N(p_1/q_1, \ldots, p_k/q_k)$ using the slope $s\in S$ is $\partial$-framed homeomorphic to the Dehn filling $N_s(p_1/q_1, \ldots, p_{h-1}/q_{h-1}, p_{h+1}/q_{h+1}, \ldots, p_k/q_k)$. Hence $a_1, b_1, \ldots, a_l, b_l$ giving such a $\partial$-framed filling of $N$ to $(M, \mathcal{F})$ via $(\mathcal{L}, (\rho_j, \sigma_j), \mathcal{P})$ is detected by $(N_s, M_s, \mathcal{F}_s, (\rho_j^{(s)}, \sigma_j^{(s)}), \mathcal{L}_s, \mathcal{P})$. Thus undetected Dehn fillings keep hyperbolic JSJ-components hyperbolic.\\

We can decide which of the JSJ-components of $N$ are Seifert fibred. Since no component of $N$ is $T^2\times I$ or $D^2\times S^1$, we may find all the finitely-many boundary slopes which are fibres in some Seifert structure. Just as above, we can add tuples to $\mathcal{N}_{red}$ detecting any Dehn filling using a fibre-parallel slope.
Hence undetected Dehn fillings do not use fibre-parallel slopes in any Seifert structure of any Seifert fibred JSJ-component.\\

Suppose there is some JSJ-component $A$ of $N$ which is Seifert fibred with Seifert data $[D^2-\{x_1, \ldots, x_b\}; q/p]$ and with at least $b$ filling boundary tori. Fix a Seifert structure on $A$ and, without loss of generality, assume the first $b$ tori of the $\partial$-framing correspond to the punctures $x_i$ in $A$ and have $H_1$-basis given by the Seifert structure on $A$. As discussed in Example \ref{ex:solidtorus}, a Dehn filling of $A$ along slopes $q_1/1, \ldots, q_b/1$ gives a solid torus. The meridian slope of this solid torus is $(q+ap)/p$, for $a=\sum_{i=1}^b q_i$. Moreover, we may assume $N-A$ is not empty, as else the Dehn filling is immediately not JSJ-exceptional. Thus, in summary, such a Dehn filling of $N$, that is a filling along $\partial A$ with $q_i=1$ for $1\leq i\leq b$, is $\partial$-framed homeomorphic to a Dehn filling of $N-A$ along $(q+ap)/p$. Therefore add to $\mathcal{N}_{red}$ the 6-tuple $(N-A, M, \mathcal{F}, (\rho'_j, \sigma'_j), \mathcal{L}', \mathcal{P}')$ with: 
\begin{itemize}
\item[-] In the $\partial$-framing on $N'$ the torus $T\subset \partial N'$ adjacent to $A$ has $H_1$-basis given by the above Seifert structure on $A$ and all other tori have $H_1$-basis inherited from $N$; 
\item[-] The integral-affine functions $\rho'_j, \sigma'_j$ are obtained from $\rho_j, \sigma_j$ by omitting $\rho_1, \sigma_1, \ldots, \rho_b, \sigma_b$ and adding the new integral-affine functions $\rho'_{k-b+1}=p, \sigma'_{k-b+1}=q+ap$ for $a=\sum_{j=1}^b \sigma_j$. Here the integral-affine functions $\rho'_{k-b+1}, \sigma'_{k-b+1}$ correspond to the filling slope in $T$;
\item[-] $\mathcal{P}'$ is obtained from $\mathcal{P}$ by merging all parts in which $\rho_1, \sigma_1, \ldots, \rho_b, \sigma_b$ have their variables into a single part of $\mathcal{P}$; 
\item[-] $\mathcal{L}'$ is $\mathcal{L}\wedge \{\rho_i=1 \mid i=1, \ldots, b\}$.
\end{itemize}
Any undetected $a_1, b_1, \ldots, a_l, b_l$ that give a Dehn filling of $N$ with $A$ becoming a solid torus is detected by this 6-tuple. Add such 6-tuples for each JSJ-component of the above form. Hence in undetected Dehn fillings no JSJ-component with Seifert data $[D^2-\{x_1, \ldots, x_b\}; q/p]$ fills to become a solid torus. By a variation on this argument, we may also add 6-tuples to $\mathcal{N}_{red}$ such that in undetected Dehn fillings no JSJ-component with Seifert data $[D^2-\{x_1, \ldots, x_b\}]$ fills to become a solid torus. As we already do not consider fibre-parallel fillings, by the discussion of Example \ref{ex:solidtorus}, these are the only ways a JSJ-component of $N$ fills to become a solid torus. Hence, in summary, in undetected Dehn fillings of $N$ no JSJ-component fills to become a solid torus.\\

Observe that the only Seifert fibred or hyperbolic 3-manifold with compressible boundary is the solid torus. Thus, by an innermost-disc argument, after an undetected Dehn filling the tori $\mathcal{T}_N$ will remain incompressible and cut the manifold into pieces which are hyperbolic or Seifert fibred. In particular, we are done if $M$ does not admit a JSJ-decomposition. Hence suppose $M$ has a JSJ-decomposition $\mathcal{T}_M$. Moreover, an undetected Dehn filling $\mathcal{N}_{red}$ is only JSJ-exceptional if Seifert fibred JSJ-components are Dehn filled to have Seifert structures with fibred slopes lining up. In this case these two Dehn fillings of JSJ-components glue to form a single Seifert fibred space. Repeating this, we see that a subset of $\mathcal{T}_N$ becomes the JSJ-decomposition of the Dehn filling of $N$, and in particular an element of $\mathcal{T}_N$ adjacent to a hyperbolic JSJ-components remains a JSJ-torus in the Dehn filling of $N$. Finally, observe that Seifert structures on JSJ-components of $N$ do not have fibred slopes lining up and thus do not extend to Seifert structures with fibred slopes lining up. Hence, an undetected Dehn filling is only JSJ-exceptional if some Seifert fibred JSJ-components fills to have non-isotopic Seifert structures, that is fills to become $T^2\times I$ or $K^2\tilde{\times} I$. \\

Suppose there is some JSJ-component $A_1$ of $N$ which has a Seifert structure with base surface a $b$-punctured disc with at most two singular fibres, each of multiplicity $2$, and with $b$ filling boundary tori. Then $A_1$ admits Dehn fillings under which its Seifert structure extends to $[D^2, 1/2, -1/2]$ up to isomorphism. Without loss of generality suppose that in the $\partial$-framing on $N$ the first $b$ boundary tori are the filling boundary tori of $A_1$ and that the $H_1$-basis on these boundary tori is given by Seifert structure on $A_1$. Then we may find a linear system $\mathcal{L}'$ such that $p_1, q_1, \ldots, p_b, q_b$ are part of a solution to $\mathcal{L}'$ if and only if $\gcd(p_1,q_1)=\ldots=\gcd(p_b, q_b)=1$ and they give a non-fibre-paralllel filling of $A_1$ to $[D^2, 1/2, -1/2]$. Let us now consider undetected $a_1, b_1, \ldots, a_l, b_l$ such that $p_1, q_1, \ldots, p_b, q_b$ are part of a solution to $\mathcal{L}'$. Suppose $A_1$ is adjacent to another Seifert fibred JSJ-component $A_2$ along a torus $T\in \mathcal{T}_N$. There are four $H_1$-basis we can consider on $T$, namely: 
\begin{itemize}
\item[-] The fibred slope $\lambda_1$ and section slope $\mu_1$ with respect to a fixed Seifert structure on $A_1$; 
\item[-] The fibred slope $\lambda_2$ and section slope $\mu_2$ with respect to the Seifert structure $[D^2, 1/2, -1/2]$ on the Dehn filling of $A_1$. Observe that $\lambda_2=\lambda_1$ and we may find an integral-affine function $\xi$ in variables $\rho_1,\sigma_1 \ldots, \rho_b, \sigma_b$ such that $\mu_2=\xi\cdot \lambda_1+\mu_1$; 
\item[-] The fibred slope $\lambda_3$ and the section slope $\mu_3$ with respect to a Seifert structure with Seifert data $[\text{M\"ob}]$ on the Dehn filling of $A_1$. By Lemma \ref{lemma:thickenedkleinbottlestructures}, we may take this Seifert structure with Seifert data $[\text{M\"ob}]$ such that $\lambda_3=-\mu_2, \mu_3=\lambda_2$; 
\item[-] The fibred slope $\lambda_4$ and section slope $\mu_4$ with respect to a fixed Seifert structure on $A_2$. We may find a fixed matrix $X\in GL_2(\mathbb Z)$ such that $(\mu_4, \lambda_4)^T=X\cdot (\mu_1, \lambda_1)^T$.
\end{itemize}
The fibred slope $\lambda_3$ lines up with the fibred slope $\lambda_4$ if and only if for some $q\in \mathbb Z, \varepsilon \in \{\pm 1\}$ we have: 
$$\begin{pmatrix}
 \mu_4\\ \lambda_4
\end{pmatrix}=
\begin{pmatrix}
 -\varepsilon & q \\
 0 & \varepsilon
\end{pmatrix}\cdot 
\begin{pmatrix}
 \mu_3 \\ \lambda_3
\end{pmatrix}=
\begin{pmatrix}
 -\varepsilon & q \\
 0 & \varepsilon
\end{pmatrix}
\begin{pmatrix}
 0 & 1 \\ -1 & 0 
\end{pmatrix}\cdot 
\begin{pmatrix}
 \mu_2 \\ \lambda_2
\end{pmatrix}=
\begin{pmatrix}
 -\varepsilon & q \\
 0 & \varepsilon
\end{pmatrix}
\begin{pmatrix}
 0 & 1 \\ -1 & 0 
\end{pmatrix}
\begin{pmatrix}
 1 & \xi \\ 0 & 1
\end{pmatrix}\cdot 
\begin{pmatrix}
 \mu_1 \\ \lambda_1
\end{pmatrix} 
$$
$$
\Leftrightarrow X=\begin{pmatrix}
 -\varepsilon & q \\
 0 & \varepsilon
\end{pmatrix}
\begin{pmatrix}
 0 & 1 \\ -1 & 0 
\end{pmatrix}
\begin{pmatrix}
 1 & \xi \\ 0 & 1
\end{pmatrix}.
$$
Without loss of generality consider the case $\varepsilon=1$. By Lemma \ref{lemma:sl2z} there is a linear system $\mathcal{L}''$ such that $p_1, q_1, \ldots, p_b, q_b, q$ is a part of a solution to the above if and only if the above matrix equation holds. Let $[\Sigma, r_1/s_1, \ldots, r_h/s_h]$ denote the chosen Seifert structure on $A_2$. Then by Example \ref{ex:gluingsfs} the gluing of $A_2$ to $K^2\tilde{\times} I$ such that $\lambda_3=\lambda_4$ and $\mu_4=-\mu_3+q\cdot \lambda_3$ is $\partial$-framed homeomorphic the Seifert fibred space with Seifert data $[\text{M\"ob}\#_\partial \Sigma, -q/1, r_1/s_1, \ldots, r_h/s_h].$ We therefore add to $\mathcal{N}_{red}$ the 6-tuple $(N', M, \mathcal{F}, (\rho'_j, \sigma'_j), \mathcal{L}\wedge \mathcal{L}'\wedge\mathcal{L}'', \mathcal{P}')$ in variables $\alpha_1, \beta_1, \ldots, \alpha_l, \beta_l, \gamma_1, \ldots, \gamma_m, q$ with: 
\begin{itemize}
 \item[-] $N'$ is $N$ with $A_1\cup A_2$ replaced by the Seifert fibred space with Seifert data $[\text{M\"ob}\#_\partial \Sigma-\{\text{pt}\}, r_1/s_1, \ldots, r_h/s_h]$; 
 \item[-] The $H_1$-basis of the $\partial$-framing on the boundary torus of $N'$ corresponding to the additional puncture $\{\text{pt}\}$ in $\text{M\"ob}\#_\partial \Sigma-\{\text{pt}\}$ given by the above Seifert structure; 
 \item[-] The remaining $\partial$-framing on $N'$ is inherited from $N$; 
 \item[-] The integral-affine functions $\rho'_j, \sigma'_j$ are obtained from $\rho_j, \sigma_j$ by omitting those integral-affine functions $\rho_1, \sigma_1, \ldots, \rho_b, \sigma_b$ corresponding to filling slopes in $A_1$ and adding integral-affine functions $\rho'_0=1$ and $\sigma'_0=-q$ corresponding to the filling slope in the boundary torus corresponding to $\{\text{pt}\}$; 
 \item[-] $\mathcal{P}'$ is obtained from $\mathcal{P}$ by merging all the parts in which $\rho_1, \sigma_1, \ldots, \rho_b, \sigma_b$ have their variables and adding to this merged part the variable $q$.
\end{itemize}
By the above discussion, any $a_1, b_1, \ldots, a_l, b_l$ for which $A_1$ Dehn fills to $K^2\tilde{\times} I$ with a fibred slope lining up with the fibred slope on $A_2$ is detected by $\mathcal{N}_{red}$. We add such 6-tuples to $\mathcal{N}_{red}$ for each JSJ-component of $N$ of the above form.\\

We also add the analogous 6-tuple to $\mathcal{N}_{red}$ for each JSJ-component $A_1$ of $N$ that is a circle bundle over a punctured M\"obius band. Hence any $a_1, b_1, \ldots, a_l, b_l$ giving a non-fibre-parallel Dehn filling in which $A_1$ becomes $K^2\tilde{\times} I$ with a fibred slope lining up with the fibred slope of an adjacent JSJ-component $A_2$ of $N$ is detected by $\mathcal{N}_{red}$. By an analogous argument, we may add 6-tuples to $\mathcal{N}_{red}$ detecting any $a_1, b_1, \ldots, a_l, b_l$ for which adjacent JSJ-components $A_1, A_2$ of $N$ both Dehn fill to $K^2\tilde{\times} I$ with fibred slopes lining up. Hence, the only undetected Dehn fillings of $N$ in which Seifert fibred JSJ-components of $N$ Dehn fill to Seifert structures with fibred slopes lining up are those with JSJ-components filling to become $T^2\times I$. Addressing this issue will be the final and most involved part of this proof. \\

Suppose there is a sequence $A_0, \ldots, A_{n+1}$ of JSJ-component of $N$ such that: 
\begin{itemize}
\item[-] $A_1, \ldots, A_n$ are distinct circle bundles over punctured annuli each with two JSJ-boundary tori and such that all other boundary tori are filling; 
\item[-] $A_i, A_{i+1}, i=0, \ldots, n$ are adjacent JSJ-components connected via the torus $T_i\in \mathcal{T}_N$; 
\item[-] $A_0, A_{n+1}$ may be empty, equal to each other, or we can have $A_0=A_n, A_{n+1}=A_1$.
\end{itemize}
Let us fix the following notation: 
\begin{itemize}
\item[-] Fix $\partial$-framings on $A_1, \ldots, A_n$ given by some fixed Seifert structure. Fix $\partial$-framings on $A_0, A_{n+1}$ which, if a Seifert structure exists, is given by a fixed choice of such a Seifert structure. With respect to the $H_1$-basis given by these $\partial$-framings on $A_i, A_{i+1}$, let the gluing map across $T_i$ be given by the matrix $Z_i\in GL_2(\mathbb Z)$. If $A_0$ or $A_{n+1}$ is empty, let $Z_0=\text{id}$ or $Z_n=\text{id}$, respectively;
\item[-] Let $A=A_1\cup\ldots\cup A_n$; 
\item[-] Let $B_0, \ldots, B_{n+1}$ denote the Dehn fillings of $A_0, \ldots, A_{n+1}$ and let $B=B_1\cup\ldots B_n$. 
\end{itemize}
We want to consider those undetected Dehn fillings of $N$ with $B_1\cong \ldots\cong B_n\cong T^2\times I$. Let $\mathcal{V}_A:=\{\rho_j, \sigma_j | \rho_j/\sigma_j \text{ is a filling slope in one in }A\}$. Suppose $a_1, b_1, \ldots, a_l, b_l, c_1, \ldots, c_m$ gives a non-fibre-parallel Dehn filling of $N$ to $(M, \mathcal{F})$ via $(\mathcal{L}, (\rho_j, \sigma_j), \mathcal{P})$. In this Dehn filling $A_1, \ldots, A_n$ become $T^2\times I$ if and only if for each filling slope $\rho_j \in \mathcal{V}_A$ we have $\rho_j=1$. Let $\mathcal{L}':=\{\rho_j=1 \mid \rho_j\in \mathcal{V}_A\}$ and assume $\alpha_1, \beta_1, \ldots, \alpha_l, \beta_l$ satisfy $\mathcal{L}'$. Moreover, by Example \ref{ex:thickenedtorus}, we know how the two $H_1$-basis on $T_0, T_n$ inherited from the $\partial$-framings on $A_0, A_{n+1}$ are related. Namely after the Dehn filling the isotopy through the product structure on $T^2\times I$ relates these $H_1$-basis of $T_0$ and $T_n$ by the matrix: $$Z:=Z_0\cdot \begin{pmatrix} 1 & \xi_1 \\ 0 & -1\end{pmatrix} \cdot Z_1\cdot \ldots \cdot Z_{n-1}\cdot \begin{pmatrix} 1 & \xi_n \\ 0 & -1\end{pmatrix}\cdot Z_n,$$ where $\xi_1, \ldots, \xi_n$ are integral-affine functions in variables $\mathcal{V}_A$. On the other hand, as discussed in Example \ref{ex:thickenedtorus}, the $H_1$-basis on the boundary tori of $T^2\times I$ given by a Seifert structure $[S^1\times I]$ are related by isotopy through the product structure via the matrix $\begin{pmatrix} 1 & 0 \\ 0 & -1 \end{pmatrix}$. Hence we may find a Seifert structure $[S^1\times I]$ on $B=T^2\times I$ such that the induced $H_1$-basis on $T_0=T^2\times \{0\}$ is equal to the $H_1$-basis inherited from $A_1$ and the induced $H_1$-basis on $T_n=T^2\times \{1\}$ is related to the $H_1$-basis inherited from $A_n$ via the matrix:
$$Z':=\begin{pmatrix} 1 & \xi_1 \\ 0 & -1\end{pmatrix} \cdot Z_1\cdot \ldots \cdot Z_{n-1}\cdot \begin{pmatrix} 1 & \xi_n \\ 0 & -1\end{pmatrix}\cdot \begin{pmatrix} 1 & 0 \\ 0 & -1\end{pmatrix}.$$ Let us now consider the following cases: \\

{\bf Case 1}. $A_0, A_{n+1}$ are empty. Then undetected $a_1, b_1, \ldots, a_l, b_l$ that are part of a solution to $\mathcal{L}'$ give a Dehn filling taking the component $A\subset N$ to a thickened torus. Check whether a component $D$ of $M$ is $T^2\times I$. If not, we immediately do not consider $a_1, b_1, \ldots, a_l, b_l$ that are part of a solution to $\mathcal{L}'$. If so, and $D$ is not involved in a $\partial$-equation of $\mathcal{F}$ then $a_1, b_1, \ldots, a_l, b_l$ is detected by $(N-A, M-D, \mathcal{F}, (\rho'_j, \sigma'_j), \mathcal{L}\wedge \mathcal{L}', \mathcal{P}')$ with:
\begin{itemize}
\item[-] $\rho'_j, \sigma'_j$ obtained from $\rho_j, \sigma_j$ by omitting the functions in $\mathcal{V}_A$, that is omitting the functions corresponding to filling slopes in $A$; 
\item[-] $\mathcal{P}'$ obtained from $\mathcal{P}$ by merging the parts in which functions of $\mathcal{V}_A$ have variables.
\end{itemize}
Hence suppose $D\cong T^2\times I$ and is involved in a $\partial$-equation of the form $\mathcal{M}_{T\times \{0\}}\cdot X=Y\cdot \mathcal{M}_{T\times \{1\}}\cdot X'$. Such $a_1, b_1, \ldots, a_l, b_l$ gives a Dehn filling satisfying this $\partial$-equation if and only if $X\cdot (Z'\cdot X')^{-1}$ is conjugate to $Y$ (for comparison see the case of $M\cong T^2\times I$ in the proof of Proposition \ref{prop:Rephrased_main_prop}). By Lemma \ref{lemma:sl2z} we may find a linear system $\mathcal{L}''$ such that $a_1, b_1, \ldots, a_l, b_l$ are part of a solution to $\mathcal{L}''$ if and only if these matrices are conjugate. Thus such $a_1, b_1, \ldots, a_l, b_l$ is detected by $(N-A, M-D, \mathcal{F}', (\rho'_j, \sigma'_j), \mathcal{L}\wedge \mathcal{L}'\wedge\mathcal{L}'', \mathcal{P}')$ with:
\begin{itemize}
\item[-] $\mathcal{F}'$ is $\mathcal{F}$ with the $\partial$-equation on $D$ omitted; 
\item[-] $\rho'_j, \sigma'_j$ obtained from $\rho_j, \sigma_j$ by omitting the functions in $\mathcal{V}_A$; 
\item[-] $\mathcal{P}'$ obtained from $\mathcal{P}$ by merging the parts in which functions of $\mathcal{V}_A$ have variables and the part in which the $\partial$-equation of $D$ has variables.
\end{itemize}
Add to $\mathcal{N}_{red}$ all the above 6-tuples for each thickened torus component $D$ of $M$ . \\

{\bf Case 2}. $A_0=A_n$ and thus $A_{n+1}=A_1$. Then undetected $a_1, b_1, \ldots, a_l, b_l$ that are part of a solution to $\mathcal{L}'$ give a Dehn filling taking the component $A\subset N$ to a torus bundle. This torus bundle has monodromy given by the matrix: 
$$Z'':=Z_0\cdot \begin{pmatrix} 1 & \xi_1 \\ 0 & -1\end{pmatrix} \cdot Z_1\cdot \ldots \cdot Z_{n-1}\cdot \begin{pmatrix} 1 & \xi_n \\ 0 & -1\end{pmatrix}$$
By Proposition \ref{Recognising_torus_bundle} we may check whether some component $D$ of $M$ is a torus bundle. If not, we are immediately done with this case. If so, find a matrix $X\in GL_2(\mathbb Z)$ such that $D$ has monodromy $X$. Then a Dehn filling of $A$ of the above form is homeomorphic to $D$ if and only if $Z''$ is conjugate to $X$ by a matrix in $SL_2(\mathbb Z)$. By Lemma \ref{lemma:sl2z}, there is a linear system $\mathcal{L}''$ in variables $\mathcal{V}_A$ such that $a_1, b_1, \ldots, a_l, b_l$ are part of a solution to $\mathcal{L}''$ if and only if $Z'', X$ are conjugate. Thus undetected $a_1, b_1, \ldots, a_l, b_l$ giving a Dehn filling of $A$ to $D$ are detected by $(N-A, M-D, \mathcal{F}, (\rho'_j, \sigma'_j), \mathcal{L}\wedge \mathcal{L}'\wedge\mathcal{L}'', \mathcal{P}')$ with:
\begin{itemize}
\item[-] $\rho'_j, \sigma'_j$ obtained from $\rho_j, \sigma_j$ by omitting the functions in $\mathcal{V}_A$; 
\item[-] $\mathcal{P}'$ obtained from $\mathcal{P}$ by merging the parts in which functions of $\mathcal{V}_A$ have variables. 
\end{itemize}
 Add to $\mathcal{N}_{red}$ such 6-tuples for all torus bundle components $D$ of $M$. \\

{\bf Case 3}. $A_{n+1}$ is empty, but $A_0$ is not: Then undetected $a_1, b_1, \ldots, a_l, b_l, c_1, \ldots, c_m$ that are part of a solution to $\mathcal{L}'$ give a Dehn filling under which $A$ becomes $T^2\times I$ and hence the Dehn filling $B$ of $N$ is homeomorphic to the corresponding Dehn filling $B_A$ of $N-A$. We may take such a homeomorphism $\psi\colon B\to B_A$ mapping $T_n$ to $T_0$ and respecting the $H_1$-basis inherited from $N$ on all boundary tori of $\partial B, \partial B_A$ apart from $T_0, T_n$. Then by our previous discussion $\mathcal{M}_{\psi, T_0}=Z^{-1}$. Thus such $a_1, b_1, \ldots, a_l, b_l$ is detected by the 6-tuple $(N-A, M, \mathcal{F}', (\rho'_j, \sigma'_j), \mathcal{L}\wedge \mathcal{L}', \mathcal{P}')$, with: \begin{itemize}
 \item[-] $\mathcal{F}'$ obtained by replacing each occurrence of $\mathcal{M}_{T}$ in $\mathcal{F}$ by $\mathcal{M}_{T}\cdot Z^{-1}$, where $T\subset \partial M$ is the torus corresponding to $T_n$; 
 \item[-] The $\rho'_j, \sigma'_j$ are obtained from $\rho_j, \sigma_j$ by omitting the functions in $\mathcal{V}_A$;
 \item[-] $\mathcal{P}'$ obtained from $\mathcal{P}$ by merging the parts in which functions of $\mathcal{V}_A$ have variables. 
\end{itemize}
Add such 6-tuples to $\mathcal{N}_{red}$. \\

{\bf Case 4}. $A_0, A_{n+1}$ are non-empty and $A_0$ is hyperbolic: Then, by the above discussion, undetected $a_1, b_1, \ldots, a_l, b_l$ that are part of a solution to $\mathcal{L}'$ give a Dehn filling under which $T_0$ will be a JSJ-torus of $\mathcal{T}_M$. Moreover such Dehn filling of $N$ will be homeomorphic to the corresponding Dehn filling of $N-A$ with the two tori $T'_0, T''_0$ corresponding to $T_0$ glued by the matrix $Z$. Hence, undetected $a_1, b_1, \ldots, a_l, b_l$ that are part of a solution to $\mathcal{L}'$ are detected by some $(N-A, M-T, \mathcal{F}', (\rho'_j, \sigma'_j), \mathcal{L}\wedge \mathcal{L}', \mathcal{P})$, with: 
\begin{itemize}
 \item[-] $T\subset \mathcal{T}_M$ a JSJ-torus on $M$. Let $T', T''$ denote the two boundary tori of $M-T$ corresponding to $T$. Fix $H_1$-basis on $T', T''$ that are equal after regluing; 
 \item[-] Fix $H_1$-basis on $T_0, T_n\subset \partial(N-A)$ inherited from $A_0, A_{n+1}$, respectively; 
 \item[-] $\mathcal{F}'$ obtained by adding to $\mathcal{F}$ the $\partial$-equation $\mathcal{M}_{T'}=\mathcal{M}_{T''}\cdot Z$;
 \item[-] As usual $\rho'_j, \sigma'_j$ are equal to $\rho_j, \sigma_j$ with those functions corresponding to filling slopes in $A$ omitted; 
 \item[-] As usual $\mathcal{P}'$ obtained from $\mathcal{P}$ by merging the parts in which functions of $\mathcal{V}_A$ have variables.
\end{itemize}
Add to $\mathcal{N}_{red}$ these 6-tuples for each each torus $T\subset \mathcal{T}_M$. \\

{\bf Case 5}. Suppose $A_0, A_{n+1}$ are Seifert fibred: Fix Seifert structures on $A_0, A_{n+1}$. Undetected $a_1, b_1, \ldots, a_l, b_l$ that are part of a solution to $\mathcal{L}'$ give a Dehn filling of $N$ that is homeomorphic to the corresponding Dehn filling of $N-A$ with $T_0, T_n$ glued by the map given by matrix $Z$ with respect to $H_1$-basis given by the Seifert structures on $A_0, A_{n+1}.$ We consider several subcases. \\

{\bf Case 5.1}. Suppose $Z$ maps the unoriented fibred slope of $A_0$ to the unoriented fibred slope of $A_{n+1}$. This is equivalent to $Z=\begin{pmatrix} -\varepsilon & q\\ 0 & \varepsilon \end{pmatrix} $ for some $q\in \mathbb Z, \varepsilon \in \{\pm 1\}$. Without loss of generality consider the $\varepsilon=1$ case. Let the Seifert structures on $A_0, A_{n+1}$ have Seifert data $[\Sigma, r_1/s_1, \ldots, r_h/s_h]$ and $[\Sigma', r'_1/s'_1, \ldots, r'_{h'}/s'_{h'}]$ respectively. Then by Example \ref{ex:gluingsfs} a gluing of $A_0, A_{n+1}$ along $Z$ is homeomorphic to $[\Sigma \#_\partial \Sigma', -q/1, r_1/s_1, \ldots, r_h/s_h, r'_1/s'_1,$ $\ldots, r'_{h'}/s'_{h'}].$ Hence undetected $a_1, b_1, \ldots, a_l, b_l$ that give Dehn fillings in which $A$ becomes a thickened torus through which the fibred slopes of $A_0, A_{n+1}$ line up are detected by the 6-tuple: $(N', M, \mathcal{F}, (\rho'_j, \sigma'_j), \mathcal{L}'', \mathcal{P}')$ in variables $\alpha_1, \beta_1, \ldots, \alpha_l, \beta_l, \gamma_1, \ldots, \gamma_m, q$ with: 
\begin{itemize}
 \item[-] $N'$ equal to $N$ with $A_0\cup A\cup A_{n+1}$ replaced by a space with Seifert data $[\Sigma \#_\partial \Sigma'-\{\text{pt}\}, r_1/s_1, \ldots, r_h/s_h,$ $r'_1/s'_1, \ldots, r'_{h'}/s'_{h'}]$; 
 \item[-] The $\partial$-framing of $N'$ inherited from $N$ apart from in the boundary torus corresponding to $\{\text{pt}\}$ which has $H_1$-basis given by the above Seifert structure; 
 \item[-] $\rho'_j, \sigma'_j$ are integral-affine functions in variables $\alpha_1, \ldots, \alpha_l$ obtained from $\rho_j, \sigma_j$ by omitting the functions in $\mathcal{V}_A$ and adding the integral-affine functions $\rho'_0=1, \sigma'_0=-q$ corresponding to the filling slope in the boundary torus given by $\{\text{pt}\}$; 
 \item[-] The linear system $\mathcal{L}''$ equal to $\mathcal{L}\wedge \mathcal{L}'\wedge \mathcal{L}_q$, where $\mathcal{L}_q$ is the linear system in variables $\mathcal{V}_A$ given by Lemma \ref{lemma:sl2z} such that $a_1, b_1, \ldots, a_l, b_l, q$ are part of a solution to $\mathcal{L}_q$ if and only if $\begin{pmatrix} -1 & q\\ 0 & 1\end{pmatrix} =Z$; 
 \item[-] The partition $\mathcal{P}'$ obtained from $\mathcal{P}$ by merging all the parts in which functions of $\mathcal{V}_A$ have variables and adding to this newly merged part the variable $q$. 
\end{itemize}
Add to $\mathcal{N}_{red}$ all such 6-tuples. \\

{\bf Case 5.2}. Suppose undetected $a_1, b_1, \ldots, a_l, b_l, c_1, \ldots, c_m$ gives a Dehn filling with $B_0\cong K^2\tilde{\times} I$ and $Z$ maps the fibred slope of this $K^2\tilde{\times} I$ that isn't inherited from $A_0$ to the fibred slope of $A_{n+1}$. This is the identical setup as the discussion of a JSJ-component $A_1$ of $N$ Dehn filling to $K^2\tilde{\times} I$ we discussed earlier in this proof with the only change being the fixed matrix $X$ in the proof above is replaced by the matrix expression $Z$. Moreover the proof for $X$ replaced by $Z$ carries through verbatim and we may add 6-tuples to $\mathcal{N}_{red}$ detecting such $a_1, b_1, \ldots, a_l, b_l$. Similarly we may add 6-tuples to $\mathcal{N}_{red}$ detecting Dehn fillings in which $B_0\cong B_{n+1}\cong K^2\tilde{\times} I$ and $A$ Dehn fills to $T^2\times I$ mapping a fibred slope on one of these $K^2\tilde{\times} I$ to a fibred slope in the other. \\

Consider an undetected Dehn filling in which $A$ becomes $T^2\times I$ and $B_0, B_{n+1}$ have unoriented fibred slopes isotopic through the product structure on $T^2\times I$. By Case 5.1 the lined-up fibred slopes on $B_0, B_{n+1}$ cannot be images of the fibred slope on $A_0, A_{n+1}$. So one of $B_0, B_{n+1}$ must have multiple non-isotopic Seifert structure, that is must be one of $T^2\times I$ or $K^2\tilde{\times} I$. By Case 5.2 and without loss of generality, we must have $B_0\cong T^2\times I$. Let $A_{-1}$ be the, possibly empty, JSJ-component of $N$ on the other of side $A_0$. Then $a_1, b_1, \ldots, a_l, b_l$ is in Case 5 with $A_0, \ldots, A_{n+1}$ replaced by $A_{-1}, \ldots, A_{n+1}$. Add to $\mathcal{N}_{red}$ all 6-tuples described in Case 5 with $A_0, \ldots, A_{n+1}$ replaced by $A_{-1}, \ldots, A_{n+1}$. More generally, add to $\mathcal{N}_{red}$ the 6-tuples so far described in Case 5 for all possible chains of JSJ-components $A_0, \ldots, A_{n+1}$. Hence we detect $a_1, b_1, \ldots, a_l, b_l$ giving a Dehn filling in which a chain of JSJ-components Dehn fills to $T^2\times I$ with unoriented fibred slopes of Seifert fibred JSJ-components on either side isotopic through the product structure on $T^2\times I$. Hence, to summarise, in an undetected Dehn filling of $N$ after collapsing all JSJ-component that become $T^2\times I$ we obtain a JSJ-decomposition of $M$. In particular in an undetected Dehn filling of $N$ every torus $T\subset \mathcal{T}_N$ will become isotopic to a JSJ-torus of $\mathcal{T}_M$. In light of this let us conclude the final subcase. \\

{\bf Case 5.3.} Suppose undetected $a_1, b_1, \ldots, a_l, b_l, c_1, \ldots, c_m$ gives a Dehn filling in which $A$ becomes $T^2\times I$: Then the Dehn filling of $N$ is homeomorphic to the corresponding Dehn filling of $N-A$ with $T_0$ glued to $T_{n}$ along a map given by the matrix $Z$. By the above discussion, the Dehn filling of $A$ collapses down to becomes a JSJ-torus $T\in \mathcal{T}_M$ after Dehn filling. Let $T', T''$ denote the two boundary tori of $M-T$ corresponding to $T$. Hence such $a_1, b_1, \ldots, a_l, b_l$ is detected by a 6-tuple of the form $(N-A, M-T, \mathcal{F}', (\rho'_j, \sigma'_j), \mathcal{L}', \mathcal{P}')$ with: 
\begin{itemize}
 \item[-] The $H_1$-basis on $T_0, T_n\subset \partial(N-A)$ inherited from the Seifert structure on $A_0, A_{n+1}$ and the remaining $\partial$-framing on $N-A$ inherited from $N$. 
 \item[-] $\partial$-framing on $M-T$ inherited from $M$ and $H_1$-basis on $T', T''$ chosen to be equal after regluing; 
 \item[-] $\mathcal{F}'$ obtain from $\mathcal{F}$ by adding the $\partial$-equation $\mathcal{M}_{T'}=\mathcal{M}_{T''}\cdot Z$; 
 \item[-] $\rho'_j, \sigma'_j$ equal to $\rho_j, \sigma_j$ with $\mathcal{V}_A$, that is those functions corresponding to filling slopes in $A$ omitted; 
 \item[-] $\mathcal{P}'$ obtained from $\mathcal{P}$ by merging the parts in which functions of $\mathcal{V}_A$ have their variables into one part.
\end{itemize}
Add such 6-tuples for all JSJ-tori $T\in \mathcal{T}_M$ and for all possible choises of $A_0, \ldots, A_{n+1}$.\\

To summarise, any undetected Dehn filling of $N$ will: 
\begin{itemize}
\item[-] Not fill hyperbolic JSJ-components of $N$ along exceptional slopes; 
\item[-] Not fill Seifert fibred JSJ-components of $N$ along fibred slopes; 
\item[-] Not fill a Seifert fibred JSJ-component of $N$ to become a solid torus; 
\item[-] Not fill a Seifert fibred JSJ-component of $N$ to become $K^2\tilde{\times} I$ with either of its two fibred slopes lining up with the fibred slopes of the Dehn filling of another Seifert fibred component; 
\item[-] Not fill a Seifert fibred JSJ-component of $N$ to become a thickened torus. 
\end{itemize}
Hence any undetected Dehn filling is not JSJ-exceptional and we have completed the proof. 
\end{proof}

\subsection{Proof of Corollary \ref{thm:hyperbolic_Dehn_Ancestry}}

We now recall and prove Corollary \ref{thm:hyperbolic_Dehn_Ancestry}:\\

\noindent
{\bf Corollary \ref{thm:hyperbolic_Dehn_Ancestry}.}
Consider $\partial$-framed manifolds $N, M$ with $N$ hyperbolic and $k\in \mathbb N$. There is a mono-quadratic system $\mathcal{Q}_{\text{hyp}}$ such that $\gcd(p_1, q_1)=\ldots=\gcd(p_{k}, q_{k})=1$, the Dehn filling of $N$ to $N(p_1/q_1, \ldots, p_k/q_k)$ is strongly non-exceptional, and $N(p_1/q_1, \ldots, p_k/q_k)$ admits a Dehn filling to $M$ if and only if $p_1, q_1, \ldots, p_{k}, q_{k}$ are part of a solution to $\mathcal{Q}_{\text{hyp}}$. Moreover, there is an algorithm that, given $M,N$, and $k$, outputs $\mathcal{Q}_{\text{hyp}}$.

\begin{proof} Recall that by Theorem \ref{prop:mainprop} we may find a mono-quadratic system $\mathcal{Q}$ such that $p_1, q_1, \ldots, p_k, q_k$ is part of a solution to $\mathcal{Q}$ if and only if $\gcd(p_1, q_1)=\ldots=\gcd(p_k, q_k)=1$ and $N(p_1/q_1, \ldots, p_k/q_k)$ admits a Dehn filling to $M$. Hence we need only address the issue of whether the Dehn filling of $N$ to $N(p_1/q_1, \ldots, p_k/q_k)$ is strongly non-exceptional. Before we begin, let us recall that a multi-slope on $N$ is a set of slopes in distinct boundary tori of $N$. Moreover observe that for a fixed multi-slope $S$ it is decidable whether the Dehn filling of $N$ to $N(S)$ is strongly non-exceptional. \\

By Theorem \ref{theorem:longFillingGeometry} and Proposition \ref{prop:find_short_slopes} we know that for a hyperbolic 3-manifold $A$ we may find a finite set of slopes $\mathcal{E}(A)$ in $\partial A$ such that any Dehn filling of $A$ avoiding slopes in $\mathcal{E}(A)$ is hyperbolic. For a set of slopes $S$ on $N$ let $\mathcal{E}(S)$ denote $S\cup\left(\bigcup_{S'\subset S}\mathcal{E}(N(S'))\right),$ where we index over all multi-slopes $S'\subset S$. Observe first that if $S$ is finite, then so is $\mathcal{E}(S)$. Moreover, observe that, by construction, for any multi-slope $S=S_1\cup S_2$ on $N$ with $S_2\cap \mathcal{E}(S_1)=\emptyset$ the Dehn filling of $N$ to $N(S)$ is strongly non-exceptional if and only if the Dehn filling of $N$ to $N(S_1)$ is strongly non-exceptional. \\

{\bf Claim:} Given $N$ we may find a finite set of slopes $\mathcal{E}_{\text{max}}$ on $\partial N$ such that any multi-slope $S$ on $N$ is of the form $S=S_1\cup S_2$ for $S_1\subset \mathcal{E}_{\text{max}}, S_2\cap \mathcal{E}(S_1)=\emptyset$. \\
Proof of Claim: Let us show that we may take $\mathcal{E}_{\text{max}}=\mathcal{E}(\mathcal{E}(\ldots \mathcal{E}(\mathcal{E}(N))\ldots ))$, where we apply $\mathcal{E}$ $b+1$ times for $b$ the number of components of $\partial N$. Take some arbitrary multi-slope $S$ on $N$ and let us modify the partition $S_1=\emptyset, S_2=S$ of $S$ until it satisfies the claim. If $S\cap \mathcal{E}(N)=\emptyset$ we are done and if not let $S_1=S\cap \mathcal{E}(N)$ and $S_2=S\setminus S_1$. Again, if $S_2\cap \mathcal{E}(\mathcal{E}(N))=\emptyset$ we are done and if not let $S_1=S\cap \mathcal{E}(\mathcal{E}(N))$ and $S_2=S\setminus S_1$. Repeat this until we are done or $S_1=S$ and $S_2=\emptyset$. Since $|S|=b$ and the size of $S_2$ decreases at each step, this process will terminate in at most $b$ steps. This proves the claim. \\

Hence for a multi-slope $S$ on $N$ the Dehn filling of $N$ to $N(S)$ is strongly non-exceptional if for some $S_1\subset \mathcal{E}_{\text{max}}$ we have $S_1 \subset S, (S\setminus S_1) \cap \mathcal{E}(S_1)=\emptyset$, and the Dehn filling of $N$ to $N(S_1)$ is strongly non-exceptional. Since $\mathcal{E}_{\text{max}}$ is finite we may in turn express this as a linear-system $\mathcal{L}$. To be precise, there is a linear system $\mathcal{L}$ such that $p_1, q_1, \ldots, p_k, q_k$ with $\gcd(p_1, q_1)=\ldots=\gcd(p_k, q_k)=1$ are part of a solution of $\mathcal{L}$ if and only if the Dehn filling of $N$ to $N(p_1/q_1, \ldots, p_k/q_k)$ is strongly non-exceptional.\\

Recall the mono-quadratic system $\mathcal{Q}$ given by Theorem \ref{prop:mainprop} as discussed above. To summarise the above: if $\mathcal{Q}_{\text{hyp}}:=\mathcal{L}\wedge \mathcal{Q}$ is a mono-quadratic system, then it fulfils the conditions of the Corollary. Observe, that each equation of $\mathcal{L}$ is of the form $(p_i, q_i)=(a,b)$ or $(p_i, q_i)\neq (a,b)$ for some $(a,b)\in \mathbb Z^2$. In particular, no equation of $\mathcal{L}$ involves both one of $p_i$ or $q_i$ and one of $p_j$ or $q_j$ for $i\neq j$. Consider again the mono-quadratic system $\mathcal{Q}$ given by Theorem \ref{prop:mainprop} as discussed above. Express $\mathcal{Q}$ as $\bigvee_i\mathcal{Q}_i$ such that for each $\mathcal{Q}_i$ there is a partition $\mathcal{P}_i$ of its variables with no equation using variables from multiple parts and at most one quadratic equation using variables from each part. From the proof of Theorem \ref{prop:mainprop}, we may choose each partition $\mathcal{P}_i$ such that each of the pairs $\{p_1, q_1\}, \ldots, \{p_k, q_k\}$ lies in the same part. Hence $\mathcal{L}\wedge \mathcal{Q}_i$ are mono-quadratic systems and therefore $\mathcal{Q}_{\text{hyp}}:=\mathcal{L}\wedge \mathcal{Q}=\bigvee_i(\mathcal{L}\wedge\mathcal{Q}_i)$ is a mono-quadratic system. This is our desired $\mathcal{Q}_{\text{hyp}}$.
\end{proof}

\bibliographystyle{alpha}
\bibliography{citation}

\end{document}